\documentclass[10pt,a4paper]{article}
\usepackage[utf8]{inputenc}
\usepackage{amsmath}
\usepackage{pifont}
\usepackage{bm}
\usepackage[american]{babel}
\usepackage[all]{xy}
\usepackage{enumitem}
\usepackage{anysize}
\usepackage[dvipsnames]{xcolor}
\usepackage{graphicx}
\usepackage{stmaryrd}
\usepackage{amsfonts,amstext,amssymb,amsmath,amsthm,pspicture,bigints}
\usepackage{dsfont}

\usepackage{fullpage}
\usepackage{moreverb}
\usepackage{pifont}
\usepackage{pst-all}
\usepackage[left=2cm,right=2cm,top=2cm,bottom=2cm]{geometry}
\usepackage{esint}
\usepackage{fourier-orns}
\usepackage{mathrsfs}
\usepackage{array}

\usepackage{hyperref}

\usepackage{authblk}
\newcolumntype{C}[1]{>{\centering\arraybackslash}b{#1}}
\newcolumntype{R}[1]{>{\raggedleft\arraybackslash}b{#1}}
\newcolumntype{L}[1]{>{\raggedright\arraybackslash}b{#1}}
\newcolumntype{M}[1]{>{\centering}m{#1}}
\newtheorem{theo}{Theorem}[section]
\newtheorem{defin}{Definition}[section]
\newtheorem{lem}{Lemma}[section]
\newtheorem{prop}{Proposition}[section]

\newtheorem{remark}{Remark}[section]
\newtheorem{hypothesis}{Hypothesis}[section]

\usepackage{bbm}
\numberwithin{equation}{section}
\usepackage{tikz}

\newcommand{\Ind}{\mathds{1}}

\newcommand{\tend}[2]{\underset{#1\to #2}{\longrightarrow} }

\newcommand{\ntriple}[1]{{\left\vert\kern-0.25ex\left\vert\kern-0.25ex\left\vert #1 
    \right\vert\kern-0.25ex\right\vert\kern-0.25ex\right\vert}}

\newcommand{\dist}{\mathrm{dist \, }}

\newcommand{\der}[2]{\frac{\dd #1}{\dd #2}}

\newcommand{\dd}{\mathrm{d}}
\newcommand{\Diff}{\mathrm{D}}

\newcommand{\supp}{\mathrm{supp}\hspace{1mm}}

\newcommand{\eps}{\varepsilon}

\newcommand{\R}{\mathbb{R}}

\renewcommand{\tilde}{\widetilde}

\title{Confinement results near point vortices on the rotating sphere}
\author{Martin Donati \qquad Emeric Roulley}
\date{}

\begin{document}

\maketitle
\begin{abstract}
    We study the Euler equation on the rotating sphere in the case where the absolute vorticity is initially sharply concentrated around several points. We follow the literature already concerning vorticity confinement for the planar Euler equations, and obtain similar results on the rotating sphere, with new challenges due to the geometry. More precisely, we show the improbability of collisions for point-vortices, logarithmic in time absolute vorticity confinement for general configurations, the optimality of this last result in general, and the existence of configurations with power-law long confinement. We take this opportunity to write a unified, self-contained, and improved version of all the proofs, previously scattered across multiple papers on the planar case, with detailed exposition for pedagogical clarity.
\end{abstract}
\tableofcontents

\section{Introduction}

To study large-scale atmospheric or oceanic dynamics, one has to take into account the natural presence of vortices of very high circulations. Going back to the work of Helmholtz \cite{Helmholtz1}, a first approximation of the motion of such vortices is given by the so called \emph{point-vortex} dynamics, which corresponds to assuming that the vortices are singular: Dirac masses of vorticity. Numerous works studied this dynamics in many different context, and we refer the interested reader to \cite{Aref_2007} for a review on point-vortex dynamics in the plane. An other important question is the measure how good of an approximation the point-vortex system is to a solution of the Euler equations consisting of sharply concentrated vortices. This question can be answered in many different ways and in this paper we will mainly focus on the problem of localization and confinement: given an initial data sharply concentrated around several points, how long does the solution of the Euler equation remains both concentrated and following the prediction of the point-vortex model ?

This question, first answered in \cite{MP83}, lead to an important series of work, and in the scope of the present article, we put the emphasis on \cite{BM18}, which is the starting point of the present paper. To the best of our knowledge, all known results on vorticity confinement are established in various fluids models, mostly the planar Euler equations, but none on the rotating sphere.

The aim of this paper is to establish equivalent results to \cite{BM18,D24a,Donati_2025_Crystal} about vorticity confinement but for the Euler equations on the rotating sphere. In addition, to complete this work, we prove a necessary property of the point-vortex dynamics on the sphere, which is the improbability of finite-time collisions.

\subsection{Point-vortices and vorticity confinement on the plane}
    At first, let us contextualize our present work on the rotating sphere by recalling the planar point-vortex system for Euler equations and the associated results. We consider the bidimensional homogeneous and incompressible Euler equations in the plane
    \begin{equation}\label{2DEuler}
        \partial_t\omega+u\cdot\nabla\omega=0,\qquad u=\nabla^{\perp}\psi,\qquad\Delta\psi=\omega.
    \end{equation}
    Solving the Laplace equation gives the following integral representation for the stream function $\psi$
    $$\psi(t,z)=\int_{\mathbb{R}^2}G_{\mathbb{R}^2}(z,\xi)\omega(t,\xi)d\xi,\qquad G_{\mathbb{R}^2}(z,\xi)\triangleq\frac{1}{2\pi}\ln\big(|z-\xi|\big).$$
    The point-vortex dynamics describes the evolution of formal solutions to the 2D Euler equations where the vorticity is concentrated on points, namely
    \begin{equation}\label{vortPV2D}
        \omega(t)=\sum_{i=1}^{N}\Gamma_i\delta_{z_i(t)},\qquad N\in\mathbb{N}^*,\qquad\Gamma_i\in\mathbb{R}^*,\qquad z_i(t)\in\mathbb{R}^2.
    \end{equation}
    Plugging \eqref{vortPV2D} inside \eqref{2DEuler}, we get formally the following set of ordinary differential equations called \textit{point-vortex system}
    $$\dot{z_i}(t)=\sum_{j=1\atop j\neq i}^{N}\frac{\Gamma_j}{2\pi}\frac{\big(z_i(t)-z_j(t)\big)^\perp}{|z_i(t)-z_j(t)|^2}=\frac{1}{\Gamma_i}\nabla_{z_i}^{\perp}H\big(z_1(t),\ldots,z_N(t)\big),\qquad H(z_1,\ldots,z_N)\triangleq\sum_{1\leqslant i,j\leqslant N\atop i\neq j}\frac{\Gamma_i\Gamma_j}{4\pi}\ln\big(|z_i-z_j|\big).$$

    Due to the singularity in the velocity when $z_i(t) = z_j(t)$, all initial data our not admissible for this dynamics, and even admissible data can lead to non-trivial blow-up of solution, which we call vortex \emph{collapses}. It happens when there exists a time $T^* < \infty$, and two indices $i\neq j\in\{1,\ldots,N\}$ such that $$\liminf_{t \to T^*} |z_i(t)-z_j(t)| = 0.$$
    However, the collisions are known to be \emph{improbable}, in the sense the the measure of the set of initial data leading to a collapse is 0. This result was proved in the torus in \cite{DurrPulvirenti_1982}, in bounded domains in \cite{Donati_2021}, and in unbounded domains such as the whole plane in \cite{Marchioro_Pulvirenti_1984} and for other models of point-vortices in \cite{Geldhauser_Romito_2020,GC22}, under an additional condition on the vortex intensities $(\Gamma_i)_{1 \leqslant i \leqslant N}$. In the case of the sphere, we prove the same result, without the additional assumption due to the fact that the sphere is a compact manifold.

    We now introduce the problems of localization and confinement of vorticity. Given a set of pairwise distinct points $z_1^0,\ldots,z_N^0$, a set of intensities $\Gamma_1,\ldots,\Gamma_N$ such that the associated solution to the point-vortex dynamics has a global solution $\big(t \mapsto z_i(t)\big)_{1\leqslant i \leqslant N}$, let us consider, for every $\eps > 0$, an initial datum $\omega_0^\eps$ for the Euler equations satisfying the following assumptions.
    \begin{hypothesis}\label{hyp:omega0}
        There exist constants $M,\varepsilon_0>0$ and $\eta \geqslant 2$ such that for every $\eps\in(0,\varepsilon_0)$,
        \begin{itemize}
        \item $\omega_0^\eps\in L^1 \cap L^\infty(\R^2)$, with $\|\omega_0^\eps\|_{L^\infty(\R^2)} \leqslant M \eps^{-\eta}$,
        \item $\displaystyle \omega_0^\eps = \sum_{i=1}^N \omega_{0,i}^\eps \quad $ where $\displaystyle \quad \frac{\omega_{0,i}^\eps}{\Gamma_i} \geqslant 0,\quad$ $\Gamma_i\triangleq\displaystyle\int_{\R^2}\omega_{0,i}^\eps(x)\dd x\quad$, 
        \item $\supp \omega_{0,i}^{\eps} \subset B(z_i^0,\eps)$.
    \end{itemize}
    \end{hypothesis}
    In view of these hypotheses, one has that $\omega_{i,0}^\eps \tend{\eps}{0} \delta_{z_i^0}$ in the sense of measures. Since $\omega$ solves the (nonlinear) transport equation~\eqref{2DEuler} by a velocity field $u$ that is divergence free, then for all times, we have that the solution $\omega^\eps$ of the Euler equations such that $\omega^\eps(0) = \omega_0^\eps$ satisfies for all times $ t\geqslant 0$ that
    \begin{itemize}
        \item $\omega^\eps(t,\cdot)\in L^1 \cap L^\infty(\R^2)$, with $\|\omega^\eps(t,\cdot)\|_{L^\infty(\R^2)} \leqslant M \eps^{-\eta}$,
        \item $\displaystyle \omega^\eps(t,\cdot) = \sum_{i=1}^N \omega_{i}^\eps(t,\cdot) \quad $ where $\displaystyle \quad \frac{\omega_{i}^\eps(t,\cdot)}{\Gamma_i} \geqslant 0,\quad$ $\displaystyle\int_{\R^2}\omega_{i}^\eps(t,x)\dd x = \Gamma_i$.
    \end{itemize}
    Moreover, the maps $\omega_i^\eps(t)$ are all compactly supported, at all times, however the information that we loose \emph{a priori} is both the localization and growth of the support. The naive bound on the growth of support would be linear in $t$, and for a single blob of positive vorticies, this bound can be improved to $(t \ln t)^{1/4}$ (\cite{ISG99}). When considering several sharply concentrated vortices, by assuming for instance Hypothesis~\ref{hyp:omega0}, the vortices remain concentrated for a time of order $\mathcal{O}(|\ln \eps|)$ around the point-vortex system. More precisely, let us consider for any $\beta <1/2$ and $\eps > 0$:
     \begin{equation*}
    \tau_{\eps,\beta}\triangleq\sup \left\{ t \geqslant 0 \quad\textnormal{s.t.}\quad\forall s \in [0,t],\quad \supp\big(\omega^\eps(s,\cdot)\big) \subset \bigcup_{i=1}^N B\big(z_i(s),\eps^\beta\big) \right\}.
    \end{equation*}
    This time $\tau_{\eps,\beta}$ is the first time the support of the vorticity exits the reunion of balls of radius $\eps^\beta$ around the point-vortex solution. Since $\eps^\beta \ll 1$ for $\eps$ small enough, $\tau_{\eps,\beta}$ is a time during which both the solution remains a sum of sharply concentrated vortices, and that those vortices are located near the solution of the point-vortex dynamics. Regarding this time $\tau_{\eps,\beta}$,
    Buttà and Marchioro proved the following.
    \begin{theo}[\cite{BM18}]\label{theo:BM18general}
        Let $(z_i^0)_{1\leqslant i\leqslant N}$ be $N$ pairwise distinct points of $\R^2$ and $(\Gamma_i)_{1\leqslant i \leqslant N}$ be some non-vanishing intensities such that the solution of the point-vortex dynamics~\eqref{PVS} with initial datum $(z_i^0)_{1\leqslant i\leqslant N}$ is global in time and satisfies the following distance condition for some $d_0>0,$
        \begin{equation*}
            \forall\,1\leqslant i,j\leqslant N, \quad i\neq j\quad\Rightarrow\quad\inf_{t \geqslant 0}|z_i(t) - z_j(t)| \geqslant d_0.
        \end{equation*}
        Let $\omega_0^\eps$ satisfying Hypothesis~\ref{hyp:omega0}. Then for every $\beta<1/2$ there exists $\eps_0\triangleq\varepsilon_0(\beta,d_0)>0$ and $\alpha\triangleq\alpha(\beta,d_0)>0$ such that for every $\eps \in (0,\eps_0)$ the solution $\omega^\eps$ of \eqref{2DEuler} with initial condition $\omega_0^\eps$ satisfies
        \begin{equation*}
            \tau_{\eps,\beta} \geqslant \alpha|\ln \eps|.
        \end{equation*}
    \end{theo}
    In conclusion, Theorem~\ref{theo:BM18general} proves that the point-vortex dynamics is a good approximation of the Euler equations for a time at least of order $\mathcal{O}(|\ln \eps|)$, where $\eps$, which describes how concentrated the initial vortices are. One can then wonder whether this logarithmic bound is optimal. In general, it is the case, as proved in \cite{D24a} using unstable configurations of point-vortices. However, there exist particular conditions under which this bound can be improved. In \cite{BM18}, in addition to Theorem~\ref{theo:BM18general}, the authors prove that three situations lead to an estimate of the form $\tau_{\eps,\beta} \geqslant \eps^{-\alpha}$, with $\alpha > 0$. The first case is taking $N=1$, a single vortex, in which case this bound becomes a rescaled version of the result of \cite{ISG99}. Then, taking a self-similar expanding configuration of vortices (with an explicit example for $N=3$), where the growth of the distances between the point-vortices is enough to enhance the bound on $\tau_{\eps,\beta}$. Last, they prove that power-law bound holds for a single vortex placed at the center of a circular rigid boundary. In \cite{DI21}, it is then proved that other bounded domains can be constructed satisfying the existence of a point around which a concentrated vortex satisfies the enhanced confinement bound. In \cite{Donati_2025_Crystal}, it is this time with special configurations for arbitrary $N \geqslant 4$ that this bound was obtained. With new techniques, \cite{Meyer} proved that for general configurations, but assuming in addition the initial vortices $\omega_{i,0}^\eps$ to be nearly radial functions, then the power-law bound holds, with larger $\alpha$ as the blobs are closer to be radial.

    This problem can be formulated for more complicated fluid equations, such as the SQG equations (see \cite{Cavallaro_Garra_Marchioro_2021}), the lake equations \cite{Hientzsch_Lacave_Miot_2022_Dynamics_of_PV_for_the_lake_eq}, and the three-dimensional Euler equations with axial (\cite{Butta_Cavallaro_Marchioro_2022_Global_Time_Evolution,Mar24,DHLM25:leapfrogging}) or helical (\cite{DLM24},\cite{Guo_Zhao}) symmetry, with a range of new difficulties arising due to changes in the Biot-Savart law. The present paper aims to have the same discussion in the case of the Euler equations on the rotating sphere.

    \medskip

We end this paragraph by mentioning some constructions of periodic (or more general) solutions near special point vortex configurations. In the plane, a single point vortex stays immobile, placed at the origine, and due to the symmetries of Euler equations, it is possible to find via bifurcation techniques periodic solutions performing a uniform rotation around it \cite{B82,GHM23,HHHM15,HHMV16,K74}. The search for quasi-periodic solutions is more delicate. In this regime, one encounters small divisors and time–space resonances, which prevent a direct perturbative construction and instead require techniques from infinite-dimensional Hamiltonian dynamics, notably KAM theory and Nash–Moser schemes. In particular, quasi-periodic motions near the Rankine vortex were first established in 2021 in the second author’s PhD thesis for the quasi-geostrophic shallow-water equations \cite{HR21}, and in the same period for generalized SQG models by Hassainia-Hmidi-Masmoudi \cite{HHM21}. Related results were obtained for the Euler equations near Kirchhoff ellipses by Berti-Hassainia-Masmoudi \cite{BHM22}, near Rankine vortices in the unit disc by Hassainia-Roulley \cite{HR22} and near annuluar patches by Hassaini-Hmidi-Roulley \cite{HHR23}. Also with weak Birkhoff normal forms, the authors in \cite{GSIP23} could find quasi-periodic vortex patches for the very singular generalized SQG equations.

The contour dynamics desingularization approach has proved to be a powerful tool for constructing families of periodic vortex patch solutions exhibiting either uniform rotation or uniform translation and following more complicated point vortex dynamics. The first results in this direction were obtained by Hmidi and Mateu \cite{HM17}, who established the existence of symmetric pairs of patches (with equal or opposite strengths). The asymmetric case was later treated by Hmidi and Hassainia \cite{HH21}. These local bifurcation results were subsequently complemented by global bifurcation analyses by Garc\'ia-Haziot \cite{GH22}, which revealed the global structure of the solution branches. Extensions to configurations with more vortices were developed by Garc\'ia, both near classical von K\'arm\'an vortex streets \cite{G20} and near Thomson polygon equilibria \cite{G21}. Finally, Hassainia and Wheeler \cite{HW22} addressed the general setting of non-degenerate point-vortex equilibria, providing a unified desingularization framework. We also mention some related works using others approaches like variational techniques \cite{CLZ21,CWWZ21,SVS10,T85} or gluing methods \cite{DDMW20}.

Beyond perturbations of steady or rigidly rotating states, the construction of vortex patches near genuinely time-dependent point-vortex motions remained largely open. A breakthrough in this direction was achieved with the periodic desingularization of the four-vortex leapfrogging configuration by Hassainia, Hmidi, and Masmoudi \cite{HHM25}, providing the first example of patch dynamics shadowing a nontrivial periodic point-vortex orbit. More recently, the longstanding problem of constructing periodic patch motions in general bounded simply connected domains was resolved by Hassainia, Hmidi, and Roulley \cite{HHR24}, through the desingularization of periodic orbits of a single point vortex.
    
	\subsection{Barotropic model}
    Our next goal is to introduce the set of equations governing the vorticity evolution in a two-dimensional layer of ideal fluid with uniform density distributed over a spherical surface, a framework known as the \textit{barotropic model}. Despite its apparent simplicity, this representation is far from trivial. In fact, it has already yielded, in physics literature, significant understanding of phenomena such as the development of polar vortices \cite{JM87,W92}. More broadly, the barotropic model occupies a key position in geophysical fluid dynamics, offering a foundational approach for interpreting large-scale atmospheric motion and supporting the modeling of atmospheric processes on Earth as well as on other planets \cite{HH04}.\\
    
    Now and for the rest of the paper, we work on the unit sphere $\mathbb{S}^2$ defined by
	$$\mathbb{S}^2\triangleq\Big\{(x_1,x_2,x_3)\in\mathbb{R}^3\quad\textnormal{s.t.}\quad x_1^2+x_2^2+x_3^2=1\Big\},$$ 
	performing a uniform rotation around the vertical axis with constant angular speed $\gamma\in\mathbb{R}$.
    Throughout the document, we denote $|\cdot|_{\mathbb{R}^3}$ the Euclidean norm in $\mathbb{R}^3$, namely
    $$\forall(x_1,x_2,x_3)\in\mathbb{R}^3,\quad |(x_1,x_2,x_3)|_{\mathbb{R}^3}^2\triangleq x_1^2+x_2^2+x_3^2$$
    and we shall use the following notation for $x\in\mathbb{R}^3$ and $r>0,$
    $$B(x,r)\triangleq\Big\{y\in\mathbb{R}^3\quad\textnormal{s.t.}\quad|x-y|_{\mathbb{R}^3}<r\Big\}.$$
    Then, we consider an homogeneous and incompressible fluid on $\mathbb{S}^2$ described by its velocity field $u$ and its pressure $P$. The 2D Euler equations on the rotating sphere is
	\begin{equation}\label{eq vorticity}
\partial_{t}\omega(t,\mathbf{x})+u(t,\mathbf{x})\cdot\nabla\big(\omega(t,\mathbf{x})+2\gamma\,x_3\big)=0.
	\end{equation}
	The divergence-free property of the velocity field and the compactness of the manifold $\mathbb{S}^2$ implies that the vorticity should satisfy the so-called Gauss constraint, namely
	$$\int_{\mathbb{S}^2}\omega(t,\mathbf{x})\dd\boldsymbol{\sigma}(\mathbf{x})=0.$$
	Finally, we define the absolute vorticity through the relation
	$$\zeta(t,\mathbf{x})\triangleq\omega(t,\mathbf{x})+2\gamma\, x_3.$$
	According to \eqref{eq vorticity}, it is a solution to the following active scalar equation
	\begin{equation}\label{EulerAV}
	    \partial_{t}\zeta(t,\mathbf{x})+u(t,\mathbf{x})\cdot\nabla \zeta(t,\mathbf{x})=0.
	\end{equation}
	Moreover, the Gauss constraint is also satisfied by the absolute vorticity
	\begin{equation*}
		\int_{\mathbb{S}^2}\zeta(t,\mathbf{x})\dd\boldsymbol{\sigma}(\mathbf{x})=0.
	\end{equation*}
	The fact that $u$ is solenoidal implies the existence of a stream function $\Psi$ such that
	$$u(t,\mathbf{x})=\nabla^{\perp}\Psi(t,\mathbf{x}).$$
	The stream function solves the Poisson equation
	$$\Delta\Psi(t,\mathbf{x})=\omega(t,\mathbf{x})$$
	and therefore is linked to the vorticity via the following integral formula, see \cite{BD15}
	$$\Psi(t,\mathbf{x})=\Psi[\omega](t,\mathbf{x})\triangleq\int_{\mathbb{S}^2}G_{\mathbb{S}^2}(\mathbf{x},\mathbf{y})\omega(t,\mathbf{y})\dd\boldsymbol{\sigma}(\mathbf{y}),\qquad G_{\mathbb{S}^2}(\mathbf{x},\mathbf{y})\triangleq\frac{1}{2\pi}\ln\left(|\mathbf{x}-\mathbf{y}|_{\R^3}\right).$$
	The norm $|\cdot|_{\R^3}$ is the usual Euclidean norm in $\mathbb{R}^3.$ 
    In terms of absolute vorticity, we have
	$$u(t,\mathbf{x})=\nabla^{\perp}\big(\Psi[\zeta](\mathbf{x})+\gamma x_3\big)=\nabla^{\perp}\Psi[\zeta](t,\mathbf{x})+\gamma\nabla^{\perp}(\mathbf{x}\cdot\mathbf{e}_3).$$
    According to Lemma \ref{lem:gradlog}, the Biot-Savart law on the rotating sphere is 
    $$u(t,\mathbf{x})=\frac{1}{2\pi}\int_{\mathbb{S}^2}\frac{\mathbf{x}\wedge\mathbf{y}}{|\mathbf{x}-\mathbf{y}|^2_{\mathbb{R}^3}}\zeta(t,\mathbf{y})d\boldsymbol{\sigma}(\mathbf{y})+\gamma\,\mathbf{e}_3\wedge\mathbf{x}.$$
    Throughout the document, we shall denote the Biot-Savart kernel as
    \begin{equation}\label{BS-kernel}
        K_{\mathbb{S}^2}(\mathbf{x},\mathbf{y})\triangleq\frac{\mathbf{x}\wedge\mathbf{y}}{\,\,|\mathbf{x}-\mathbf{y}|_{\mathbb{R}^3}^2}\cdot
    \end{equation}

\subsection{Point-vortices on the rotating sphere}

An absolute vorticity point vortex distribution is a formal solution of \eqref{eq vorticity} in the form
$$\zeta(t,\mathbf{x})=\sum_{i=1}^{N}\Gamma_i\delta_{\mathbf{x}_i(t)},$$
where $N\in\mathbb{N}\setminus\{0,1\}$ is the number of points, $\mathbf{x}_1(t),\ldots,\mathbf{x}_N(t)\in\mathbb{S}^2$ are the points at time $t\geqslant0$ and $\Gamma_1,\ldots,\Gamma_N\in\mathbb{R}$ are the intensities subject to the Gauss constraint
    \begin{equation}\label{Gauss:PV}
        \sum_{i=1}^{N}\Gamma_i=0.
    \end{equation}
    In what follows, we denote
    $$\Gamma\triangleq\{\Gamma_1,\ldots,\Gamma_N\}.$$
    The point-vortex system on the rotating unit 2-sphere is the time evolution law for the points, namely
    \begin{equation}\label{PVS}
    \forall\, 1\leqslant i\leqslant N,\quad\begin{cases}
        \displaystyle \der{}{t} \mathbf{x}_i(t)=\sum_{j=1\atop j\neq i}^{N}\frac{\Gamma_j}{2\pi}\frac{\mathbf{x}_i(t)\wedge\mathbf{x}_j(t)}{|\mathbf{x}_i(t)-\mathbf{x}_j(t)|_{\R^3}^2}+\gamma\,\mathbf{e}_3\wedge\mathbf{x}_i(t), \vspace{1mm}\\
        \mathbf{x}_i(0) = \mathbf{x}_i^0.
    \end{cases}
    \end{equation}
    The model, introduced in \cite{bogomolov77}, was later studied in many works, and we refer the interested reader to \cite{KN98}, and in particular, the study of relative equilibria in \cite{CB03,LP05,LPMR11}, or in a more physically relevant context, \cite{MOKHOV2013}.
Let us mention that the dynamics \eqref{PVS} is Hamiltonian 
$$\forall\, 1\leqslant i\leqslant N,\quad\der{}{t}\mathbf{x}_i(t)=\frac{1}{\Gamma_i}\nabla_{\mathbf{x}_i}^{\perp}\mathcal{H}\big(\mathbf{x}_1(t),\ldots,\mathbf{x}_N(t)\big),$$
associated to the energy $\mathcal{H}$ (cf. Lemma \ref{lem:gradlog} and \eqref{gradperp:scalarprod}) related to the kinetic energy and center of mass, which are two conserved quantities
$$\mathcal{H}(\mathbf{x}_1,\ldots,\mathbf{x}_N)\triangleq\sum_{1\leqslant i,j\leqslant N\atop i\neq j} \frac{\Gamma_i\Gamma_j}{4\pi} \ln\big(|\mathbf{x}_i-\mathbf{x}_j|_{\R^3}\big) + \gamma \mathbf{e}_3 \cdot \sum_{i=1}^N \Gamma_i\mathbf{x}_i.$$
In \cite{CLW25}, the authors desingularized vortex pairs on the sphere at rest and briefly mention how to treat the rotating case. Later Sakajo and Sun studied the $C^1$ and patch-type regularization of Von-K\'arm\'an vortex streets \cite{SZ24,SZ25}. The bifurcation of one and two-interface vortex caps from zonal (i.e. longitude independent) solutions has been obtained in \cite{GHR25}. The filamentation phenomenon with linear growth of the perimeter near monotone zonal vortex caps has been studied in \cite{MR25} exploiting the stability result of monotone zonal flows of Caprino-Marchioro \cite{CM88} based on the conservation of the momentum with respect to the vertical axis.

On surfaces of non constant curvature, the point-vortex dynamics has a very different behaviour as the first order term is then given by the derivative of the curvature, see \cite{BD15}. One would expect on such manifold that the confinement results that can be obtained would be similar to the three-dimensional-with-symmetry cases, as previously discussed, coaxial vortex rings and helical filaments, or in the lake equation for instance. On the sphere, the constant curvature means that instead, the result that we obtain are similar to the two-dimensional results, resulting in the presence of non-trivial geometry, without the singularity of motion observed in the curved cases.

\subsection{Main results}

Our first result is the improbability of point-vortex collisions, necessary to justify that most initial data have global solutions. It reads informally as follows. 
    \begin{theo}\label{theo:improbabilite}\textbf{(Improbability of point vortex collision)}\\
        Let $N\in\mathbb{N}\setminus\{0,1\}.$ Then, for almost every initial conditions $(\mathbf{x}_i^0)_{1\leqslant i \leqslant N}\in(\mathbb{S}^2)^N$, the point vortex system \eqref{PVS} has a global solution.
    \end{theo}
    The precise statement and proof Theorem~\ref{theo:P0} are given in Section~\ref{sec:P0}. We then turn to the localization and confinement problems. Analogously to the planar case, for given pairwise-distinct points on the sphere $\mathbf{x}_1^0,\ldots,\mathbf{x}_N^0$, and intensities satisfying the Gauss condition \eqref{Gauss:PV}, let us consider a family $(\zeta_0^\eps)_{\eps > 0}$ of initial data to the Euler equation in absolute vorticity~\eqref{EulerAV} satisfying the following assumptions.
    \begin{hypothesis}\label{hyp:zeta}
    There exist constants $M,\varepsilon_0>0$ and $\eta \geqslant 2$ such that for every $\eps\in(0,\varepsilon_0)$,
    \begin{itemize}
        \item $\zeta_0^\eps\in L^\infty(\mathbb{S}^2)$, with $\|\zeta_0^\eps\|_{L^\infty(\mathbb{S}^2)} \leqslant M \eps^{-\eta}$.
        \item $\displaystyle \zeta_0^\eps = \sum_{i=1}^N \zeta_{0,i}^\eps \quad $ where $\displaystyle \quad \frac{\zeta_{0,i}^\eps}{\Gamma_i} \geqslant 0,\quad$ $\Gamma_i\triangleq\displaystyle\int_{\mathbb{S}^2}\zeta_{0,i}^\eps(\mathbf{x})\dd\boldsymbol{\sigma}(\mathbf{x})\quad$ and $\quad\displaystyle\sum_{i=1}^{N}\Gamma_i=0,$ 
        \item $\supp \zeta_{0,i}^{\eps} \subset B(\mathbf{x}_i^0,\eps)$,
    \end{itemize}
    These initial conditions provide a solution to equation~\eqref{EulerAV} denoted $\zeta^{\varepsilon}(t,\mathbf{x}).$
\end{hypothesis}
    For the same reasons that in the planar case, the decomposition as a sum of complactly supported vortices of circulation $\Gamma_i$ remains true at all time (these facts are given in details in Section~\ref{sec:estimates}). Then, define for any $\beta <1/2$ and $\eps > 0$ the exit time
     \begin{equation}\label{def:exit-time}
    \tau_{\eps,\beta}\triangleq\sup \left\{ t \geqslant 0 \quad\textnormal{s.t.}\quad\forall s \in [0,t],\quad \supp\big(\zeta^\eps(s,\cdot)\big) \subset \bigcup_{i=1}^N B\big(\mathbf{x}_i(s),\eps^\beta\big) \right\},
    \end{equation}

    We then prove the following result, analogous to Theorem~\ref{theo:BM18general}.
    \begin{theo}\label{theo:general_conf}\textbf{(Logarithmic time scale for vorticity confinement)}\\
        Let $(\mathbf{x}_i^0)_{1\leqslant i\leqslant N}$ be $N$ pairwise distinct points of $\mathbb{S}^2$ and $(\Gamma_i)_{1\leqslant i \leqslant N}$ be some non-vanishing intensities satisfying the Gauss condition \eqref{Gauss:PV}. Assume that such that the solution of the point-vortex dynamics~\eqref{PVS} with initial datum $(\mathbf{x}_i^0)_{1\leqslant i\leqslant N}$ is global in time and satisfies the following distance condition for some $d_0>0,$
        \begin{equation}\label{dist-min-points}
            \forall\,1\leqslant i,j\leqslant N, \quad i\neq j\quad\Rightarrow\quad\inf_{t \geqslant 0}|\mathbf{x}_i(t) - \mathbf{x}_j(t)|_{\R^3} \geqslant d_0.
        \end{equation}
        Then for every $\beta<1/2$ there exists $\eps_0\triangleq\varepsilon_0(\beta,d_0)>0$ and $\alpha\triangleq\alpha(\beta,d_0)>0$ such that for every $\eps \in (0,\eps_0)$ the solution $\zeta^\eps$ of \eqref{EulerAV} with initial condition $\zeta_0^\eps$ subjected to the Hypothesis \ref{hyp:zeta} near the points $(\mathbf{x}_i^0)_{1\leqslant i\leqslant N}$ satisfies
        \begin{equation*}
            \tau_{\eps,\beta} \geqslant \alpha|\ln \eps|.
        \end{equation*}
    \end{theo}
    This bound is optimal, as we show in the following result, analogous to the one obtained in \cite{D24a}.
    \begin{theo}\label{theo:optimal}\textbf{(Optimality of the logarithmic time confinement)}
    There exists a choice of $(\mathbf{x}_1,\ldots,\mathbf{x}_N)$ and intensities $\Gamma_1,\ldots,\Gamma_N$ satisfying \eqref{Gauss:PV} and \eqref{dist-min-points} such that there exists $\beta_0 < 1/2$, $\eta \geqslant 2$ such that for any $\beta \in (\beta_0,1)$, there exists $\alpha_0 > 0$ such that for any $\eps > 0$ small enough, there exists $\zeta_0^\eps$ satisfying Hypothesis~\ref{hyp:zeta} such that
    \begin{equation*}
        \tau_{\eps,\beta} \leqslant \alpha_0 |\ln \eps|.
    \end{equation*}
    \end{theo}
    Differently said, there exist configurations that realize a logarithmic exit time. However, under certain conditions, discussed in Section~\ref{sec:confinement_ameliore} and Section~\ref{sec:configurations}, the bound can be improved.
    \begin{theo}\label{theo:strong_conf}\textbf{(Improved confinement time for special configurations)}\\
        There exists a choice of $(\mathbf{x}_1,\ldots,\mathbf{x}_N)$ and intensities $\Gamma_1,\ldots,\Gamma_N$ satisfying \eqref{Gauss:PV} and \eqref{dist-min-points} such that for every $\beta<1/2$ and $\alpha < \min(\beta,2-4\beta)$, there exists $\eps_0>0$  such that for every $\eps \in (0,\eps_0)$, the solution $\zeta^\eps$ of \eqref{EulerAV} with initial condition $\zeta_0^\eps$ subjected to the Hypothesis \ref{hyp:zeta} near the points $(\mathbf{x}_i^0)_{1\leqslant i\leqslant N}$ satisfies
        \begin{equation*}
            \tau_{\eps,\beta} \geqslant \eps^{-\alpha}.
        \end{equation*}
    \end{theo}
    Here note that it is only the choice of the configuration and intensities that ensure that \emph{any} initial datum satisfying~\eqref{hyp:zeta} leads to a solution having a power-law exit time.

    \subsection{Comments on the particularity of the sphere and main new ideas}

    In the plane, both the Euler equations and the point-vortex dynamics satisfy a very important property, which is the conservation of the vortex impulse, respectively defined as
    \begin{equation*}
        m(t) = \int_{\R^2} x \omega(t,x) \dd x,
    \end{equation*}
    for the Euler equations and analogously
    \begin{equation*}
        m(t) = \sum_{i=1}^n a_i z_i(t),
    \end{equation*}
    for the point-vortex dynamics. On the (non-rotating) sphere, those quantities are also conserved for the absolute vorticity, but so defined, the center of absolute vorticity does \emph{not} lie on the sphere. Therefore, all usual estimates for the effect of an exterior Lipschitz-continuous velocity field which can be estimated by controlling quantities of the form $F(x) - F(b(t))$ cannot be written as $F$ is not defined \emph{a priori} outside of the sphere. However, a remarkable property of the Green Function of the sphere 
    \begin{equation*}
        G_{\mathbb{S}^2}(\mathbf{x},\mathbf{y}) = \frac{1}{2\pi} \ln \big(|\mathbf{x}-\mathbf{y}|_{\R^3}\big)
    \end{equation*}
    is that it can trivially be extended to $\R^3 \times \R^3 \setminus \{ \mathbf{x} = \mathbf{y} \}$, so that exterior velocity fields created by absolute vorticity can themselves be extended to $\R^3$. In the proof of our main result, we thus control the motion of centers of absolute vorticity that do not lie on the sphere, which adds some difficulties that we found to be mostly technical.

    In addition, the Gauss constraint prevents us from studying the case $N=1$, as we only consider blobs of absolute vorticity with constant signs. Finally, both the geometry and the rotation of the sphere influence vortex equilibria, so that the particular configurations that we study to prove Theorems~\eqref{theo:optimal} and~\ref{theo:strong_conf} slightly differ from the similar configuration used in the planar case. Nevertheless, we are able to discuss the case of polar vortex crystals in Section~\ref{sec:polar_VC}, which is of particular interest, for which we recover properties of the planar case at the limit when vortices are close to the pole.

    \medskip

    The paper is organized as follows. In Section~\ref{sec:P0}, we prove Theorem~\ref{theo:improbabilite}. This section is completely independent of the rest of the paper. Then in Section~\ref{sec:estimates}, we study general properties of the solution to equation~\eqref{EulerAV} with initial datum satisfying Hypothesis~\eqref{hyp:zeta}. With these estimates established, in Section~\ref{sec:confinement_general}, we prove Theorem~\ref{theo:general_conf}. Proofs of Theorems~\ref{theo:optimal} and~\ref{theo:strong_conf} are done in Sections~\ref{sec:confinement_general} and~\ref{sec:confinement_ameliore} respectively, conditionally to the existence of suitable configurations of point-vortices. Then in Section~\ref{sec:configurations}, we prove the existence of these configurations, closing the proofs of these theorems, and discuss various examples.

\medskip

\noindent \textbf{Acknowledgements :} Part of this work was conducted when Martin Donati was supported by the grant BOURGEONS ANR-23-CE40-0014-01 of the French National Research Agency. Emeric Roulley is supported by the ERC STARTING GRANT 2021 “Hamiltonian Dynamics, Normal Forms and Water Waves” (HamDyWWa), project Number: 101039762. Both authors thank Matthieu Brachet for the numerical simulations of vortices on the sphere.

\section{Improbability of point vortex collisions: proof of Theorem~\ref{theo:improbabilite}}\label{sec:P0}
This section is devoted to the proof of Theorem \ref{theo:improbabilite}.\\

We denote $\mathcal{B}(\mathbb{S}^2)$ the Borel $\sigma$-algebra of the unit sphere $\mathbb{S}^2.$  Let us recall that $\boldsymbol{\sigma}$ is the surface measure on $\big(\mathbb{S}^2,\mathcal{B}(\mathbb{S}^2)\big).$ It is well-known that the area of the unit sphere -- that can be recovered by simple integration -- is $\boldsymbol{\sigma}(\mathbb{S}^2)=4\pi.$ Therefore, the measure
$$\mathbb{P}\triangleq\frac{\boldsymbol{\sigma}}{4\pi}$$
is a probability measure on $\big(\mathbb{S}^2,\mathcal{B}(\mathbb{S}^2)\big).$  Fix $N\in\mathbb{N}\setminus\{0,1\}$, then the Borel $\sigma$-algebra on $(\mathbb{S}^2)^N$ is
$$\mathcal{B}\left((\mathbb{S}^2)^N\right)\triangleq\underbrace{\mathcal{B}(\mathbb{S}^2)\otimes\ldots\otimes\mathcal{B}(\mathbb{S}^2)}_{N\textnormal{ times}}.$$
We define the product probability measure on $\left((\mathbb{S}^2)^N,\mathcal{B}\left((\mathbb{S}^2)^N\right)\right)$ via
\begin{equation*}
            \mathbb{P}_N\triangleq\underbrace{\mathbb{P}\otimes\ldots\otimes\mathbb{P}}_{N\textnormal{ times}}.
\end{equation*}
Now, let us consider the set of admissible initial positions for the $N$-point-vortex dynamics.
\begin{equation}\label{def:admissible-set}
    \mathcal{A}_N\triangleq\left\{ \mathbf{X}=(\mathbf{x}_1,\ldots,\mathbf{x}_N)\in(\mathbb{S}^2)^N \quad\textnormal{s.t.}\quad\forall\, 1\leqslant i,j\leqslant N,\quad i\neq j\quad\Rightarrow\quad\mathbf{x}_i\neq\mathbf{x}_j\right\}.
\end{equation}
The set $\mathcal{A}_N$ is an open subset of $(\mathbb{S}^2)^N.$ Therefore, $\mathcal{A}_N\in\mathcal{B}\left((\mathbb{S}^2)^N\right).$ In addition, it is easy to show that it is of full measure, namely
$$\mathbb{P}_N(\mathcal{A}_N)=1.$$
The Borel $\sigma$-algebra of $\mathcal{A}_N$ is obtained by the trace topology
$$\mathcal{B}(\mathcal{A}_N)\triangleq\mathcal{A}_N\cap\mathcal{B}\left((\mathbb{S}^2)^N\right).$$
We define the conditional probability measure $\mathbb{P}_N(\cdot|\mathcal{A}_N)$ with respect to the almost sure event $\mathcal{A}_N$ by
$$\forall A\in\mathcal{B}\left((\mathbb{S}^2)^N\right),\quad\mathbb{P}_N(A|\mathcal{A}_N)\triangleq\frac{\mathbb{P}_N\left(A\cap\mathcal{A}_N\right)}{\mathbb{P}_N(\mathcal{A}_N)}=\mathbb{P}_N(A\cap\mathcal{A}_N).$$
Notice that $\mathbb{P}_N$ and $\mathbb{P}_N(\cdot|\mathcal{A}_N)$ coincide on $\mathcal{B}(\mathcal{A}_N).$ Therefore, in what follows, we still denote $\mathbb{P}_N$ restriction of $\mathbb{P}_N(\cdot|\mathcal{A}_N)$ to the $\sigma$-algebra $\mathcal{B}(\mathcal{A}_N)$ and we will work with the probability space $\left(\mathcal{A}_N,\mathcal{B}(\mathcal{A}_N),\mathbb{P}_N\right).$
For every $\mathbf{X}=\left(\mathbf{x}_1,\ldots,\mathbf{x}_N\right) \in \mathcal{A}_N$, by a classical application of the local Cauchy-Lipschitz theory, the point-vortex system~\eqref{PVS} with initial configuration $\mathbf{X}$ has a unique solution $t\mapsto S_t(\mathbf{X})\triangleq\big(\mathbf{x}_1(t),\ldots,\mathbf{x}_N(t)\big)$ continuous on a maximal time interval that we denote by $[0,T(\mathbf{X}))$, with $T(\mathbf{X}) \in  (0,+\infty]$. Let us observe that $T(\mathbf{X})<+\infty$ if and only if there is a collision (in the broad sense).

Let us denote by $\mathscr{C}_N$ the \emph{set of collisions}, namely the set of initial data leading to a finite time collision in the point-vortex dynamics
\begin{equation}\label{def:CN}
    \mathscr{C}_N \triangleq \{ \mathbf{X}\in \mathcal{A}_N \quad\textnormal{s.t.}\quad T(\mathbf{X}) < +\infty \}.
\end{equation}
\begin{remark}
    A priori, it is not immediate that the function $\mathbf{X}\mapsto T(\mathbf{X})$ is measurable and therefore one cannot a priori state that the set $\mathscr{C}_N$ is measurable.
\end{remark} 
The Theorem \ref{theo:improbabilite} can be reformulated as follows.
\begin{theo}\label{theo:P0}
    Let $N\in\mathbb{N}\setminus\{0,1\}.$ Then, 
    $$\mathscr{C}_N\in\mathcal{B}(\mathcal{A}_N)\qquad\textnormal{and}\qquad\mathbb{P}_N(\mathscr{C}_N)=0.$$
\end{theo}
\begin{proof} 
We consider the notion of almost collision and its associated regularized dynamics.
Fix $\varepsilon\in(0,1)$ and consider $\ln_\eps\in C^\infty\big([0,+\infty),\R\big)$ a non-decreasing function satisfying the conditions
\begin{equation}\label{def:log_eps}
    \begin{aligned}
        \forall r \in [\eps,+\infty), \quad \ln_\eps(r)&=\ln(r),\\
        \forall r>0,\quad|\ln_{\varepsilon}(r)|&\leqslant |\ln(r)|,\\
        \forall r>0,\quad\ln_{\varepsilon}'(r)&\leqslant \frac{1}{r}\cdot
    \end{aligned}
\end{equation}
Then, we define the regularized energy $\mathcal{H}_{\varepsilon}:(\mathbb{S}^2)^N\to\mathbb{R}$ by
\begin{equation*}
    \mathcal{H}_\eps(\mathbf{x}_1,\ldots,\mathbf{x}_N)\triangleq\sum_{1\leqslant i,j\leqslant N\atop i\neq j} \frac{\Gamma_i\Gamma_j}{4\pi} \ln_\eps\big(|\mathbf{x}_i-\mathbf{x}_j|_{\R^3}\big) + \gamma \mathbf{e}_3 \cdot \sum_{i=1}^N\Gamma_i\mathbf{x}_i 
\end{equation*}
leading, by virtue of \eqref{gradperp:scalarprod}, to the Hamiltonian regularized dynamics
\begin{equation}\label{eq:regularized}
    \forall\,1\leqslant i\leqslant N,\quad\der{}{t} \mathbf{x}^\eps_i(t) =\frac{1}{\Gamma_i}\nabla_{\mathbf{x}_i}^{\perp}\mathcal{H}_{\varepsilon}\big(\mathbf{x}_1^{\varepsilon}(t),\ldots,\mathbf{x}_N^{\varepsilon}(t)\big)=\sum_{j=1\atop j\neq i}^{N}\frac{\Gamma_j}{2\pi}\nabla_{\mathbf{x}_i}^\perp \ln_\eps \big(|\mathbf{x}_i^\eps(t)-\mathbf{x}_j^\eps(t)|_{\R^3}\big)+\gamma\,\mathbf{e}_3\wedge\mathbf{x}_i^\eps(t).
\end{equation}
Since $\ln_\eps$ is smooth on $[0,+\infty)$, by a trivial application of the Cauchy-Lipshitz theory, this dynamics has a global smooth solution for every initial data in $(\mathbb{S}^2)^N$. We denote by $(t,\mathbf{X})\mapsto S_t^{\varepsilon}(\mathbf{X})\triangleq\big(\mathbf{x}_1^{\varepsilon}(t),\ldots,\mathbf{x}_N^{\varepsilon}(t)\big)$ the flow of the regularized system~\eqref{eq:regularized} that is continuous is both variables and therefore defines a stochastic process $(S_t^{\varepsilon})_{t\geqslant0}$ (i.e. a continuous family of measurable functions) over $\big((\mathbb{S}^2)^N,\mathcal{B}\big((\mathbb{S}^2)^N\big)\big).$ Let us define for $\mathbf{X}\in\mathcal{A}_N$ the first time of $\eps$-collision $T_\eps(\mathbf{X})$ given by
\begin{equation}\label{def:Teps}
    T_\eps(\mathbf{X}) \triangleq \inf \{ t >0 \quad\textnormal{s.t.}\quad \exists\, 1\leqslant i_0,j_0\leqslant N,\quad i_0\neq j_0 \quad\textnormal{and}\quad |\mathbf{x}_{i_0}^{\varepsilon}(t)-\mathbf{x}_{j_0}^{\varepsilon}(t)|_{\R^3} \leqslant \eps \}.
\end{equation}
Observe that we can write
$$T_{\varepsilon}(\mathbf{X})=\inf\left\{t\geqslant0\quad\textnormal{s.t.}\quad S_t^{\varepsilon}(\mathbf{X})\in(\mathbb{S}^2)^N\setminus\mathcal{A}_N^{\varepsilon}\right\},$$
where
$$\mathcal{A}_N^{\varepsilon}\triangleq\left\{ \mathbf{X}=(\mathbf{x}_1,\ldots,\mathbf{x}_N)\in(\mathbb{S}^2)^N \quad\textnormal{s.t.}\quad\forall\, 1\leqslant i,j\leqslant N,\quad i\neq j\quad\Rightarrow\quad|\mathbf{x}_i-\mathbf{x}_j|_{\mathbb{R}^3}>\varepsilon\right\}.$$
With this expression, thanks to the measurability of $(S_t^{\varepsilon})_{t\geqslant0},$ we can say that $T_{\varepsilon}$ is a random variable (i.e. a measurable function) over $\big(\mathcal{A}_N,\mathcal{B}(\mathcal{A}_N)\big),$ that is a stopping time adapted to the natural filtration $(\mathcal{F}_t^{\varepsilon})_{t\geqslant0}$ of $(S_t^{\varepsilon})_{t\geqslant0}$ defined by
$$\mathcal{F}_t^{\varepsilon}\triangleq\sigma(S_s^{\varepsilon},\,0\leqslant s\leqslant t).$$
By construction of $\ln_\eps$ and $T_{\varepsilon}$ in \eqref{def:log_eps}-\eqref{def:Teps} and uniqueness of the local solutions, we have that
\begin{align}
&\forall \, \mathbf{X}\in\mathcal{A}_N,\quad T_{\varepsilon}(\mathbf{X})\leqslant T(\mathbf{X}),\label{TeminT}\\
    &\forall \, \mathbf{X} \in \mathcal{A}_N, \quad \forall t \leqslant T_\eps(\mathbf{X}), \quad S_t(\mathbf{X})=S_t^\eps(\mathbf{X}),\label{coincide}
\end{align}
Because of \eqref{coincide}, the solution $(t,\mathbf{X})\mapsto S_t^{\varepsilon}(\mathbf{X})$ is called \emph{regularized dynamics until $\eps$-collisions} since it coincides with the real point-vortex dynamics until the first $\eps$-collision.
The set $\mathscr{C}_N$ defined in \eqref{def:CN} can be written
\begin{equation*}
    \mathscr{C}_N=\bigcup_{\tau\in\mathbb{N}^*}\{\mathbf{X}\in\mathcal{A}_N\quad\textnormal{s.t.}\quad T(\mathbf{X})\leqslant\tau\}.
\end{equation*}
Fix $\tau\in\mathbb{N}^*$, then the inequality \eqref{TeminT} implies the inclusion
\begin{equation*}
    \{\mathbf{X}\in\mathcal{A}_N\quad\textnormal{s.t.}\quad T(\mathbf{X})\leqslant\tau\}\subset\{ \mathbf{X} \in \mathcal{A}_N  \quad\textnormal{s.t.}\quad T_\eps(\mathbf{X})\leqslant\tau\}.
\end{equation*}
Now, we fix $\eta \in (0,1)$ and define the continuous function $\phi_\eps : (\mathbb{S}^2)^N \to \R_+^*$ through
\begin{equation*}
    \phi_\eps(\mathbf{X})\triangleq\sum_{1\leqslant i,j\leqslant N\atop i\neq j} \exp\left(- \eta \ln_\eps \big(|\mathbf{x}_i - \mathbf{x}_j|_{\mathbb{R}^3}\big)\right).
\end{equation*}
Assume now that $\mathbf{X}\in\mathcal{A}_N$ is such that $T_\eps(\mathbf{X}) < +\infty$. Then, there exists $(i_0,j_0)\in\llbracket 1,N\rrbracket^2$ with $i_0\neq j_0$ such that
$$\left|\mathbf{x}_{i_0}\big(T_{\varepsilon}(\mathbf{X})\big)-\mathbf{x}_{j_0}\big(T_{\varepsilon}(\mathbf{X})\big)\right|_{\mathbb{R}^3}=\varepsilon.$$
Since all the terms in the sum defining the function $\phi_{\varepsilon}$ are positive, then
\begin{equation*}
    \phi_\eps\big(S_{T_\eps(\mathbf{X})}^\eps(\mathbf{X})\big) \geqslant \exp\left(-\eta\ln_{\varepsilon}\left(\left|\mathbf{x}_{i_0}\big(T_{\varepsilon}(\mathbf{X})\big)-\mathbf{x}_{j_0}\big(T_{\varepsilon}(\mathbf{X})\big)\right|_{\mathbb{R}^3}\right)\right)=\eps^{-\eta}.
\end{equation*}
The previous estimate implies the following inclusion
\begin{equation*}
   \{ \mathbf{X} \in \mathcal{A}_N\quad\textnormal{s.t.}\quad T_\eps(\mathbf{X}) \leqslant \tau \} \subset \left\{ \mathbf{X}\in (\mathbb{S}^2)^N \quad\textnormal{s.t.}\quad \sup_{t \in [0,\tau]} \phi_\eps\big(S_t^\eps(\mathbf{X})\big) \geqslant \eps^{-\eta}\right\}.
\end{equation*}
Applying Markov inequality yields
\begin{equation*}
    \mathbb{P}_N\Big(\{ \mathbf{X} \in \mathcal{A}_N \quad\textnormal{s.t.}\quad T_{\varepsilon}(\mathbf{X}) \leqslant \tau \}  \Big) \leqslant \eps^\eta \int_{(\mathbb{S}^2)^N} \sup_{t \in [0,\tau]} \phi_\eps\big(S_t^\eps (\mathbf{X})\big) \dd \mathbb{P}_N(\mathbf{X}).
\end{equation*}
Then we claim the following.
\begin{lem}\label{lem:bound_improb}
    There exists a constant $C$ depending only on $N$, the intensities $\Gamma_1,\ldots,\Gamma_N$, and $\eta$ such that for every $\tau\in\mathbb{N}^*$,
    \begin{equation*}
        \int_{(\mathbb{S}^2)^N} \sup_{t \in [0,\tau]} \phi_\eps(S_t^\eps \mathbf{X}) \dd \mathbb{P}_N(\mathbf{X}) \leqslant C(1+\tau).
    \end{equation*}
\end{lem}
We delay the proof of Lemma \ref{lem:bound_improb} for the time being, to conclude that
\begin{equation}\label{bnd:probTE}
    \mathbb{P}_N\Big(\{ \mathbf{X} \in \mathcal{A}_N\quad\textnormal{s.t.}\quad T_{\varepsilon}(\mathbf{X}) \leqslant\tau \} \Big) \leqslant C\eps^\eta(1+\tau).
\end{equation}
We consider a sequence $(\varepsilon_n)_{n\in\mathbb{N}}\in(0,1)^{\mathbb{N}}$ decreasing and converging to $0.$ 
Therefore, by virtue of \eqref{bnd:probTE},
\begin{equation}\label{meas:Ten}
    \forall n\in\mathbb{N},\quad\mathbb{P}_N\Big(\{ \mathbf{X} \in \mathcal{A}_N\quad\textnormal{s.t.}\quad T_{\varepsilon_n}(\mathbf{X}) \leqslant\tau \} \Big) \leqslant C\eps_n^\eta(1+\tau).
\end{equation}
By decreasing property, the $\varepsilon_n$-collision must happen before the $\varepsilon_{n+1}$-collision. Hence, we have
\begin{equation}\label{croi:Ten}
    \forall\, \mathbf{X}\in\mathcal{A}_N,\quad\forall n\in\mathbb{N},\quad T_{\varepsilon_n}(\mathbf{X})\leqslant T_{\varepsilon_{n+1}}(\mathbf{X}).
\end{equation}
Invoking the monotone convergence theorem, the upper bound \eqref{TeminT} implies the convergence of the increasing sequence $\big(T_{\varepsilon_n}(\mathbf{X})\big)_{n\in\mathbb{N}}.$ Since the solution lives on the sphere, by contraposition of the principle of a priori majoration, one must have
$$\lim_{n\to+\infty}T_{\varepsilon_n}(\mathbf{X})=T(\mathbf{X}).$$
The property \eqref{croi:Ten} implies that the family
$$\big(\{\mathbf{X}\in\mathcal{A}_N\quad\textnormal{s.t.}\quad T_{\varepsilon_n}(\mathbf{X})\leqslant\tau\}\big)_{n\in\mathbb{N}}$$
is decreasing for the inclusion operation. Besides, as a countable intersection, the set
$$\{\mathbf{X}\in\mathcal{A}_N\quad\textnormal{s.t.}\quad T(\mathbf{X})\leqslant\tau\}=\bigcap_{n\in\mathbb{N}}\{\mathbf{X}\in\mathcal{A}_N\quad\textnormal{s.t.}\quad T_{\varepsilon_n}(\mathbf{X})\leqslant\tau\}\in\mathcal{B}(\mathcal{A}_N).$$
Then, by continuity of the measure and \eqref{meas:Ten}, we get
\begin{align*}
    \mathbb{P}_N\big(\{\mathbf{X}\in\mathcal{A}_N\quad\textnormal{s.t.}\quad T(\mathbf{X})\leqslant\tau\}\big)&=\mathbb{P}_N\left(\bigcap_{n\in\mathbb{N}}\{\mathbf{X}\in\mathcal{A}_N\quad\textnormal{s.t.} \quad T_{\varepsilon_n}(\mathbf{X})\leqslant\tau\}\right)\\
    &=\lim_{n\to+\infty}\mathbb{P}_N\left(\{\mathbf{X}\in\mathcal{A}_N\quad\textnormal{s.t.}\quad T_{\varepsilon_n}(\mathbf{X})\leqslant\tau\right)=0.
\end{align*}
As a countable union,
$$\mathscr{C}_N=\bigcup_{\tau\in\mathbb{N}^*}\{\mathbf{X}\in\mathcal{A}_N\quad\textnormal{s.t.}\quad T(\mathbf{X})\leqslant\tau\}\in\mathcal{B}(\mathcal{A}_N)$$
and by Boole inequality
$$0\leqslant\mathbb{P}_N(\mathscr{C}_N)\leqslant\sum_{\tau\in\mathbb{N}^*}\mathbb{P}_N\big(\{\mathbf{X}\in\mathcal{A}_N\quad\textnormal{s.t.}\quad T(\mathbf{X})\leqslant\tau\}\big)=0,\qquad\textnormal{i.e.}\qquad\mathbb{P}_N(\mathscr{C}_N)=0.$$
This concludes the proof of Theorem \ref{theo:P0}.
\end{proof}
We are left to prove Lemma~\ref{lem:bound_improb}.
\begin{proof}[Proof of Lemma \ref{lem:bound_improb}]
Let $\tau\in\mathbb{N}^*.$ We define
    \begin{equation*}
    \Phi_\eps(t,\mathbf{X})\triangleq\phi_\eps\big(S_t^\eps(\mathbf{X})\big).
\end{equation*}
Then, since the system~\eqref{eq:regularized} is autonomous, then for any $s,t\in[0,T(\mathbf{X}))$ with $t+s\in[0,T(\mathbf{X})),$ we have
$$S_{t+s}^{\varepsilon}(\mathbf{X})=S_t^{\varepsilon}\big(S_s^{\varepsilon}(\mathbf{X})\big).$$
Hence,
\begin{equation*}
    \Phi_\eps(t,\mathbf{X}) = \phi_{\varepsilon}\big(S_t^{\varepsilon}(\mathbf{X})\big)=\phi_{\varepsilon}\Big(S_0^{\varepsilon}\big(S_t^{\varepsilon}(\mathbf{X})\big)\Big)=\Phi_\eps\big(0,S_t^{\varepsilon}(\mathbf{X})\big).
\end{equation*}
Therefore,
\begin{align*}
    \phi_\eps\big(S_t^\eps(\mathbf{X})\big)&=\phi_{\varepsilon}\big(S_0^{\varepsilon}(\mathbf{X})\big)+\int_{0}^{t}\partial_{t}\phi_{\varepsilon}\big(S_s^{\varepsilon}(\mathbf{X})\big)\dd s\\
    &=\phi_\eps(\mathbf{X}) + \int_0^t \partial_t \Phi_\eps\big(0,S_s^\eps(\mathbf{X})\big)\dd s.
\end{align*}
Consequently,
\begin{equation*}
    \sup_{t \in [0,\tau]} \phi_\eps(S_t^\eps \mathbf{X})  \leqslant \phi_\eps(\mathbf{X}) + \int_0^\tau \left|\partial_t \Phi_\eps\big(0,S_s^\eps(\mathbf{X})\big)\right|\dd s.
\end{equation*}
By using the Fubini-Tonelli Theorem and the fact that the flow $t\mapsto S_t^\eps$ is Hamiltonian, we infer
\begin{align*}
    \int_{(\mathbb{S}^2)^N} \sup_{t \in [0,\tau]} \phi_\eps(S_t^\eps \mathbf{X}) \dd \mathbb{P}_N(\mathbf{X}) &  \leqslant\int_{(\mathbb{S}^2)^N} \phi_\eps(\mathbf{X}) \dd \mathbb{P}_N(\mathbf{X}) + \int_0^\tau \int_{(\mathbb{S}^2)^N} \left|\partial_t \Phi_\eps\big(0,S_s^\eps(\mathbf{X})\big)\right|\dd \mathbb{P}_N(\mathbf{X})\dd s \\
    & = \int_{(\mathbb{S}^2)^N} \phi_\eps(\mathbf{X})\dd \mathbb{P}_N(\mathbf{X}) + \tau \int_{(\mathbb{S}^2)^N} |\partial_t \Phi_\eps(0,\mathbf{X})|\dd \mathbb{P}_N(\mathbf{X}).
\end{align*}
This new expression only involves the computation of properties of the flow at time 0, meaning that at this point the regularization is not needed anymore. Indeed, for $\mathbf{X} \in \mathcal{A}_N$, therefore in particular for almost every $\mathbf{X}\in(\mathbb{S}^2)^N$, using the conditions \eqref{def:log_eps}, we get
\begin{equation*}
    \phi_\eps(\mathbf{X})\leqslant\phi(\mathbf{X}), \qquad |\partial_t \Phi_\eps(0,\mathbf{X})| \leqslant|\partial_t \Phi(0,\mathbf{X})|,
\end{equation*}
where 
\begin{equation*}
    \phi(\mathbf{X})\triangleq\begin{cases}
        \displaystyle\sum_{1\leqslant i,j\leqslant N\atop i\neq j} |\mathbf{x}_i-\mathbf{x}_j|_{\R^3}^{-\eta}, & \textnormal{if }\mathbf{X}\in\mathcal{A}_N,\\
        0, & \textnormal{otherwise}
    \end{cases}\qquad\textnormal{and}\qquad\Phi(t,\mathbf{X})\triangleq\begin{cases}
        \phi\big(S_t(\mathbf{X})\big), & \textnormal{if }\mathbf{X}\in\mathcal{A}_N,\\
        0, & \textnormal{otherwise.}
    \end{cases}
\end{equation*}
First remark that by Fubini-Tonelli Theorem
\begin{align*}
    \int_{(\mathbb{S}^2)^N} \phi(\mathbf{X}) \dd \mathbb{P}_N(\mathbf{X}) &= \frac{1}{(4\pi)^N}\sum_{1\leqslant i,j\leqslant N\atop i \neq j} \int_{\mathbb{S}^2} \ldots \int_{\mathbb{S}^2} |\mathbf{x}_i-\mathbf{x}_j|_{\R^3}^{-\eta} \dd \boldsymbol{\sigma}(\mathbf{x}_1) \ldots \dd \boldsymbol{\sigma}(\mathbf{x}_N)\\
    &=\frac{1}{(4\pi)^N}\sum_{1\leqslant i,j\leqslant N\atop i \neq j} \int_{\mathbb{S}^2}\ldots\int_{\mathbb{S}^2}\left(\int_{\mathbb{S}^2} |\mathbf{x}_i-\mathbf{x}_j|_{\R^3}^{-\eta} \dd \boldsymbol{\sigma}(\mathbf{x}_j)\right)\dd \boldsymbol{\sigma}(\mathbf{x}_1)\ldots\dd \boldsymbol{\sigma}(\mathbf{x}_{j-1})\dd \boldsymbol{\sigma}(\mathbf{x}_{j+1})\ldots\dd \boldsymbol{\sigma}(\mathbf{x}_N).
\end{align*}
But, using the rotation invariance and spherical coordinates, for any $\alpha>0$ and any $1\leqslant i\leqslant N,$ we have (recall that $\mathtt{N}$ is the north pole)
\begin{align*}
    \int_{\mathbb{S}^2} |\mathbf{x}_i-\mathbf{x}_j|_{\R^3}^{-\alpha} \dd \boldsymbol{\sigma}(\mathbf{x}_j)&=\int_{\mathbb{S}^2} |\mathtt{N}-\mathbf{x}_j|_{\R^3}^{-\alpha} \dd \boldsymbol{\sigma}(\mathbf{x}_j)\\
    &=\int_{0}^{2\pi}\int_{0}^{\pi}\frac{\sin(\theta)}{|\mathtt{N}-\psi_1(\theta,\varphi)|_{\mathbb{R}^3}^{\alpha}}d\theta d\varphi\\
    &=2^{2-\alpha}\pi\int_{0}^{\pi}\frac{\cos\left(\frac{\theta}{2}\right)}{\sin^{\alpha-1}\left(\frac{\theta}{2}\right)}d\theta\triangleq 4\pi C_{\alpha}.
\end{align*}
Consequently, the previous integral is independant of $i$ and by comparison with Riemann integrals, we get
\begin{equation}\label{integrability_cond}
    \int_{\mathbb{S}^2} |\mathbf{x}_i-\mathbf{x}_j|_{\R^3}^{-\alpha} \dd \boldsymbol{\sigma}(\mathbf{x}_j)<\infty\qquad\textnormal{iff}\qquad\alpha<2.
\end{equation}
Since $\eta\in(0,1)$, we get the integrability and
$$\int_{(\mathbb{S}^2)^N} \phi(\mathbf{X}) \dd \mathbb{P}_N(\mathbf{X})=N(N-1)C_{\eta}<+\infty.$$
Besides, 
$$\partial_t\Phi(0,\mathbf{X})=\nabla\phi(\mathbf{X})\cdot\begin{pmatrix}
    \dot{\mathbf{x}}_1(0)\\
    \vdots\\
    \dot{\mathbf{x}}_N(0)
\end{pmatrix}.$$
By \eqref{PVS}, \eqref{eq:maj_x_wedge_y}, and Cauchy-Schwarz inequality, we infer
\begin{align*}
    |\partial_t \Phi(0,\mathbf{X})| & = \eta\left|\sum_{i=1}^N \sum_{j=1\atop j\neq i}^{N} \sum_{k=1\atop k\neq i}^{N} \left(\frac{\Gamma_i}{2\pi} \frac{\mathbf{x}_i \wedge \mathbf{x}_k}{|\mathbf{x}_i -\mathbf{x}_k|^2_{\R^3}} +\gamma \mathbf{e}_3 \wedge \mathbf{x}_i \right) \cdot\frac{\mathbf{x}_i-\mathbf{x}_j}{|\mathbf{x}_i-\mathbf{x}_j|_{\R^3}^{\eta+2}} \right| \\
    & \leqslant C \sum_{i=1}^N \sum_{j=1\atop j\neq i}^{N} \sum_{k=1\atop k\neq i}^{N} \frac{1}{|\mathbf{x}_i - \mathbf{x}_k|_{\R^3}} \frac{1}{|\mathbf{x}_i-\mathbf{x}_j|_{\R^3}^{\eta+1}}\cdot
\end{align*}
Since $\eta+1<2$, proceeding as before and using \eqref{integrability_cond}, we can conclude
\begin{equation*}
    \int_{(\mathbb{S}^2)^N} |\partial_t \Phi(0,\mathbf{X})| \dd \mathbb{P}_N(\mathbf{X}) < +\infty.
\end{equation*}
Combining the foregoing calculation ends the proof of Lemma~\ref{lem:bound_improb}.
\end{proof}

\section{Estimates on vortex evolutions}\label{sec:estimates}

In this section, we establish some important properties on the solution to equations \eqref{EulerAV} with initial datum $\zeta_0^{\varepsilon}$ satisfying Hypothesis~\ref{hyp:zeta}. The initial positions $\mathbf{x}_1^0,\ldots,\mathbf{x}_N^0 \in \mathbb{S}^2$ and intensities $\Gamma_1,\ldots,\Gamma_N \in \R^*$ are fixed once and for all satisfying the Gauss constraint~\ref{Gauss:PV} such that the associated solution of the point-vortex dynamics~\eqref{PVS} satisfies~\eqref{dist-min-points}. We denote constants whose value is not important by $C$, and those constants are allowed to depend on $N$, $\mathbf{x}_1^0,\ldots,\mathbf{x}_N^0$ and intensities $\Gamma_1,\ldots,\Gamma_N$.

Let us consider an initial datum satisfying the Hypothesis~\ref{hyp:zeta}, namely a superposition of compactly supported blobs with disjoint supports. We denote $(t,\mathbf{x})\mapsto\zeta^\eps(t,\mathbf{x})$ the unique global-in -time associated weak solution of equations \eqref{EulerAV} provided by the Yudovich theory \cite{Y63}. Due to the transport nature of the equation \eqref{EulerAV}, the blob structure is preserved (at least locally in time) that is the solution decomposes as 
$$\zeta^\eps = \sum_{i=1}^N \zeta_i^\eps,$$
with for any $i\in\{1,\ldots,N\},$ the $\zeta_i^{\eps}$ being a blob that satisfies the Lagrangian property
$$\zeta_i^{\varepsilon}(t,\mathbf{x})=\zeta_{0,i}^{\varepsilon}\big(\phi_{t}^{-1}(\mathbf{x})\big),\qquad\partial_{t}\phi_t(\mathbf{x})=u\big(t,\phi_t(\mathbf{x})\big),\qquad\phi_0(\mathbf{x})=\mathbf{x}.$$
For the rest of this section, we fix an index $i\in\{1,\ldots,N\}$. Since (at least for short time) the supports of the blobs are disjoints, then the blob $\zeta_i^{\varepsilon}$ solves the following problem (locally in time)
\begin{equation}\label{pb:zetai}
\begin{cases}    \partial_{t}\zeta_i^{\varepsilon}+(u_i^{\varepsilon}+F_i^{\varepsilon})\cdot\nabla\zeta_i^{\varepsilon}=0,\vspace{1mm}\\
    \displaystyle u_i^{\varepsilon}(t,\mathbf{x})\triangleq\frac{1}{2\pi}\int_{\mathbb{S}^2} \frac{\mathbf{x}\wedge\mathbf{y}}{|\mathbf{x}-\mathbf{y}|^2_{\R^3}} \zeta_i^\eps(t,\mathbf{y})\dd\boldsymbol{\sigma}(\mathbf{y}), 
\end{cases}
\end{equation}
where the perturbation field $F_i^{\varepsilon}$ is defined by
\begin{equation}\label{def:Fi}
     F_i^{\varepsilon}(t,\mathbf{x})\triangleq\frac{1}{2\pi}\sum_{j=1\atop j\neq i}^{N}\int_{\mathbb{S}^2}\frac{\mathbf{x}\wedge\mathbf{y}}{|\mathbf{x}-\mathbf{y}|_{\mathbb{R}^3}^2}\zeta_{j}^{\varepsilon}(t,\mathbf{y})\dd\boldsymbol{\sigma}(\mathbf{y})+\gamma \mathbf{e}_3\wedge\mathbf{x}.
\end{equation}
Doing so, we simply denote the influence of the other blobs as an exterior velocity field $F_i^\eps$. Since for any $\mathbf{y}\in\mathbb{S}^2$, we have $\mathbf{x}\cdot(\mathbf{x}\wedge\mathbf{y})=0,$ then $F_i^{\varepsilon}(t,\mathbf{x})\in T_{\mathbf{x}}\mathbb{S}^2=\mathtt{span}^{\perp}(\mathbf{x})$ where the orthogonal is understood in the sense of the usual scalar product in $\mathbb{R}^3.$ By construction of the exit time $\tau_{\varepsilon,\beta}$ in \eqref{def:exit-time}, for $\varepsilon$ small enough, we obtain from the minimal distance assumption \eqref{dist-min-points} that for every $t \leqslant \tau_{\eps,\beta}$ and for every $i\neq j \in \{1,\ldots,N\}$,
\begin{equation}\label{distance-supports-timet}
    \dist \Big( B\big(\mathbf{x}_i(t),\eps^\beta\big) \, , \, \supp\big(\zeta_j^\eps(t,\cdot)\big)\Big) \geqslant \frac{d_0}{2},
\end{equation}
where by definition for two subsets $A,B\subset\mathbb{R}^3,$
$$\dist(A,B)\triangleq\inf_{a\in A\atop b\in B}\|a-b\|_{\mathbb{R}^3}.$$
As a consequence, the function $F_i$ can be extended into
\begin{equation}\label{extension:Fi}
    F_i^\eps\in C^0\Big([0,\tau_{\varepsilon,\beta}]\,,\,C^\infty\big( B\big(\mathbf{x}_i(t),\eps^\beta\big),\mathbb{R}^3\big)\Big).
\end{equation}
Due to \eqref{extension:Fi}, there exists a constant $D$ independent of $\eps$ such that
\begin{equation}\label{maj:F_Lip}
    \max_{t \in [0,\tau_{\eps,\beta}]} \max_{1 \leqslant i \leqslant N} \sup_{\mathbf{x},\mathbf{y}\in B(\mathbf{x}_i(t),\eps^\beta)} \frac{|F_i^\eps(t,\mathbf{x})-F_i^\eps(t,\mathbf{y})|}{|\mathbf{x}-\mathbf{y}|} \leqslant D.
\end{equation}
\begin{defin}\label{def:Dprime} Let us denote by $D_\eps$ the smallest constant such that for all $t\leqslant \tau_{\eps,\beta}$, for all $i \in \{1,\ldots,N\}$, there holds both that
\begin{itemize}
    \item for all $x,y \in  B(\mathbf{x}_i(t),\eps^\beta)$ such that $x\neq y$,
    \begin{equation}\label{eq:Flip1}
     \frac{\Big|\big(F_i^\eps(t,x)-F_i^\eps(t,y)\big)\cdot(x-y\big)\Big|}{|x-y|^2}  \leqslant D_\eps
\end{equation}
    \item for all $x \in B(\mathbf{x}_i(t),\eps^\beta)$,
    \begin{equation}\label{eq:Flip2}
    \left| \left(F_i^\eps(t,x) -\int_{\mathbb{S}^2} F_i^\eps(t,\mathbf{y}) \frac{\zeta_i^\eps(t,\mathbf{y})}{\Gamma_i}\dd\boldsymbol{\sigma}(\mathbf{y})\right) \cdot(x-c_i^\eps(t)\big)\right| \leqslant D_\eps |x-c_i^\eps(t)| \int_{\mathbb{S}^2} |x-\mathbf{y}|\frac{\zeta_i^\eps(t,\mathbf{y})}{\Gamma_i}\dd\boldsymbol{\sigma}\cdot
\end{equation}
\end{itemize}
\end{defin}
One can easily check that by relation~\eqref{maj:F_Lip}, $D_\eps \leqslant D$, which in particular ensures that $D_\eps < \infty$. However, in the proof of Theorem~\ref{theo:strong_conf}, the particular choice of configuration will lead to prove that $D_\eps = \mathcal{O}(\eps^\beta)$ as $\eps \to0$ in the more precise bounds~\eqref{eq:Flip1} and~\eqref{eq:Flip2}, which will be a crucial tool to obtain longer confinement times.

\subsection{Estimates on the vorticity moments}
A key point for describing the mass spreading is the control of the moments. Indeed, the first moment, the center of mass, is used as a \emph{localization} property, namely where in space is located the support of the absolute vorticity. Higher-order moments are a measure of concentration, which we call \emph{weak confinement}: when those moments are small, it means that the support of the vorticity is \emph{mostly} concentrated. Concluding the proof of Theorem~\ref{theo:BM18general} cannot rely only on weak confinement, as it requires that support is completely controlled: in the end we obtain what we call this time \emph{strong confinement} by controlling each particle trajectories carrying non vanishing absolute vorticity, using the weak confinement estimates. This is the now standard method introduced by Marchioro and Pulvirenti (see for instance \cite{MP93b}) and refined over time in many works. In particular, in this paper we will obtain weak confinement estimates through the control of higher-order vorticity moments, as introduced in \cite{ISG99}, then used in \cite{Donati_2021} in the context of concentrated vortices.

Before defining the moments let us first reprove that the mass is a conserved quantity. By applying the divergence Theorem and invoking the divergence-free property of $u_i^{\varepsilon}+F_i^{\varepsilon}$, we infer
\begin{align*}
    \frac{d}{dt}\int_{\mathbb{S}^2}\zeta_i^{\varepsilon}(t,\mathbf{x})\dd\boldsymbol{\sigma}(\mathbf{x})&=\int_{\mathbb{S}^2}\partial_{t}\zeta_i^{\varepsilon}(t,\mathbf{x})\dd\boldsymbol{\sigma}(\mathbf{x})\\
    &=-\int_{\mathbb{S}^2}\big(u_i^{\varepsilon}(t,\mathbf{x})+F_i^{\varepsilon}(t,\mathbf{x})\big)\cdot\nabla\zeta_i^{\varepsilon}(t,\mathbf{x})\dd\boldsymbol{\sigma}(\mathbf{x})\\
    &=\int_{\mathbb{S}^2}\nabla\cdot\big(u_i^{\varepsilon}(t,\mathbf{x})+F_i^{\varepsilon}(t,\mathbf{x})\big)\zeta_i^{\varepsilon}(t,\mathbf{x})\dd\boldsymbol{\sigma}(\mathbf{x})\\
    &=0.
\end{align*}
Therefore,
\begin{equation}\label{mass conserv}
    \forall t\geqslant0,\quad\int_{\mathbb{S}^2}\zeta_i^{\varepsilon}(t,\mathbf{x})\dd\boldsymbol{\sigma}(\mathbf{x})=\int_{\mathbb{S}^2}\zeta_i^{\varepsilon}(0,\mathbf{x})\dd\boldsymbol{\sigma}(\mathbf{x})=\int_{\mathbb{S}^2}\zeta_{i,0}^{\varepsilon}(\mathbf{x})\dd\boldsymbol{\sigma}(\mathbf{x})=\Gamma_i.
\end{equation}
Now, we define the center of vorticity (in $\R^3$) of $\zeta_i^\eps$ by 
    \begin{equation}\label{center of vort}
    c_i^{\varepsilon}(t)\triangleq\frac{1}{\Gamma_i}\int_{\mathbb{S}^2}\mathbf{x}\zeta_i^{\varepsilon}(t,\mathbf{x})\dd\boldsymbol{\sigma}(\mathbf{x})
    \end{equation}
    and its renormalized second moment
    \begin{equation}\label{def moment}
        I_i^{\varepsilon}(t)\triangleq\frac{1}{\Gamma_i}\int_{\mathbb{S}^2}|\mathbf{x}-c_i^{\varepsilon}(t)|_{\mathbb{R}^3}^2\zeta_i^{\varepsilon}(t,\mathbf{x})\dd\boldsymbol{\sigma}(\mathbf{x}).
    \end{equation}
    Please note that, unlike the planar case, with our definition the center of vorticity does not lie within $\mathbb{S}^2$, but instead within its convex envelope in $\R^3$ which is the unit ball. Then, we draw inspiration from \cite[page 19]{ISG99} in the planar case, by defining higher-order moments of $\zeta^\eps$, for every $n\in\mathbb{N}^*$, by
    \begin{equation}\label{def moment 4n}
        m_{n,i}^\eps(t) \triangleq \frac{1}{\Gamma_i} \int_{\mathbb{S}^2} |\mathbf{x}-c_i^\eps(t)|_{\R^3}^{4n} \zeta_i^\eps(t,\mathbf{x}) \dd \boldsymbol{\sigma}(\mathbf{x}).
    \end{equation}
    We start with the estimates at the initial time.
    \begin{lem}\label{lem:initestim}
       We have that
            $$|c_i^\eps(0) - \mathbf{x}_i^0|_{\mathbb{R}^3} \leqslant \eps, \qquad   I_i^\eps(0) \leqslant 4 \eps^2, \qquad m_{n,i}^\eps(0) \leqslant 16^n \eps^{4n}. $$ 
    \end{lem}
    \begin{proof}
        
    In view of Hypothesis \ref{hyp:zeta}, we have
    $$\supp\big(\zeta_i^{\varepsilon}(0,\cdot)\big)\subset B(\mathbf{x}_i^0,\eps).$$
    Since $B(\mathbf{x}_i^0,\eps)$ is a convex set in $\mathbb{R}^3$, then by definition of the convex hull, we obtain
    \begin{equation*}
        \textnormal{Conv}_{\mathbb{R}^3}\Big(\supp\big(\zeta_i^{\varepsilon}(0,\cdot)\big)\Big)\subset B(\mathbf{x}_i^0,\eps).
    \end{equation*}
    Besides, by Hypothesis \ref{hyp:zeta} and \eqref{center of vort}, the initial center of mass is a barycenter with positive coefficients of elements in $\supp\big(\zeta_i^{\varepsilon}(0,\cdot)\big).$ Consequently,
    \begin{equation*}
        c_i^{\varepsilon}(0)\in\textnormal{Conv}_{\mathbb{R}^3}\Big(\supp\big(\zeta_i^{\varepsilon}(0,\cdot)\big)\Big),
    \end{equation*}
    and thus
    \begin{equation}\label{init:ecenter}
        \left|c_i^{\varepsilon}(0)-\mathbf{x}_i^0\right|\leqslant\varepsilon.
    \end{equation}
    This is the first desired estimate. For the second one, we observe that \eqref{init:ecenter} implies, by means of triangular inequality, that
    \begin{equation}\label{init:distc}
        \max_{\mathbf{x}\in\supp\big(\zeta_i^{\varepsilon}(0,\cdot)\big)}|\mathbf{x}-c_i^{\varepsilon}(0)|_{\R^3}\leqslant |\mathbf{x}_i^0-c_i^{\varepsilon}(0)|_{\R^3}+\max_{\mathbf{x}\in\supp\big(\zeta_i^{\varepsilon}(0,\cdot)\big)}|\mathbf{x}-\mathbf{x}_i^0|_{\R^3}\leqslant2\varepsilon.
    \end{equation}
    Plugging the estimate \eqref{init:distc} into \eqref{def moment} and using \eqref{mass conserv} yields
    $$I_i^\eps(0) \leqslant 4 \eps^2.$$ 
    Plugging the estimate \eqref{init:distc} into \eqref{def moment 4n} and using \eqref{mass conserv} yields
    $$m_{n,i}^\eps(0) \leqslant 16^n \eps^{4n}.$$ 
    This concludes the proof of Lemma \ref{lem:initestim}.
    \end{proof}

    We now estimate the growth of the vorticity moments.
    \begin{lem}\label{lem:control_I}
        For every $t\leqslant  \tau_{\eps,\beta}$ and every $i \in \{1,\ldots,N\}$,
        \begin{equation*}
            \left| \der{}{t} I_i^{\varepsilon}(t)\right|\leqslant 2D_\eps I_i^{\varepsilon}(t).
        \end{equation*}
    \end{lem}
    \begin{proof}
    Differentiating in time the moment \eqref{def moment} leads to
    \begin{align*}
        \der{}{t} I_i^{\varepsilon}(t)&=\frac{2}{\Gamma_i}\der{}{t} c_i^{\varepsilon}(t)\cdot\int_{\mathbb{S}^2}\big(\mathbf{x}-c_i^{\varepsilon}(t)\big)\zeta_i^{\varepsilon}(t,\mathbf{x})\dd\boldsymbol{\sigma}(\mathbf{x})+\frac{1}{\Gamma_i}\int_{\mathbb{S}^2}|\mathbf{x}-c_i^{\varepsilon}(t)|_{\mathbb{R}^3}^2\partial_{t}\zeta_i^{\varepsilon}(t,\mathbf{x})\dd\boldsymbol{\sigma}(\mathbf{x})\\
        &\triangleq J_1^{\varepsilon,i}(t)+J_{2}^{\varepsilon,i}(t).
    \end{align*}
    First observe that from \eqref{mass conserv}, we obtain
    \begin{equation}\label{center cancel}
        \begin{aligned}
        \int_{\mathbb{S}^2}\big(\mathbf{x}-c_i^{\varepsilon}(t)\big)\zeta_i^{\varepsilon}(t,\mathbf{x})\dd\boldsymbol{\sigma}(\mathbf{x})&=\int_{\mathbb{S}^2}\mathbf{x}\zeta_i^{\varepsilon}(t,\mathbf{x})\dd\boldsymbol{\sigma}(\mathbf{x})-c_i^{\varepsilon}(t)\int_{\mathbb{S}^2}\zeta_i^{\varepsilon}(t,\mathbf{x})\dd\boldsymbol{\sigma}(\mathbf{x})\\
        &=\Gamma_i c_i^{\varepsilon}(t)-\Gamma_i c_i^{\varepsilon}(t)\\
        &=0.
    \end{aligned}
    \end{equation}
    Therefore, $J_1^{\varepsilon,i}(t)=0.$ Now, we turn to the estimation of $J_2^{\varepsilon,i}(t).$ From \eqref{pb:zetai}, we can write
    \begin{align*}
        J_2^{\varepsilon,i}(t)&=-\frac{1}{\Gamma_i}\int_{\mathbb{S}^2}|\mathbf{x}-c_i^{\varepsilon}(t)|_{\mathbb{R}^3}^2\big(u_i^{\varepsilon}(t,\mathbf{x})+F_i^{\varepsilon}(t,\mathbf{x})\big)\cdot\nabla\zeta_i^{\varepsilon}(t,\mathbf{x})\dd\boldsymbol{\sigma}(\mathbf{x}).
    \end{align*}
    Using the divergence Theorem together with the divergence-free property of the vector field $u_i^{\varepsilon}+F_i^{\varepsilon}$, we infer
    \begin{align*}
        J_2^{\varepsilon,i}(t)&=\frac{2}{\Gamma_i}\int_{\mathbb{S}^2}\big(\mathbf{x}-c_i^{\varepsilon}(t)\big)\cdot\big(u_i^{\varepsilon}(t,\mathbf{x})+F_i^{\varepsilon}(t,\mathbf{x})\big)\zeta_i^{\varepsilon}(t,\mathbf{x})\dd\boldsymbol{\sigma}(\mathbf{x})\\
        &=\frac{2}{\Gamma_i}\int_{\mathbb{S}^2}\big(\mathbf{x}-c_i^{\varepsilon}(t)\big)\cdot u_i^{\varepsilon}(t,\mathbf{x})\zeta_i^{\varepsilon}(t,\mathbf{x})\dd\boldsymbol{\sigma}(\mathbf{x})+\frac{2}{\Gamma_i}\int_{\mathbb{S}^2}\big(\mathbf{x}-c_i^{\varepsilon}(t)\big)\cdot F_i^{\varepsilon}(t,\mathbf{x})\zeta_i^{\varepsilon}(t,\mathbf{x})\dd\boldsymbol{\sigma}(\mathbf{x}).
    \end{align*}
    Inserting the expression of $u_i^{\varepsilon}$ in \eqref{pb:zetai} into the first term of the right hand-side above gives
    \begin{align*}
        \int_{\mathbb{S}^2}\big(\mathbf{x}-c_i^{\varepsilon}(t)\big)\cdot u_i^{\varepsilon}(t,\mathbf{x})\zeta_i^{\varepsilon}(t,\mathbf{x})\dd\boldsymbol{\sigma}(\mathbf{x})&=\frac{1}{2\pi}\int_{\mathbb{S}^2}\int_{\mathbb{S}^2}\frac{\big(\mathbf{x}-c_i^{\varepsilon}(t)\big)\cdot(\mathbf{x}\wedge\mathbf{y})}{|\mathbf{x}-\mathbf{y}|_{\mathbb{R}^3}^2}\zeta_i^{\varepsilon}(t,\mathbf{x})\dd\boldsymbol{\sigma}(\mathbf{y})\dd\boldsymbol{\sigma}(\mathbf{x})\\
        &=-\frac{1}{2\pi}c_i^{\varepsilon}(t)\cdot\int_{\mathbb{S}^2}\int_{\mathbb{S}^2}\frac{\mathbf{x}\wedge\mathbf{y}}{|\mathbf{x}-\mathbf{y}|_{\mathbb{R}^3}^2}\zeta_i^{\varepsilon}(t,\mathbf{y})\zeta_i^{\varepsilon}(t,\mathbf{x})\dd\boldsymbol{\sigma}(\mathbf{y})\dd\boldsymbol{\sigma}(\mathbf{x}).
    \end{align*}
    We have used the fact that the vectors $\mathbf{x}$ and $\mathbf{x}\wedge\mathbf{y}$ are orthogonal. Now, by a anti-symmetry of the role of $\mathbf{x}$ and $\mathbf{y}$, we find
    \begin{equation}\label{simpli1}
    \int_{\mathbb{S}^2}\int_{\mathbb{S}^2}\frac{\mathbf{x}\wedge\mathbf{y}}{|\mathbf{x}-\mathbf{y}|_{\mathbb{R}^3}^2}\zeta_i^{\varepsilon}(t,\mathbf{x})\zeta_i^{\varepsilon}(t,\mathbf{y})\dd\boldsymbol{\sigma}(\mathbf{x})\dd\boldsymbol{\sigma}(\mathbf{y})=0.
    \end{equation}
    Hence,
    \begin{align*}
        J_2^{\varepsilon,i}(t)&=\frac{2}{\Gamma_i}\int_{\mathbb{S}^2}\big(\mathbf{x}-c_i^{\varepsilon}(t)\big)\cdot F_i^{\varepsilon}(t,\mathbf{x})\zeta_i^{\varepsilon}(t,\mathbf{x})\dd\boldsymbol{\sigma}(\mathbf{x}).
    \end{align*}
    Using \eqref{center cancel}, we add a vanishing term in this expression to obtain that
    \begin{align*}
    J_2^{\varepsilon,i}(t)&=\frac{2}{\Gamma_i}\int_{\mathbb{S}^2}\big(\mathbf{x}-c_i^{\varepsilon}(t)\big)\cdot\Big(F_i^{\varepsilon}(t,\mathbf{x})-F_i^{\varepsilon}\big(t,c_i^{\varepsilon}(t)\big) \Big)\zeta_i^{\varepsilon}(t,\mathbf{x})\dd\boldsymbol{\sigma}(\mathbf{x}),
    \end{align*}
    where $F_i^{\varepsilon}\big(t,c_i^\eps(t)\big)$ must be understood in the sense of the extension \eqref{extension:Fi}. Since the blob is of constant sign, by using Cauchy-Schwarz inequality together with Definition~\ref{def:Dprime}, more precisely of equation~\eqref{eq:Flip1}, we get
    \begin{align*}
        |J_2^{\varepsilon,i}(t)|&\leqslant 2 D_\eps \int_{\mathbb{S}^2}|\mathbf{x}-c_i^{\varepsilon}(t)|^2_{\mathbb{R}^3}\frac{\zeta_i^{\varepsilon}(t,\mathbf{x})}{\Gamma_i}\dd\boldsymbol{\sigma}(\mathbf{x})\\
        &=2D_\eps I_i^\eps(t).
    \end{align*}
    Combining the foregoing calculations gives
    \begin{align*}
        \left| \der{}{t} I_i^{\varepsilon}(t)\right|\leqslant 2D_\eps I_i^{\varepsilon}(t).
    \end{align*}
    This achieves the proof of Lemma \ref{lem:control_I}.
\end{proof}
We now turn to evolution of the center of vorticity. For this, we introduce the notations
\begin{equation}\label{def:vecC}
    \Vec{c}^{\,\varepsilon}(t)\triangleq\big(c_1^{\varepsilon}(t),\ldots,c_{N}^{\varepsilon}(t)\big)
\end{equation}
and the point vortex vector field $\mathcal{F}$ given by
\begin{equation}\label{def:PVvectorfield}
    \mathcal{F}\triangleq(\mathcal{F}_1,\ldots,\mathcal{F}_N),\qquad\mathcal{F}_i(\mathbf{x}_1,\ldots,\mathbf{x}_N)\triangleq \frac{1}{\Gamma_i} \nabla^\perp_{\mathbf{x}_i} \mathcal{H}(\mathbf{x}_1,\ldots,\mathbf{x}_N) = \frac{1}{2\pi}\sum_{j=1\atop j\neq i}^{N}\Gamma_j\frac{\mathbf{x}_i\wedge\mathbf{x}_j}{\,\,|\mathbf{x}_i-\mathbf{x}_j|_{\mathbb{R}^3}}+\gamma\mathbf{e}_3\wedge\mathbf{x}_i.
\end{equation}
A priori $\mathcal{F}_i:\mathcal{A}_N\to T\mathbb{S}^2$ is well-defined on the admissible set $\mathcal{A}_N$ introduced in \eqref{def:admissible-set}. However, for later purposes, we may rather consider $\mathcal{F}_i$ as in \eqref{def:PVvectorfield} but defined on the set
$$\mathscr{D}_N\triangleq\prod_{k=1}^{N}\bigcup_{t\in[0,\tau_{\varepsilon,\beta}]}\textnormal{Conv}_{\mathbb{R}^3}\Big(\supp\big(\zeta_k^{\varepsilon}(t,\cdot)\big)\Big).$$
By continuity of the trajectories, the set $\mathscr{D}_N$ is closed in $\mathbb{R}^3$. By construction, $\mathcal{F}_i$ is Lipschitz (actually smooth) on $\mathscr{D}_N.$

\begin{lem}\label{lem:c}
    There exists a constant $C$ such that for every $t \leqslant \tau_{\eps,\beta}$ and every $i\in\{1,\ldots,N\}$, we have
    \begin{equation*}
        \left|\der{}{t}c_i^\eps(t) - \mathcal{F}_i\big(\vec{c}^{\,\eps}(t)\big)\right|_{\mathbb{R}^3} \leqslant C \sum_{j=1}^N \sqrt{I_j^\eps(t)},
    \end{equation*}
    where $\vec{c}^{\,\varepsilon}(t)$ and $\mathcal{F}_i$ have been introduced in \eqref{def:vecC} and \eqref{def:PVvectorfield}, respectively.
\end{lem}
\begin{proof}
    Differentiating in time $c_i^{\varepsilon}$, we obtain from \eqref{center of vort} and \eqref{pb:zetai} that
    \begin{align*}
        \der{}{t} c_i^\eps(t)&=\frac{1}{\Gamma_i}\int_{\mathbb{S}^2}\mathbf{x}\partial_{t}\zeta_i^{\varepsilon}(t,\mathbf{x})\dd\boldsymbol{\sigma}(\mathbf{x})\\
        &=-\frac{1}{\Gamma_i}\int_{\mathbb{S}^2}\mathbf{x}\big(u_i^{\varepsilon}(t,\mathbf{x})+F_i^{\varepsilon}(t,\mathbf{x})\big)\cdot\nabla\zeta_i^{\varepsilon}(t,\mathbf{x})\dd\boldsymbol{\sigma}(\mathbf{x}).
    \end{align*}
    Applying, once again the divergence Theorem, we get
\begin{align*}
    \der{}{t} c_i^\eps(t)&=\frac{1}{\Gamma_i}\int_{\mathbb{S}^2}\big(u_i^{\varepsilon}(t,\mathbf{x})+F_i^{\varepsilon}(t,\mathbf{x})\big)\zeta_i^{\varepsilon}(t,\mathbf{x})\dd\boldsymbol{\sigma}(\mathbf{x})\\
    &=\frac{1}{\Gamma_i}\int_{\mathbb{S}^2}F_i^{\varepsilon}(t,\mathbf{x})\zeta_i^{\varepsilon}(t,\mathbf{x})\dd\boldsymbol{\sigma}(\mathbf{x})+\frac{1}{2\pi\Gamma_i}\int_{\mathbb{S}^2}\int_{\mathbb{S}^2}\frac{\mathbf{x}\wedge\mathbf{y}}{|\mathbf{x}-\mathbf{y}|_{\mathbb{R}^3}^2}\zeta_i^{\varepsilon}(t,\mathbf{x})\zeta_i^{\varepsilon}(t,\mathbf{y})\dd\boldsymbol{\sigma}(\mathbf{x})\dd\boldsymbol{\sigma}(\mathbf{y}).
\end{align*}
    From \eqref{simpli1}, we infer
    \begin{equation}\label{derivee de c}
        \der{}{t} c_i^\eps(t) = \frac{1}{\Gamma_i}\int_{\mathbb{S}^2} F_i^\eps(t,\mathbf{x}) \zeta_i^\eps(t,\mathbf{x})\dd\boldsymbol{\sigma}(\mathbf{x}).
    \end{equation}
    Therefore, by using \eqref{mass conserv}, the bound~\eqref{maj:F_Lip} and the Cauchy-Schwarz inequality, we obtain
    \begin{align*}
        \left|\der{}{t} c_i^\eps(t) - \mathcal{F}_i\big(t,\vec{c}^{\,\varepsilon}(t)\big)\right|_{\mathbb{R}^3} & = \left| \frac{1}{\Gamma_i}\int_{\mathbb{S}^2} F_i^\eps(t,\mathbf{x}) \zeta_i^\eps(t,\mathbf{x})\dd\boldsymbol{\sigma}(\mathbf{x}) - \mathcal{F}_i\big(t,\vec{c}^{\,\varepsilon}(t)\big) \right|_{\mathbb{R}^3}  \\
        & = \left| \frac{1}{\Gamma_i}\int_{\mathbb{S}^2} \Big(F_i^\eps(t,\mathbf{x}) -F_i^\eps\big(t,c_i^\eps(t)\big)\Big)\zeta_i^\eps(t,\mathbf{x})\dd\boldsymbol{\sigma}(\mathbf{x}) - \big(\mathcal{F}_i\big(t,\vec{c}^{\,\varepsilon}(t)\big)-F_i^\eps\big(t,c_i^\eps(t)\big)\Big) \right|_{\mathbb{R}^3} \\
        & \leqslant D \sqrt{I_i^\eps(t)} + \big|\mathcal{F}_i\big(t,\vec{c}^{\,\varepsilon}(t)\big)-F_i^\eps\big(t,c_i^\eps(t)\big)\big|_{\mathbb{R}^3}.
    \end{align*}
    Using the definitions \eqref{BS-kernel}, \eqref{def:Fi}, \eqref{def:vecC} and \eqref{def:PVvectorfield}, we can write
    \begin{align*}
        \mathcal{F}_i\big(t,\vec{c}^{\,\varepsilon}(t)\big)-F_i^{\varepsilon}\big(t,c_i^{\varepsilon}(t)\big) & = \sum_{j=1\atop j\neq i}^{N}\frac{\Gamma_j}{2\pi}K_{\mathbb{S}^2}\big(c_i^\eps(t),c_j^\eps(t)\big)+\gamma\,\mathbf{e}_3\wedge c_i^\eps(t)\\ 
        & \qquad - \frac{1}{2\pi}\sum_{j=1\atop j\neq i}^{N}\int_{\mathbb{S}^2}K_{\mathbb{S}^2}\big(c_i^{\varepsilon}(t),\mathbf{y}\big)\zeta_{j}^{\varepsilon}(t,\mathbf{y})\dd\boldsymbol{\sigma}(\mathbf{y})-\gamma \mathbf{e}_3\wedge c_i^{\varepsilon}(t)\\
        & = \frac{1}{2\pi}\sum_{j=1\atop j\neq i}^{N} \int_{\mathbb{S}^2} \Big( K_{\mathbb{S}^2}\big(c_i^\eps(t),c_j^\eps(t)\big) - K_{\mathbb{S}^2}\big(c_i^{\varepsilon}(t),\mathbf{y}\big)\Big)\zeta_{j}^{\varepsilon}(t,\mathbf{y})\dd\boldsymbol{\sigma}(\mathbf{y}).
    \end{align*}
    Thanks to \eqref{distance-supports-timet}, for $t\leqslant\tau_{\varepsilon,\beta},$ the function $\mathbf{y}\mapsto K\big(c_i^{\varepsilon}(t),\mathbf{y}\big)$ is of class $C^1$ on $\displaystyle\bigcup_{j=1\atop j\neq i}^{N}\textnormal{Conv}\Big(\supp\big(\zeta_j^{\varepsilon}(t,\cdot)\big)\Big)$. Hence, by mean value Theorem, we have for any $j\in\{1,\ldots,N\}\setminus\{i\}$ and any $\mathbf{y}\in\supp\big(\zeta_j^{\varepsilon}(t,\cdot)\big),$
    \begin{equation*}
        \left|K\big(c_i^\eps(t),c_j^\eps(t)\big) - K\big(c_i^{\varepsilon}(t),\mathbf{y}\big)\right|_{\mathbb{R}^3} \leqslant C\left|c_j^{\varepsilon}(t)-\mathbf{y}\right|_{\mathbb{R}^3}.
    \end{equation*}
    Thus, by means of triangle and Cauchy-Schwarz inequalities, we find
    \begin{align*}
        \left|\mathcal{F}_i\big(t,\vec{c}^{\,\varepsilon}(t)\big)-F_i^{\varepsilon}\big(t,c_i^{\varepsilon}(t)\big)\right|_{\mathbb{R}^3} & \leqslant \frac{C}{2\pi} \sum_{j=1\atop j\neq i}^{N}|\Gamma_j| \int_{\mathbb{S}^2} |c_j^\eps(t)-\mathbf{y}|_{\mathbb{R}^3} \sqrt{\frac{\zeta_{j}^{\varepsilon}(t,\mathbf{y})}{\Gamma_j}}\sqrt{\frac{\zeta_{j}^{\varepsilon}(t,\mathbf{y})}{\Gamma_j}}\dd\boldsymbol{\sigma}(\mathbf{y})\\
        & \leqslant C  \sum_{j=1\atop j\neq i}^{N} \sqrt{I_j^\eps(t)}.
    \end{align*}
    This ends the proof of Lemma \ref{lem:c}.
\end{proof}

\begin{lem}\label{lem:4n}
    For every $\eps$ small enough and for every $t \leqslant \tau_{\eps,\beta}$,
    \begin{equation*}
        \der{}{t} m_{n,i}^\eps(t) \leqslant \big(m_{n,i}^{\eps}(t)\big)^\frac{n-1}{n} \left(\frac{70n^2|\Gamma_i|}{\pi}I_i^\eps(t) + 4n D_\eps \left( \big(m_{n,i}^{\eps}(t)\big)^\frac{1}{n} + 2\sqrt{I_i^\eps(t)} \eps^{3\beta} \right) \right).
    \end{equation*}
\end{lem}
\begin{proof}
    Throughout the proof, we fix $n\in\mathbb{N}^*$. Differentiating in time \eqref{def moment 4n} and using \eqref{pb:zetai} leads to
    \begin{align*}
        \der{}{t} m_{n,i}^\eps(t) & = \frac{4n}{\Gamma_i} \int_{\mathbb{S}^2}|\mathbf{x}-c_i^\eps(t)|_{\R^3}^{4n-1} \left(\der{}{t}|\mathbf{x}-c_i^\eps(t)|\right)\zeta_i^\eps(t,\mathbf{x})\dd \boldsymbol{\sigma}(\mathbf{x})\\
        & \quad+\frac{1}{\Gamma_i}\int_{\mathbb{S}^2}|\mathbf{x}-c_i^\eps(t)|_{\R^3}^{4n}\partial_{t}\zeta_i^\eps(t,\mathbf{x})\dd \boldsymbol{\sigma}(\mathbf{x})\\
        &= \frac{4n}{\Gamma_i} \int_{\mathbb{S}^2}|\mathbf{x}-c_i^\eps(t)|_{\R^3}^{4n-1} \left(\der{}{t}|\mathbf{x}-c_i^\eps(t)|\right)\zeta_i^\eps(t,\mathbf{x})\dd \boldsymbol{\sigma}(\mathbf{x})\\
        & \quad -\frac{1}{\Gamma_i}\int_{\mathbb{S}^2}|\mathbf{x}-c_i^\eps(t)|_{\R^3}^{4n}\big(u_i^{\varepsilon}(t,\mathbf{x})+F_i^{\varepsilon}(t,\mathbf{x})\big)\cdot\nabla\zeta_i^\eps(t,\mathbf{x})\dd \boldsymbol{\sigma}(\mathbf{x}).
    \end{align*}
    Applying the divergence theorem and using that $u_i^\varepsilon,F_i^{\varepsilon}$ are solenoidal yields
    \begin{align*}
        \der{}{t} m_{n,i}^\eps(t) & = -\frac{4n}{\Gamma_i} \int_{\mathbb{S}^2} \der{}{t}c_i^\eps(t) \cdot \big(\mathbf{x}-c_i^\eps(t)\big) |\mathbf{x}-c_i^\eps(t)|_{\R^3}^{4n-2} \zeta_i^\eps(t,\mathbf{x})\dd \boldsymbol{\sigma}(\mathbf{x})\\
        & \quad +\frac{4n}{\Gamma_i} \int_{\mathbb{S}^2}\big(\mathbf{x} - c_i^\eps(t)\big) \cdot \big(u_i^\eps(t,\mathbf{x}) + F_i^\eps(t,\mathbf{x})\big) |\mathbf{x}-c_i^\eps(t)|_{\R^3}^{4n-2} \zeta_i^\eps(t,\mathbf{x}) \dd \boldsymbol{\sigma}(\mathbf{x}).
    \end{align*}
    Using \eqref{derivee de c} and \eqref{pb:zetai}, we can rearrange the previous expression as
    \begin{align*}
        \der{}{t} m_{n,i}^\eps(t) & = \frac{4n}{2\pi\Gamma_i} \iint_{\mathbb{S}^2\times\mathbb{S}^2}(\mathbf{x} - c_i^\eps(t)) \cdot \frac{\mathbf{x}\wedge \mathbf{y}}{|\mathbf{x}-\mathbf{y}|_{\R^3}^2}|\mathbf{x}-c_i^\eps(t)|_{\R^3}^{4n-2} \zeta_i^\eps(t,\mathbf{y})\zeta_i^\eps(t,\mathbf{x}) \dd \boldsymbol{\sigma}(\mathbf{y})\dd \boldsymbol{\sigma}(\mathbf{x}) \\
        &\quad + \frac{4n}{\Gamma_i}\int_{\mathbb{S}^2} \big(\mathbf{x} - c_i^\eps(t)\big) \cdot \left(F_i^\eps(t,\mathbf{x}) - \frac{1}{\Gamma_i}\int_{\mathbb{S}^2} F_i^\eps(t,\mathbf{y})\zeta_i^\eps(t,\mathbf{y})\dd \boldsymbol{\sigma}(\mathbf{y})\right)|\mathbf{x}-c_i^\eps(t)|_{\R^3}^{4n-2} \zeta_i^\eps(t,\mathbf{x}) \dd \boldsymbol{\sigma}(\mathbf{x}) \\
        &\triangleq A_1^{\eps}(t) + A_2^{\eps}(t).
    \end{align*}
    We first deal with the term $A_1^{\eps}(t).$ Notice that by definition \eqref{center of vort} and orthogonality argument, we have for every $\mathbf{x}\in\mathbb{S}^2$,
    \begin{equation*}
        \int_{\mathbb{S}^2} (\mathbf{x} - c_i^\eps(t)) \cdot \frac{\mathbf{x}\wedge \mathbf{y}}{|\mathbf{x}-c_i^\eps(t)|_{\R^3}^2} \zeta_i^\eps(t,\mathbf{y})\dd \boldsymbol{\sigma}(\mathbf{y}) = \Gamma(\mathbf{x} - c_i^\eps(t)) \cdot \frac{\mathbf{x}\wedge c_i^\eps(t)}{|\mathbf{x}-c_i^\eps(t)|_{\R^3}^2} = 0.
    \end{equation*} 
    Therefore,
    \begin{equation*}
        A_1^{\eps}(t) = \frac{2n}{\pi\Gamma_i} \iint_{\mathbb{S}^2\times\mathbb{S}^2}\mathcal{K}^{\eps}(t,\mathbf{x},\mathbf{y})|\mathbf{x}-c_i^\eps(t)|_{\R^3}^{4n-2} \zeta_i^\eps(t,\mathbf{y})\zeta_i^\eps(t,\mathbf{x}) \dd \boldsymbol{\sigma}(\mathbf{y})\dd \boldsymbol{\sigma}(\mathbf{x}),
    \end{equation*}
    with
    \begin{equation*}
        \mathcal{K}^{\eps}(t,\mathbf{x},\mathbf{y}) \triangleq \big(\mathbf{x} - c_i^\eps(t)\big) \cdot (\mathbf{x}\wedge \mathbf{y}) \left( \frac{1}{|\mathbf{x}-\mathbf{y}|_{\R^3}^2} - \frac{1}{|\mathbf{x}-c_i^\eps(t)|^2_{\R^3}}\right).
    \end{equation*}
    We now split the domain of integration $\mathbb{S}^2 \times \mathbb{S}^2$ into three subdomains
    \begin{align*}
        E_1^{\eps}(t) & \triangleq \left\{ (\mathbf{x},\mathbf{y}) \in \mathbb{S}^2 \times \mathbb{S}^2 \quad\textnormal{s.t.}\quad |\mathbf{y}-c_i^\eps(t)|_{\mathbb{R}^3} \leqslant \left(1-\frac{1}{2n} \right) |\mathbf{x}-c_i^\eps(t)|_{\mathbb{R}^3} \right\}, \\
        E_2^{\eps}(t) & \triangleq \left\{ (\mathbf{x},\mathbf{y}) \in \mathbb{S}^2 \times \mathbb{S}^2 \quad\textnormal{s.t.}\quad \left(1-\frac{1}{2n}\right)|\mathbf{x}-c_i^\eps(t)|_{\mathbb{R}^3} < |\mathbf{y}-c_i^\eps(t)|_{\mathbb{R}^3} < \left(1-\frac{1}{2n}\right)^{-1} |\mathbf{x}-c_i^\eps(t)|_{\mathbb{R}^3} \right\}, \\
        E_3^{\eps}(t) & \triangleq \left\{ (\mathbf{x},\mathbf{y}) \in \mathbb{S}^2 \times \mathbb{S}^2 \quad\textnormal{s.t.}\quad|\mathbf{x}-c_i^\eps(t)|_{\mathbb{R}^3} \leqslant \left(1-\frac{1}{2n} \right) |\mathbf{y}-c_i^\eps(t)|_{\mathbb{R}^3} \right\}.
    \end{align*}
    This leads to write
    \begin{equation*}
        A_1^{\eps} = A_{1,1}^{\eps} + A_{1,2}^{\eps} + A_{1,3}^{\eps},
    \end{equation*}
    where for any $k\in\{1,2,3\},$
    \begin{equation*}
        A_{1,k}^{\eps}(t) = \frac{2n}{\pi\Gamma_i} \iint_{E_k^{\eps}(t)}\mathcal{K}(t,\mathbf{x},\mathbf{y})|\mathbf{x}-c_i^\eps(t)|_{\R^3}^{4n-2} \zeta_i^\eps(t,\mathbf{y})\zeta_i^\eps(t,\mathbf{x}) \dd \boldsymbol{\sigma}(\mathbf{y})\dd \boldsymbol{\sigma}(\mathbf{x}).
    \end{equation*}
    We now work on the expression of $\mathcal{K}^{\eps}$.
    We get the identity
    \begin{equation}\label{useful-identity00}
        \frac{1}{|\mathbf{x}-\mathbf{y}|_{\R^3}^2} - \frac{1}{|\mathbf{x}-c_i^\eps(t)|^2_{\R^3}}= \frac{\Big(2\big(\mathbf{x}-c_i^\eps(t)\big) - \big(\mathbf{y} - c_i^\eps(t)\big)\Big) \cdot \big(\mathbf{y}-c_i^\eps(t)\big) }{|\mathbf{x}-c_i^\eps(t)|^2_{\R^3}|\mathbf{x}-\mathbf{y}|^2_{\R^3}},
    \end{equation}
    and since $\big(\mathbf{x}-c_i^\eps(t)\big) \cdot (\mathbf{x}\wedge \mathbf{y}) = \big(\mathbf{y}-c_i^\eps(t)\big) \cdot (\mathbf{x}\wedge \mathbf{y})$, we conclude that
    \begin{equation}\label{new-calK}
        \mathcal{K}^{\eps}(t,\mathbf{x},\mathbf{y}) = \big(\mathbf{y}-c_i^\eps(t)\big) \cdot (\mathbf{x}\wedge \mathbf{y})\frac{\Big(2\big(\mathbf{x}-c_i^\eps(t)\big) - \big(\mathbf{y} - c_i^\eps(t)\big)\Big) \cdot \big(\mathbf{y}-c_i^\eps(t)\big) }{|\mathbf{x}-c_i^\eps(t)|^2_{\R^3}|\mathbf{x}-\mathbf{y}|^2_{\R^3}}\cdot
    \end{equation}
    We now focus on $A_{1,1}^{\eps}(t)$. Take $(\mathbf{x},\mathbf{y}) \in E_1^{\eps}(t).$ By definition,
    \begin{equation}\label{estimE1}
        |\mathbf{y}-c_i^\eps(t)|_{\mathbb{R}^3} \leqslant \left(1-\frac{1}{2n} \right) |\mathbf{x}-c_i^\eps(t)|_{\mathbb{R}^3}.
    \end{equation}
    Hence, by left-triangular inequality, we deduce that
    \begin{equation}\label{minE1}
        \begin{aligned}
            |\mathbf{x}-\mathbf{y}|_{\R^3}&\geqslant|\mathbf{x}-c_i^\eps(t)|_{\R^3}-|\mathbf{y}-c_i^\eps(t)|_{\R^3}\\
        &\geqslant\frac{1}{2n}|\mathbf{x}-c_i^\eps(t)|_{\R^3}.
        \end{aligned}
    \end{equation}
    Also, by right-triangular inequality and \eqref{estimE1}, we have
    \begin{equation}\label{trieqE1}
        \begin{aligned}
        |2(\mathbf{x} - c_i^\eps(t)) - (\mathbf{y}-c_i^\eps(t))|_{\R^3} &\leqslant2|\mathbf{x} - c_i^\eps(t)|_{\R^3}+|\mathbf{y}-c_i^\eps(t)|_{\R^3}\\
        &\leqslant3 |\mathbf{x}-c_i^\eps(t)|_{\R^3}.
    \end{aligned}
    \end{equation}
    Putting together \eqref{new-calK}, \eqref{eq:maj_x_wedge_y}, \eqref{minE1} and \eqref{trieqE1} gives, by using Cauchy-Schwarz inequality,
    \begin{align*}
        \left|\mathcal{K}^{\eps}(t,\mathbf{x},\mathbf{y})\right|
        &\leqslant|\mathbf{y}-c_i^\eps(t)|_{\R^3}|\mathbf{x}-\mathbf{y}|_{\R^3}\frac{3|\mathbf{x}-c_i^\eps(t)|_{\R^3}|\mathbf{y}-c_i^\eps(t)|_{\R^3}}{|\mathbf{x}-c_i^\eps(t)|_{\R^3}^2|\mathbf{x}-\mathbf{y}|_{\R^3}^2}\\
        &\leqslant\frac{3|\mathbf{y}-c_i^\eps(t)|^2_{\R^3}}{|\mathbf{x} - c_i^\eps(t)|_{\R^3}|\mathbf{x}-\mathbf{y}|_{\R^3}}\\
        &\leqslant\frac{6n|\mathbf{y}-c_i^\eps(t)|^2_{\R^3}}{|\mathbf{x}-c_i^\eps(t)|^2_{\R^3}}\cdot
    \end{align*}
    Consequently, we obtain
    \begin{equation}\label{e-A11}
        \begin{aligned}
        |A_{1,1}^{\eps}(t)| &\leqslant \frac{12 n^2|\Gamma_i|}{\pi} \iint_{E_1^{\eps}(t)} |\mathbf{y}-c_i^\eps(t)|^2_{\R^3}|\mathbf{x}-c_i^\eps(t)|_{\R^3}^{4n-4} \frac{\zeta_i^\eps(t,\mathbf{y})}{\Gamma_i}\frac{\zeta_i^\eps(t,\mathbf{x})}{\Gamma_i} \dd \boldsymbol{\sigma}(\mathbf{y})\dd \boldsymbol{\sigma}(\mathbf{x})\\
        &\leqslant\frac{12 n^2|\Gamma_i|}{\pi} I_i^\eps(t) m_{n-1,i}^{\eps}(t).
    \end{aligned}
    \end{equation}
    Now, we deal with the term $A_{1,3}^{\eps}(t).$ Take $(\mathbf{x},\mathbf{y})\in E_3^{\eps}(t).$ By definition, 
    \begin{equation}\label{estimE3}
        |\mathbf{x}-c_i^\eps(t)|_{\mathbb{R}^3} \leqslant \left(1-\frac{1}{2n} \right) |\mathbf{y}-c_i^\eps(t)|_{\mathbb{R}^3}.
    \end{equation}
    Compared to \eqref{estimE1}, in \eqref{estimE3} we have exchanged the roles of $\mathbf{x}$ and $\mathbf{y}$, hence the bound \eqref{minE1} becomes
    \begin{equation}\label{minE3}
        |\mathbf{x}-\mathbf{y}|_{\mathbb{R}^3}\geqslant\frac{1}{2n}|\mathbf{y}-c_i^\eps(t)|_{\mathbb{R}^3},
    \end{equation}
    while \eqref{trieqE1} becomes
    \begin{equation}\label{trieqE3}
        |2(\mathbf{x} - c_i^\eps(t)) - (\mathbf{y}-c_i^\eps(t))|_{\R^3}\leqslant3 |\mathbf{y}-c_i^\eps(t)|_{\R^3}.
    \end{equation}
    Putting together \eqref{new-calK}, \eqref{eq:maj_x_wedge_y}, \eqref{minE3} and \eqref{trieqE3} gives, by Cauchy-Schwarz inequality,
    \begin{align*}
        \left|\mathcal{K}^{\eps}(t,\mathbf{x},\mathbf{y})\right|&\leqslant|\mathbf{y}-c_i^\eps(t)|_{\R^3}|\mathbf{x}-\mathbf{y}|_{\R^3}\frac{3|\mathbf{y}-c_i^\eps(t)|_{\R^3}^2}{|\mathbf{x}-c_i^\eps(t)|_{\R^3}^2|\mathbf{x}-\mathbf{y}|_{\R^3}^2}\\
        &\leqslant\frac{3|\mathbf{y}-c_i^\eps(t)|^3_{\R^3}}{|\mathbf{x} - c_i^\eps(t)|_{\R^3}^2|\mathbf{x}-\mathbf{y}|_{\R^3}}\\
        &\leqslant\frac{6n|\mathbf{y}-c_i^\eps(t)|^2_{\R^3}}{|\mathbf{x}-c_i^\eps(t)|^2_{\R^3}}\cdot
    \end{align*}
    Thus, we obtain that
    \begin{equation}\label{e-A13}
        \begin{aligned}
        |A_{1,3}^{\eps}(t)|&\leqslant \frac{12n^2|\Gamma_i|}{\pi} \iint_{E_3^{\eps}(t)} |\mathbf{y}-c_i^\eps(t)|^2_{\R^3}|\mathbf{x}-c_i^\eps(t)|_{\R^3}^{4n-4} \frac{\zeta_i^\eps(t,\mathbf{y})}{\Gamma_i}\frac{\zeta_i^\eps(t,\mathbf{x})}{\Gamma_i} \dd \boldsymbol{\sigma}(\mathbf{y})\dd \boldsymbol{\sigma}(\mathbf{x})\\
        &\leqslant\frac{12 n^2|\Gamma_i|}{\pi} I_i^\eps(t) m_{n-1,i}(t).
    \end{aligned}
    \end{equation}
    Regarding $A_{1,2}^{\eps}(t)$, we split into two terms
    \begin{align*}
        A_{1,2}^{\eps}(t) & = \frac{2n}{\pi \Gamma_i} \iint_{E_2(t)} \big(\mathbf{x} - c_i^\eps(t)\big) \cdot \frac{\mathbf{x}\wedge \mathbf{y}}{|\mathbf{x}-\mathbf{y}|_{\R^3}^2}|\mathbf{x}-c_i^\eps(t)|_{\R^3}^{4n-2} \zeta_i^\eps(t,\mathbf{y})\zeta_i^\eps(t,\mathbf{x}) \dd \boldsymbol{\sigma}(\mathbf{y})\dd \boldsymbol{\sigma}(\mathbf{x}) \\
        & \quad-\frac{2n}{\pi \Gamma_i} \iint_{E_2(t)} \big(\mathbf{x} - c_i^\eps(t)\big) \cdot (\mathbf{x}\wedge \mathbf{y})|\mathbf{x}-c_i^\eps(t)|_{\R^3}^{4n-4} \zeta_i^\eps(t,\mathbf{y})\zeta_i^\eps(t,\mathbf{x}) \dd \boldsymbol{\sigma}(\mathbf{y})\dd \boldsymbol{\sigma}(\mathbf{x})\\
        &\triangleq A_{1,2,1}^{\eps}(t)+A_{1,2,2}^{\eps}(t).
    \end{align*}
    Take $(\mathbf{x},\mathbf{y})\in E_2^{\eps}(t).$ By definition, we have
    \begin{equation}\label{estimE2}
        \left(1-\frac{1}{2n}\right)|\mathbf{x}-c_i^\eps(t)|_{\mathbb{R}^3} < |\mathbf{y}-c_i^\eps(t)|_{\mathbb{R}^3} < \left(1-\frac{1}{2n}\right)^{-1} |\mathbf{x}-c_i^\eps(t)|_{\mathbb{R}^3}.
    \end{equation}
    In particular,
    $$|\mathbf{x}-c_i^\eps(t)| \leqslant 2 |\mathbf{y}-c_i^\eps(t)|.$$
    Invoking one more time the Cauchy-Schwarz inequality and using \eqref{eq:maj_x_wedge_y} together with the right-triangle inequality, we infer
    \begin{align*}
        \left|\big(\mathbf{x} - c_i^\eps(t)\big) \cdot (\mathbf{x}\wedge \mathbf{y})\right|&\leqslant|\mathbf{x}-c_i^\eps(t)|_{\mathbb{R}^3}|\mathbf{x}-\mathbf{y}|_{\R^3}\\
        &\leqslant2|\mathbf{y}-c_i^\eps(t)|_{\R^3}\left(|\mathbf{x}-c_i^\eps(t)|_{\mathbb{R}^3}+|\mathbf{y}-c_i^\eps(t)|_{\mathbb{R}^3}\right)\\
        &\leqslant6|\mathbf{y}-c_i^\eps(t)|_{\mathbb{R}^3}^2.
    \end{align*}
    We deduce that
    \begin{equation}\label{e-A122}
        \begin{aligned}
        \left|A_{1,2,2}^{\eps}(t)\right|&\leqslant\frac{12n|\Gamma_i|}{\pi}\iint_{E_2^{\eps}(t)}|\mathbf{y}-c_i^\eps(t)|^2_{\R^3}|\mathbf{x}-c_i^\eps(t)|_{\R^3}^{4n-4} \frac{\zeta_i^\eps(t,\mathbf{y})}{\Gamma_i}\frac{\zeta_i^\eps(t,\mathbf{x})}{\Gamma_i} \dd \boldsymbol{\sigma}(\mathbf{y})\dd \boldsymbol{\sigma}(\mathbf{x})\\
        &\leqslant\frac{12 n|\Gamma_i|}{\pi} I_i^\eps(t) m_{n-1,i}(t).
    \end{aligned}
    \end{equation}
    As for $A_{1,2,1}^{\eps}(t),$ noticing that 
    \begin{equation*}
        (\mathbf{x} - c_i^\eps(t)) \cdot \mathbf{x}\wedge \mathbf{y} = - c_i^\eps(t) \cdot \mathbf{x}\wedge \mathbf{y},
    \end{equation*}
    we then symmetrize and get
    \begin{align*}
        A_{1,2,1}^{\eps}(t)&= -\frac{2nc_i^\eps(t)}{2\pi\Gamma_i} \cdot \iint_{E_2^{\eps}(t)}\frac{\mathbf{x}\wedge \mathbf{y}}{|\mathbf{x}-\mathbf{y}|_{\R^3}^2}\Big(|\mathbf{x}-c_i^\eps(t)|_{\R^3}^{4n-2}-|\mathbf{y}-c_i^\eps(t)|_{\R^3}^{4n-2}\Big) \zeta_i^\eps(t,\mathbf{y})\zeta_i^\eps(t,\mathbf{x}) \dd \boldsymbol{\sigma}(\mathbf{y})\dd \boldsymbol{\sigma}(\mathbf{x}) \\
        &= -\frac{n}{\pi\Gamma_i}\iint_{E_2^{\eps}(t)} (\mathbf{y} - c_i^\eps(t)) \cdot\frac{\mathbf{x}\wedge \mathbf{y}}{|\mathbf{x}-\mathbf{y}|_{\R^3}^2}\Big(|\mathbf{x}-c_i^\eps(t)|_{\R^3}^{4n-2}-|\mathbf{y}-c_i^\eps(t)|_{\R^3}^{4n-2}\Big) \zeta_i^\eps(t,\mathbf{y})\zeta_i^\eps(t,\mathbf{x}) \dd \boldsymbol{\sigma}(\mathbf{y})\dd \boldsymbol{\sigma}(\mathbf{x}).
    \end{align*}
    Now, using \eqref{estimE2}, we obtain that
    \begin{align*}
        \Big||\mathbf{x}-c_i^\eps(t)|_{\R^3}^{4n-2}-|\mathbf{y}-c_i^\eps(t)|_{\R^3}^{4n-2} \Big| & \leqslant |\mathbf{x}-\mathbf{y}|_{\R^3} \sum_{j=0}^{4n-3} |\mathbf{x}-c_i^\eps(t)|_{\R^3}^{4n-3-j} |\mathbf{y}-c_i^\eps(t)|_{\R^3}^j \\
        & \leqslant |\mathbf{x}-\mathbf{y}|_{\R^3}|\mathbf{x}-c_i^\eps(t)|_{\R^3}^{4n-4} |\mathbf{y}-c_i^\eps(t)|_{\R^3}  \left(\sum_{j=1}^{4n-3} \left(1-\frac{1}{2n}\right)^{-(j-1)} + \left(1-\frac{1}{2n}\right)^{-1} \right).
    \end{align*}
    Standard computations give
    \begin{align*}
        \sum_{j=1}^{4n-3} \left(1-\frac{1}{2n}\right)^{-(j-1)}&=\sum_{j=0}^{4n-4} \left(1-\frac{1}{2n}\right)^{-j}\\
        &=\frac{\left(1-\frac{1}{2n}\right)^{-(4n-3)}-1}{\frac{2n}{2n-1}-1}\\
        &=(2n-1)\left(\left(1-\frac{1}{2n}\right)^{-4n+3}-1\right)\\
        &\leqslant 2n\left(1-\frac{1}{2n}\right)^{-4n}\\
        &\leqslant 32n.
    \end{align*}
    To get the last inequality, we have used the fact that $n\mapsto\left(1-\frac{1}{2n}\right)^{-2n}$ is decreasing. Added to the fact that
    $$\left(1-\frac{1}{2n}\right)^{-1}\leqslant2\leqslant2n,$$
    we deduce that
    \begin{align*}
        \Big||\mathbf{x}-c_i^\eps(t)|_{\R^3}^{4n-2}-|\mathbf{y}-c_i^\eps(t)|_{\R^3}^{4n-2} \Big|\leqslant 34n|\mathbf{x}-\mathbf{y}|_{\R^3}|\mathbf{x}-c_i^\eps(t)|_{\R^3}^{4n-4} |\mathbf{y}-c_i^\eps(t)|_{\R^3}.
    \end{align*}
    With this in hand, then the Cauchy-Schwarz inequality and \eqref{eq:maj_x_wedge_y} imply
    \begin{equation}\label{e-A121}
        \begin{aligned}
        |A_{1,2,1}^{\eps}(t)|&\leqslant\frac{34n^2|\Gamma_i|}{\pi}\iint_{E_2^{\eps}(t)}|\mathbf{y}-c_i^\eps(t)|^2_{\R^3}|\mathbf{x}-c_i^\eps(t)|_{\R^3}^{4n-4} \frac{\zeta_i^\eps(t,\mathbf{y})}{\Gamma_i}\frac{\zeta_i^\eps(t,\mathbf{x})}{\Gamma_i} \dd \boldsymbol{\sigma}(\mathbf{y})\dd \boldsymbol{\sigma}(\mathbf{x})\\
        &\leqslant\frac{34n^2|\Gamma_i|}{\pi} I_i^\eps(t) m_{n-1,i}^\eps(t).
    \end{aligned}
    \end{equation}
    Putting together \eqref{e-A122} and \eqref{e-A121}, we infer
    \begin{equation}\label{e-A12}
        |A_{1,2}^{\eps}(t)| \leqslant \frac{46n^2|\Gamma_i|}{\pi}I_i^\eps(t)m_{n-1,i}^\eps(t) .
    \end{equation}
    Combining \eqref{e-A11}, \eqref{e-A13} and \eqref{e-A12}, we obtain
    \begin{equation}\label{e-A1}
        |A_1^{\eps}(t)| \leqslant \frac{70n^2|\Gamma_i|}{\pi}I_i^\eps(t)m_{n-1,i}^\eps(t).
    \end{equation}
We now turn to the analysis of $A_2^{\eps}(t)$, which we recall to be
\begin{equation*}
    A_2^{\eps}(t) = \frac{4n}{\Gamma_i}\int_{\mathbb{S}^2} \big(\mathbf{x} - c_i^\eps(t)\big) \cdot \left(F_i^\eps(t,\mathbf{x}) - \frac{1}{\Gamma_i}\int_{\mathbb{S}^2} F_i^\eps(t,\mathbf{y})\zeta_i^\eps(t,\mathbf{y})\dd \boldsymbol{\sigma}(\mathbf{y})\right)|\mathbf{x}-c_i^\eps(t)|_{\R^3}^{4n-2} \zeta_i^\eps(t,\mathbf{x}) \dd \boldsymbol{\sigma}(\mathbf{x}).
\end{equation*}By definition~\ref{def:Dprime} of $D_\eps$, more precisely by relation~\eqref{eq:Flip2}, we have that
    \begin{equation*}
        |A_2^{\eps}(t)| \leqslant 4n D_\eps \iint_{\mathbb{S}^2\times\mathbb{S}^2}|\mathbf{x}-\mathbf{y}|_{\R^3} \frac{\zeta_i^\eps(t,\mathbf{y})}{\Gamma_i}\dd \boldsymbol{\sigma}(\mathbf{y})|\mathbf{x}-c_i^\eps(t)|_{\R^3}^{4n-1} \frac{\zeta_i^\eps(t,\mathbf{x})}{\Gamma_i} \dd \boldsymbol{\sigma}(\mathbf{x}).
    \end{equation*}
    By right-triangular inequality, $|\mathbf{x}-\mathbf{y}|_{\R^3} \leqslant |\mathbf{x}-c_i^\eps(t)|_{\R^3} + |\mathbf{y}-c_i^\eps(t)|_{\R^3}$ and therefore
    \begin{equation*}
        |A_2^{\eps}(t)| \leqslant 4n D_\eps \left( m_{n,i}^\eps(t) + \int_{\mathbb{S}^2}|\mathbf{y}-c_i^\eps(t)|_{\R^3} \frac{\zeta_i^\eps(t,\mathbf{y})}{\Gamma_i}\dd \boldsymbol{\sigma}(\mathbf{y})\int_{\mathbb{S}^2}|\mathbf{x}-c_i^\eps(t)|_{\R^3}^{4n-1} \frac{\zeta_i^\eps(t,\mathbf{x})}{\Gamma_i} \dd \boldsymbol{\sigma}(\mathbf{x})\right).
    \end{equation*}
    By Cauchy-Schwarz inequality and the definition of $\Gamma_i$, we get
    $$\int_{\mathbb{S}^2}|\mathbf{y}-c_i^\eps(t)|_{\R^3} \frac{\zeta_i^\eps(t,\mathbf{y})}{\Gamma_i}\dd \boldsymbol{\sigma}(\mathbf{y})=\int_{\mathbb{S}^2}|\mathbf{y}-c_i^\eps(t)|_{\R^3} \sqrt{\frac{\zeta_i^\eps(t,\mathbf{y})}{\Gamma_i}}\sqrt{\frac{\zeta_i^\eps(t,\mathbf{y})}{\Gamma_i}}\dd \boldsymbol{\sigma}(\mathbf{y})\leqslant\sqrt{I_i^{\eps}(t)}.$$
    Besides, given $\mathbf{x}\in\textnormal{supp}\big(\zeta_i^{\eps}(t,\cdot)\big)$, since $c_i^\eps$ belongs to the convex envelop of $\textnormal{supp}\big(\zeta_i^{\eps}(t,\cdot)\big)$, we have that
    $$|\mathbf{x}-c_i^\eps(t)|_{\R^3}\leqslant 2 \eps^\beta.$$
    Combining the foregoing calculations leads to
    \begin{equation}\label{e-A2}
        |A_2^{\eps}(t)| \leqslant 4n D_\eps \left( m_{n,i}^\eps(t) + 2\sqrt{I_i^\eps(t)} \eps^{3\beta} m_{n-1,i}^\eps(t)\right).
    \end{equation}
    Gathering \eqref{e-A1} and \eqref{e-A2}, we deduce
    \begin{equation*}
        \der{}{t} m_{n,i}^\eps(t) \leqslant \frac{70n^2|\Gamma_i|}{\pi}I_i^\eps(t)m_{n-1,i}^{\eps}(t) + 4n D_\eps \left( m_{n,i}^\eps(t) + 2\sqrt{I_i^\eps(t)} \eps^{3\beta} m_{n-1,i}^\eps(t)\right).
    \end{equation*}
    Using Hölder's inequality, we notice that
    \begin{equation*}
        m_{n-1,i}^\eps(t)  = \int_{\mathbb{S}^2}\left(\frac{1}{\Gamma_i}|\mathbf{x}-c_i^\eps(t)|_{\R^3}^{4n} \zeta_i^\eps(t,\mathbf{x})\right)^{\frac{n-1}{n}} \left( \frac{1}{\Gamma_i} \zeta_i^\eps(t,\mathbf{x})\right)^{\frac{1}{n}}\dd \boldsymbol{\sigma}(\mathbf{x}) \leqslant \big(m_{n,i}^{\eps}(t)\big)^\frac{n-1}{n},
    \end{equation*}
    which concludes the proof of this lemma.
    \end{proof}

All these estimates are weak confinement properties. We now introduce the strong confinement tools by controlling the growth of the support of the absolute vorticity.

\subsection{Growth of the support}\label{sec:growth}
    Let us introduce for all $t\geqslant0,$
\begin{equation}\label{def:radmax}
    R_i^{\varepsilon}(t) \triangleq \inf \big\{ r > 0\quad\textnormal{s.t.}\quad \supp\big(\zeta_i^\eps(t,\cdot)\big)\subset B\big(c_i^\eps(t),r\big) \big\}.
\end{equation}
By compactness of the support, we get the existence of $X_i^{\varepsilon}(t)\in\supp\big(\zeta_i^{\varepsilon}(t,\cdot)\big)$ such that 
\begin{equation}\label{formula:radmax}
    |X_i^{\varepsilon}(t)-c_i^\eps(t)|_{\mathbb{R}^3}=R_i^{\varepsilon}(t).
\end{equation}\\
We denote by $s\mapsto X_{t,i}^{\varepsilon}(s)$ the trajectory passing through $X_i^{\varepsilon}(t)$ at time $t$, i.e. the solution of the Cauchy problem
\begin{equation}\label{def:flowt}
    \der{}{s}X_{t,i}^{\varepsilon}(s)=u_i^{\varepsilon}\big(s,X_{t,i}^{\varepsilon}(s)\big)+F_i^{\varepsilon}\big(s,X_{t,i}^{\varepsilon}(s)\big),\qquad X_{t,i}^{\varepsilon}(t)=X_i^{\varepsilon}(t).
\end{equation}
These trajectories are continuous, so $t\mapsto R_i^\eps(t)$ is also continuous.
Then, we have the following lemma used to estimate the growth of the support.
\begin{lem}\label{lem:vitesse_radiale}
For any $t\leqslant \tau_{\eps,\beta}$ we have that
\begin{equation*}
    \der{}{s}\left|X_{t,i}^{\varepsilon}(s)-c_i^\eps(s)\right|_{\mathbb{R}^3}\Big|_{s=t} \leqslant D_\eps R_i^{\varepsilon}(t) + \frac{11I_i^\eps(t)}{2(R_i^\eps(t))^3} + \left(\frac{M\eps^{-\eta}|\Gamma_i|}{\pi}\mathtt{m}_{t,i}^\eps\left(\frac{R_i^{\varepsilon}(t)}{2}\right)\right)^\frac{1}{2},
\end{equation*}
where
\begin{equation}\label{defttm}
    \mathtt{m}_{t,i}^\eps(r)\triangleq\int_{\mathbb{S}^2 \setminus \mathcal{C}(c_i^\eps(t),r)}\zeta_i^\eps(t,\mathbf{x})\dd \boldsymbol{\sigma} (\mathbf{x}),
\end{equation}
with
$$\forall x\in\mathbb{R}^3,\quad\forall r>0,\quad\mathcal{C}(x,r)\triangleq\Big\{y\in\mathbb{S}^2\quad\textnormal{s.t.}\quad|x-y|_{\mathbb{R}^3}\leqslant r\Big\}=B(x,r)\cap\mathbb{S}^2.$$
\end{lem}

\begin{proof}
    We follow the original proof by Buttà-Marchioro in the planar case \cite[Lem 2.5]{BM18} and adapt it to our situation. We fix $t\leqslant\tau_{\varepsilon,\beta}.$ By definition \eqref{def:flowt}, one readily has by differenciating
    $$\der{}{s}|X_{t,i}^{\varepsilon}(s)-c_i^\eps(s)|_{\mathbb{R}^3}\Big|_{s=t}=\left(u_i^{\varepsilon}\big(t,X_{i}^{\varepsilon}(t)\big)+F_i^{\varepsilon}\big(t,X_{i}^{\varepsilon}(t)\big)-\der{}{t}c_i^{\varepsilon}(t)\right)\cdot\frac{X_{i}^{\varepsilon}(t)-c_i^{\varepsilon}(t)}{\left|X_{i}^{\varepsilon}(t)-c_i^{\varepsilon}(t)\right|_{\mathbb{R}^3}}\cdot$$
    Using \eqref{pb:zetai}, \eqref{mass conserv} and \eqref{derivee de c}, we can write
    \begin{align*}
        \der{}{s}\left|X_{t,i}^{\varepsilon}(s)-c_i^\eps(s)\right|_{\mathbb{R}^3}\Big|_{s=t}&=\frac{1}{2\pi}\int_{\mathbb{S}^2}\frac{X_{i}^{\varepsilon}(t)\wedge\mathbf{y}}{|X_{i}^{\varepsilon}(t)-\mathbf{y}|_{\mathbb{R}^3}^2}\zeta_i^{\varepsilon}(t,\mathbf{y})d\boldsymbol{\sigma}(\mathbf{y})\cdot\frac{X_{i}^{\varepsilon}(t)-c_i^{\varepsilon}(t)}{\left|X_{i}^{\varepsilon}(t)-c_i^{\varepsilon}(t)\right|_{\mathbb{R}^3}}\\
        &\quad+\frac{1}{\Gamma_i}\left(\int_{\mathbb{S}^2}\left[F_i^{\varepsilon}\big(t,X_{i}^{\varepsilon}(t)\big)-F_i^{\varepsilon}(t,\mathbf{y})\right]\zeta_i^{\varepsilon}(t,\mathbf{x})d\boldsymbol{\sigma}(\mathbf{y})\right)\cdot\frac{X_{i}^{\varepsilon}(t)-c_i^{\varepsilon}(t)}{\left|X_{i}^{\varepsilon}(t)-c_i^{\varepsilon}(t)\right|_{\mathbb{R}^3}}\cdot
    \end{align*}
    Then, by definition of $D_{\varepsilon}$ in \ref{def:Dprime}, using relations~\eqref{eq:Flip2}, \eqref{mass conserv} and the definition \eqref{def:radmax}, we get
    \begin{align*}
        \left|\int_{\mathbb{S}^2}\left[F_i^{\varepsilon}\big(t,X_{i}^{\varepsilon}(t)\big)-F_i^{\varepsilon}(t,\mathbf{y})\right]\frac{\zeta_i^{\varepsilon}(t,\mathbf{y})}{\Gamma_i}d\boldsymbol{\sigma}(\mathbf{y})\right|&\leqslant D_\eps\int_{\mathbb{S}^2}|X_{i}^{\varepsilon}(t)-\mathbf{y}|_{\mathbb{R}^3}\frac{\zeta_i^{\varepsilon}(t,\mathbf{y})}{\Gamma_i}d\boldsymbol{\sigma}(\mathbf{y})\\
        &\leqslant D_\eps R_i^{\varepsilon}(t).
    \end{align*}
    Now, we split
    \begin{align*}
        \frac{1}{2\pi}\int_{\mathbb{S}^2}\frac{X_{i}^{\varepsilon}(t)\wedge\mathbf{y}}{|X_{i}^{\varepsilon}(t)-\mathbf{y}|_{\mathbb{R}^3}^2}\zeta_i^{\varepsilon}(t,\mathbf{y})d\boldsymbol{\sigma}(\mathbf{y})\cdot\frac{X_{i}^{\varepsilon}(t)-c_i^{\varepsilon}(t)}{\left|X_{i}^{\varepsilon}(t)-c_i^{\varepsilon}(t)\right|_{\mathbb{R}^3}}&=H_1+H_2,
    \end{align*}
    where
    \begin{align*}
        H_1&\triangleq\frac{1}{2\pi}\int_{\mathcal{C}\left(c_i^{\varepsilon}(t),\frac{R_i^{\varepsilon}(t)}{2}\right)}\frac{X_{i}^{\varepsilon}(t)-c_i^{\varepsilon}(t)}{\left|X_{i}^{\varepsilon}(t)-c_i^{\varepsilon}(t)\right|_{\mathbb{R}^3}}\cdot\frac{X_{i}^{\varepsilon}(t)\wedge\mathbf{y}}{|X_{i}^{\varepsilon}(t)-\mathbf{y}|_{\mathbb{R}^3}^2}\zeta_i^{\varepsilon}(t,\mathbf{y})d\boldsymbol{\sigma}(\mathbf{y}),\\
        H_2&\triangleq\frac{1}{2\pi}\int_{\mathbb{S}^2\setminus\mathcal{C}\left(c_i^{\varepsilon}(t),\frac{R_i^{\varepsilon}(t)}{2}\right)}\frac{X_{i}^{\varepsilon}(t)-c_i^{\varepsilon}(t)}{\left|X_{i}^{\varepsilon}(t)-c_i^{\varepsilon}(t)\right|_{\mathbb{R}^3}}\cdot\frac{X_{i}^{\varepsilon}(t)\wedge\mathbf{y}}{|X_{i}^{\varepsilon}(t)-\mathbf{y}|_{\mathbb{R}^3}^2}\zeta_i^{\varepsilon}(t,\mathbf{y})d\boldsymbol{\sigma}(\mathbf{y}).
    \end{align*}
    Observe that 
    \begin{align*}
        \big(X_{i}^{\varepsilon}(t)-c_i^{\varepsilon}(t)\big)\cdot\int_{\mathbb{S}^2}\Big(X_{i}^{\varepsilon}(t)\wedge\mathbf{y}\big)\zeta_i^{\varepsilon}(t,\mathbf{y})d\boldsymbol{\sigma}(\mathbf{y})&=\big(X_{i}^{\varepsilon}(t)-c_i^{\varepsilon}(t)\big)\cdot\left(X_{i}^{\varepsilon}(t)\wedge\int_{\mathbb{S}^2}\mathbf{y}\zeta_i^{\varepsilon}(t,\mathbf{y})d\boldsymbol{\sigma}(\mathbf{y})\right)\\
        &=\big(X_{i}^{\varepsilon}(t)-c_i^{\varepsilon}(t)\big)\cdot\big(X_{i}^{\varepsilon}(t)\wedge c_i^{\varepsilon}(t)\big)\\
        &=0.
    \end{align*}
    Consequently, we can write
    $$H_1=H_{1,1}-H_{1,2},$$
    with
    \begin{align*}
        H_{1,1}&=\frac{1}{2\pi}\int_{\mathcal{C}\left(c_i^{\varepsilon}(t),\frac{R_i^{\varepsilon}(t)}{2}\right)}\frac{X_{i}^{\varepsilon}(t)-c_i^{\varepsilon}(t)}{\left|X_{i}^{\varepsilon}(t)-c_i^{\varepsilon}(t)\right|_{\mathbb{R}^3}}\cdot\big(X_{i}^{\varepsilon}(t)\wedge\mathbf{y}\big)\left(\frac{1}{|X_{i}^{\varepsilon}(t)-\mathbf{y}|_{\mathbb{R}^3}^2}-\frac{1}{|X_{i}^{\varepsilon}(t)-c_i^{\varepsilon}(t)|_{\mathbb{R}^3}^2}\right)\zeta_i^{\varepsilon}(t,\mathbf{y})d\boldsymbol{\sigma}(\mathbf{y}),\\
        H_{1,2}&\triangleq\frac{1}{2\pi}\int_{\mathbb{S}^2\setminus\mathcal{C}\left(c_i^{\varepsilon}(t),\frac{R_i^{\varepsilon}(t)}{2}\right)}\frac{X_{i}^{\varepsilon}(t)-c_i^{\varepsilon}(t)}{\left|X_{i}^{\varepsilon}(t)-c_i^{\varepsilon}(t)\right|_{\mathbb{R}^3}^3}\cdot\big(X_{i}^{\varepsilon}(t)\wedge\mathbf{y}\big)\zeta_i^{\varepsilon}(t,\mathbf{y})d\boldsymbol{\sigma}(\mathbf{y}).
    \end{align*}
    Using the identities \eqref{useful-identity00} and \eqref{formula:radmax}, the term $H_{1,1}$ becomes
    $$H_{1,1}=\frac{1}{2\pi\left(R_i^{\varepsilon}(t)\right)^3}\int_{\mathcal{C}\left(c_i^{\varepsilon}(t),\frac{R_i^{\varepsilon}(t)}{2}\right)}\big(X_{i}^{\varepsilon}(t)-c_i^{\varepsilon}(t)\big)\cdot\big(X_{i}^{\varepsilon}(t)\wedge\mathbf{y}\big)\frac{\Big(2\big(X_{i}^{\varepsilon}(t)-c_i^\eps(t)\big) - \big(\mathbf{y} - c_i^\eps(t)\big)\Big) \cdot \big(\mathbf{y}-c_i^\eps(t)\big) }{|X_{i}^{\varepsilon}(t)-\mathbf{y}|^2_{\R^3}}\zeta_i^{\varepsilon}(t,\mathbf{y})d\boldsymbol{\sigma}(\mathbf{y}).$$
    Notice that $\big(X_{i}^{\varepsilon}(t)-c_i^{\varepsilon}(t)\big)\cdot\big(X_{i}^{\varepsilon}(t)\wedge\mathbf{y}\big)=\big(\mathbf{y}-c_i^{\varepsilon}(t)\big)\cdot\big(X_{i}^{\varepsilon}(t)\wedge\mathbf{y}\big)$. Hence,
    $$H_{1,1}=\frac{1}{2\pi\left(R_i^{\varepsilon}(t)\right)^3}\int_{\mathcal{C}\left(c_i^{\varepsilon}(t),\frac{R_i^{\varepsilon}(t)}{2}\right)}\big(\mathbf{y}-c_i^{\varepsilon}(t)\big)\cdot\big(X_{i}^{\varepsilon}(t)\wedge\mathbf{y}\big)\frac{\Big(2\big(X_{i}^{\varepsilon}(t)-c_i^\eps(t)\big) - \big(\mathbf{y} - c_i^\eps(t)\big)\Big) \cdot \big(\mathbf{y}-c_i^\eps(t)\big) }{|X_{i}^{\varepsilon}(t)-\mathbf{y}|^2_{\R^3}}\zeta_i^{\varepsilon}(t,\mathbf{y})d\boldsymbol{\sigma}(\mathbf{y}).$$
    Now, remark that for $\mathbf{y}\in\mathcal{C}\left(c_i^{\varepsilon}(t),\frac{R_i^{\varepsilon}(t)}{2}\right)$, we have $|\mathbf{y}-c_i^{\varepsilon}(t)|_{\mathbb{R}^3}\leqslant\frac{R_i^{\varepsilon}(t)}{2}\cdot$ Thus, by \eqref{formula:radmax} and left-triangular inequality, we infer
    $$|X_{i}^{\varepsilon}(t)-\mathbf{y}|_{\R^3}\geqslant|X_{i}^{\varepsilon}(t)-c_i^{\varepsilon}(t)|_{\R^3}-|\mathbf{y}-c_i^{\varepsilon}(t)|_{\R^3}\geqslant\frac{R_i^{\varepsilon}(t)}{2}\cdot$$
    As a consequence, by triangular inequality and \eqref{formula:radmax},
    $$\left|2\big(X_{i}^{\varepsilon}(t)-c_i^\eps(t)\big) - \big(\mathbf{y} - c_i^\eps(t)\big)\right|_{\mathbb{R}^3}\leqslant|X_{i}^{\varepsilon}(t)-\mathbf{y}|_{\mathbb{R}^3}+|X_{i}^{\varepsilon}(t)-c_i^{\varepsilon}(t)|_{\mathbb{R}^3}\leqslant 3|X_{i}^{\varepsilon}(t)-\mathbf{y}|_{\mathbb{R}^3}.$$
    Applying one more time the Cauchy-Schwarz inequality, we deduce that
    $$|H_{1,1}|\leqslant\frac{3|\Gamma_i|I_{i}^{\varepsilon}(t)}{2\pi\left(R_i^{\varepsilon}(t)\right)^3}\cdot$$
    Now we shall estimate $H_{1,2}.$ By construction \eqref{def:radmax}-\eqref{formula:radmax}, for $\mathbf{y}\in\supp\big(\zeta_i^{\varepsilon}(t,\cdot)\big),$
    $$|X_i^{\varepsilon}(t)-\mathbf{y}|_{\mathbb{R}^3}\leqslant|X_i^{\varepsilon}(t)-c_i^{\varepsilon}(t)|_{\mathbb{R}^3}+|c_i^{\varepsilon}(t)-\mathbf{y}|_{\mathbb{R}^3}\leqslant 2R_i^{\varepsilon}(t).$$
    Hence, Cauchy-Schwarz ineaquality implies
    $$\left|\frac{X_{i}^{\varepsilon}(t)-c_i^{\varepsilon}(t)}{\left|X_{i}^{\varepsilon}(t)-c_i^{\varepsilon}(t)\right|_{\mathbb{R}^3}^3}\cdot\big(X_{i}^{\varepsilon}(t)\wedge\mathbf{y}\big)\right|\leqslant\frac{2}{R_i^{\varepsilon}(t)}\cdot$$
    Therefore,
    $$|H_{1,2}|\leqslant\frac{|\Gamma_i|}{\pi R_i^{\varepsilon}(t)}\int_{\mathbb{S}^2\setminus\mathcal{C}\left(c_i^{\varepsilon}(t),\frac{R_i^{\varepsilon}(t)}{2}\right)}\frac{\zeta_i^{\varepsilon}(t,\mathbf{y})}{\Gamma_i}d\boldsymbol{\sigma}(\mathbf{y}).$$
    We consider the measure $\mu_{t,i}^{\varepsilon}$ defined on the measurable space $\big(\mathbb{S}^2,\mathcal{B}(\mathbb{S}^2)\big)$ given by
    $$\mu_{t,i}^{\varepsilon}(A)=\int_{A}\frac{\zeta_i^{\varepsilon}(t,\mathbf{y})}{\Gamma_i}\dd \sigma(\mathbf{y}).$$
    According to \eqref{mass conserv}, $\mu_{t,i}^{\varepsilon}$ is a probability measure on $\mathbb{S}^2$, which is absolutely continuous with respect to $\boldsymbol{\sigma}.$ The classical Markov inequality gives
    \begin{align*}
        |H_{1,2}|&\leqslant\frac{|\Gamma_i|}{\pi R_i^{\varepsilon}(t)}\mu_{t,i}^{\varepsilon}\left(\left\{\mathbf{y}\in\mathbb{S}^2\quad\textnormal{s.t.}\quad|\mathbf{y}-c_i^{\varepsilon}(t)|_{\mathbb{R}^3}\geqslant\frac{R_i^{\varepsilon}(t)}{2}\right\}\right)\\
        &=\frac{|\Gamma_i|}{\pi R_i^{\varepsilon}(t)}\mu_{t,i}^{\varepsilon}\left(\left\{\mathbf{y}\in\mathbb{S}^2\quad\textnormal{s.t.}\quad|\mathbf{y}-c_i^{\varepsilon}(t)|_{\mathbb{R}^3}^2\geqslant\left(\frac{R_i^{\varepsilon}(t)}{2}\right)^2\right\}\right)\\
        &\leqslant\frac{4|\Gamma_i|}{\pi\left(R_i^{\varepsilon}(t)\right)^3}\int_{\mathbb{S}^2}|\mathbf{y}-c_i^{\varepsilon}(t)|_{\mathbb{R}^3}^2d\mu_{t,i}^{\varepsilon}(\mathbf{y})\\
        &=\frac{4|\Gamma_i|}{\pi\left(R_i^{\varepsilon}(t)\right)^3}\int_{\mathbb{S}^2}|\mathbf{y}-c_i^{\varepsilon}(t)|_{\mathbb{R}^3}^2\frac{\zeta_i^{\varepsilon}(t,\mathbf{y})}{\Gamma_i}\dd\sigma(\mathbf{y})\\
        &=\frac{4|\Gamma_i|I_i^{\varepsilon}(t)}{\pi\left(R_i^{\varepsilon}(t)\right)^3}\cdot
    \end{align*}
    At last, we focus on the term $H_2.$ One readily gets
    $$|H_2|\leqslant\frac{|\Gamma_i|}{2\pi}\mathcal{I}_i^{\varepsilon}(t),\qquad\mathcal{I}_i^{\varepsilon}(t)\triangleq\int_{\mathbb{S}^2\setminus\mathcal{C}\left(c_i^{\varepsilon}(t),\frac{R_i^{\varepsilon}(t)}{2}\right)}\frac{1}{|X_i^{\varepsilon}(t)-\mathbf{y}|_{\mathbb{R}^3}}\frac{\zeta_i^{\varepsilon}(t,\mathbf{y})}{\Gamma_i}d\boldsymbol{\sigma}(\mathbf{y}).$$
    The integrand is monotonically unbounded as $y\to X_{i}^{\varepsilon}(t)$, so the maximum of the integral is obtained when we rearrange the vorticity mass as close as possible to the singularity. In view of the Hypothesis \ref{hyp:zeta} and since, by \eqref{defttm}, $\mathtt{m}_{t,i}^\eps\left(\frac{R_i^{\varepsilon}(t)}{2}\right)$ is equal to the total amount of vorticity in $\mathbb{S}^2\setminus\mathcal{C}\left(c_i^{\varepsilon}(t),\frac{R_i^{\varepsilon}(t)}{2}\right)$, this rearrangement gives
    $$\mathcal{I}_i^{\varepsilon}(t)\leqslant\max\left\{\int_{\mathbb{S}^2}\frac{1}{|\mathtt{N}-\mathbf{y}|_{\mathbb{R}^3}}\zeta(\mathbf{y})d\boldsymbol{\sigma}(\mathbf{y}),\qquad\int_{\mathbb{S}^2}\zeta(\mathbf{y})d\boldsymbol{\sigma}(\mathbf{y})=\mathtt{m}_{t,i}^\eps\left(\frac{R_i^{\varepsilon}(t)}{2}\right),\qquad0\leqslant\zeta\leqslant \frac{M\varepsilon^{-\eta}}{|\Gamma_i|}\right\}.$$
    Let us recall that $\mathtt{N}$ denotes the North pole of the sphere. The previous maximum is obtained for $\zeta\equiv\frac{M\varepsilon^{-\eta}}{|\Gamma_i|}$ on the spherical cap $\mathcal{C}(\mathtt{N},r)$ for $r>0$ such that
    \begin{equation}\label{eq:areacap}
        \frac{M\varepsilon^{-\eta}}{|\Gamma_i|}\int_{\mathcal{C}(\mathtt{N},r)}d\boldsymbol{\sigma}(\mathbf{y})=\mathtt{m}_{t,i}^\eps\left(\frac{R_i^{\varepsilon}(t)}{2}\right).
    \end{equation}
    Let us consider on the unit sphere, the geodesic distance $\mathtt{d}_{\mathbb{S}^2}.$ By definition of the function $\sin$, one has
    $$\forall\mathbf{x},\mathbf{y}\in\mathbb{S}^2,\quad|\mathbf{x}-\mathbf{y}|_{\R^3}=2\sin\left(\frac{\mathtt{d}_{\mathbb{S}^2}(\mathbf{x},\mathbf{y})}{2}\right).$$
    Therefore,
    \begin{equation}\label{areacap}
        \int_{\mathcal{C}(\mathtt{N},r)}d\boldsymbol{\sigma}(\mathbf{y})=2\pi\int_{0}^{2\arcsin\left(\frac{r}{2}\right)}\sin(\theta)d\theta=2\pi\Big(1-\cos\big(2\arcsin\left(\tfrac{r}{2}\right)\big)\Big)=\pi r^2.
    \end{equation}
    Inserting \eqref{areacap} into \eqref{eq:areacap}, we find
    $$r=\sqrt{\frac{\varepsilon^{\eta}|\Gamma_i|}{M\pi}\mathtt{m}_{t,i}^\eps\left(\frac{R_i^{\varepsilon}(t)}{2}\right)}.$$
    Besides,
    $$\int_{\mathcal{C}(\mathtt{N},r)}\frac{1}{|\mathtt{N}-\mathbf{y}|_{\mathbb{R}^3}}d\boldsymbol{\sigma}(\mathbf{y})=2\pi\int_{0}^{2\arcsin\left(\frac{r}{2}\right)}\frac{\sin(\theta)}{2\sin\left(\frac{\theta}{2}\right)}d\theta=2\pi\int_{0}^{2\arcsin\left(\frac{r}{2}\right)}\cos\left(\frac{\theta}{2}\right)d\theta=2\pi r.$$
    Hence,
    $$\mathcal{I}_{i}^{\varepsilon}(t)\leqslant \frac{2\pi r M\varepsilon^{-\eta}}{|\Gamma_i|}=2\pi\sqrt{\frac{M\varepsilon^{-\eta}}{\pi|\Gamma_i|}\mathtt{m}_{t,i}^\eps\left(\frac{R_i^{\varepsilon}(t)}{2}\right)}.$$
    Thus, 
    $$|H_2|\leqslant\sqrt{\frac{M\varepsilon^{-\eta}|\Gamma_i|}{\pi}\mathtt{m}_{t,i}^\eps\left(\frac{R_i^{\varepsilon}(t)}{2}\right)}.$$
    This achieves the proof of Lemma \ref{lem:vitesse_radiale}.
\end{proof}

\section{Logarithmic confinement results}\label{sec:confinement_general}
Here, we use the results of the previous section to prove Theorems \ref{theo:general_conf} and \ref{theo:optimal}.
\subsection{Proof of Theorem \ref{theo:general_conf}}\label{sec:proofgenerale}
Recall the definition~\eqref{maj:F_Lip} of $D$, and that in all generality, $D_\eps \leqslant D$. Let $\alpha > 0$ to be chosen later, and let
\begin{equation}\label{choice delta}
    \delta \triangleq 2\alpha D.
\end{equation}

We thus have the following estimates on the vorticity moments.

\begin{lem}\label{lem:moments_generaux}
    There exists a constant $C$ such that for $\eps>0$ small enough, for any $i \in \{1,\ldots,N\}$ and any $t \leqslant \min(\tau_{\eps,\beta},\alpha |\ln \eps|)$,
    \begin{equation*}
        I_i^\eps(t) \leqslant C \eps^{2-\delta}, \qquad |c_i^\eps(t)-\mathbf{x}_i(t)|_{\mathbb{R}^3} \leqslant C \eps^{1-\frac{\delta}{2}-\alpha}, \qquad m_{n,i}^\eps(t) \leqslant C_n \eps^{(2-\delta-Cn \alpha)n}.
    \end{equation*}
\end{lem}
\begin{proof} Fix $i\in\{1,\ldots,N\}$ and $t \leqslant\min(\tau_{\varepsilon,\beta},\alpha|\ln \eps|)$.\\
$\blacktriangleright$ \textit{First estimate :} Since, $t\leqslant\tau_{\varepsilon,\beta},$ we can apply Lemma~\ref{lem:control_I} and use Gronwall's Lemma and the relation~\eqref{maj:F_Lip} to get  
\begin{equation*}
    |I_i^\eps(t)| \leqslant I_i^\eps(0) \exp\big( 2 D t).
\end{equation*}
Making appeal to Lemma~\ref{lem:initestim}, using the fact that $t \leqslant\alpha|\ln \eps|$ and the definition of $\delta$ in \eqref{choice delta}, we conclude that
\begin{equation}\label{e-It}
    I_i^\eps(t) \leqslant C \eps^{2-\delta}.
\end{equation}
$\blacktriangleright$ \textit{Second estimate :} Since $t\leqslant\tau_{\varepsilon,\beta},$ we can apply Lemma~\ref{lem:c} together with \eqref{e-It}, to get
\begin{equation*}
    \left|\der{}{t}\vec{c}^{\,\eps}(t) - \mathcal{F}\big(\vec{c}^{\,\eps}(t)\big)\right|_{\mathbb{R}^{3N}} \leqslant C \eps^{1-\frac{\delta}{2}}.
\end{equation*} 
Denoting $\mathbf{X}(t)$ the point vortex solution
$$\der{}{t}\mathbf{X}(t)=\mathcal{F}\big(\mathbf{X}(t)\big),$$
we can apply the variant of Gronwall's lemma provided in Lemma~\ref{lem:gronwall} and deduce that
\begin{equation*}
    |\vec{c}^{\,\eps}(t) - \mathbf{X}(t)|_{\mathbb{R}^{3N}} \leqslant \big(t C \eps^{1-\frac{\delta}{2}}+ |\vec{c}^{\,\eps}(0) - \mathbf{X}(0)|_{\mathbb{R}^{3N}} \big) \exp(C t).
\end{equation*}
Therefore, using Lemma~\ref{lem:initestim}, we conclude that for $\eps$ small enough, for every $i \in \{1,\ldots,N\}$ and every $t \leqslant \min(\tau_{\eps,\beta},\alpha |\ln \eps|)$ that
\begin{equation*}
    |c_i^\eps(t)-\mathbf{x}_i(t)| \leqslant C \eps^{1-\frac{\delta}{2} - \alpha}.
\end{equation*}
$\blacktriangleright$ \textit{Third estimate :} Combining the estimate of Lemma~\ref{lem:4n} together with \eqref{e-It} and \eqref{maj:F_Lip}, we infer
    \begin{equation*}
        \der{}{t} m_{n,i}^\eps(t) \leqslant C n^2\big(m_{n,i}^{\eps}(t)\big)^\frac{n-1}{n} \left(\eps^{2-\delta} + \big(m_{n,i}^{\eps}(t)\big)^\frac{1}{n} + \eps^{3\beta+1-\frac{\delta}{2}} \right).
    \end{equation*}
    Assuming that $3\beta+1 > 2$, namely that $\beta >1/3$, which we can do since proving Theorem~\ref{theo:general_conf} for some $\beta$ implies it for any $\beta' < \beta$, then the rightmost term is negligible compared to $\eps^{2-\delta}$. Therefore, we obtain
    \begin{equation*}
        \der{}{t} m_{n,i}^\eps(t) \leqslant C n^2\big(m_{n,i}^{\eps}(t)\big)^\frac{n-1}{n} \left(\eps^{2-\delta} + \big(m_{n,i}^{\eps}(t)\big)^\frac{1}{n} \right).
    \end{equation*}
    We apply Lemma~\ref{lem:4nODE} together with Lemma~\ref{lem:initestim}, leading to
    \begin{equation*}
              m_{n,i}^\eps(t) \leqslant \left( -\eps^{2-\delta} + \left(\eps^{2-\delta} + 16\eps^4 \right) e^{Cn t} \right)^n.
    \end{equation*}
    Thus, since $t \leqslant \alpha |\ln \eps|$, by compared growth for $\varepsilon$ small enough, we have that
    \begin{equation*}
        m_{n,i}^\eps(t) \leqslant C_n \eps^{(2-\delta-Cn \alpha)n}.
    \end{equation*}
This ends the proof of Lemma \ref{lem:moments_generaux}.

\end{proof}

The first result that we infer from these estimates is the following control on the vorticity far from the center of mass, which we call \emph{weak} confinement. Recall the definition of $\mathtt{m}_{t,i}^\eps$ given in~\eqref{defttm}. 
\begin{lem}\label{lem:4n_estExt}
    Let $\beta'\triangleq\frac{\beta + \frac{1}{2}}{2} \in \left(\beta,\tfrac{1}{2}\right)$. For every $\nu > 0$, there exists $\eps > 0$ depending only on $\beta$ and $\nu$ such that provided $\alpha$ is chosen such small enough (depending on $\nu$ and $\beta$ only), for every $\eps \in (0,\eps_0)$, every $t\leqslant  \min(\tau_{\eps,\beta},\alpha |\ln \eps|)$ and any $r \geqslant\tfrac{\eps^{\beta'}}{2}$,
    \begin{equation*}
    \max_{1\leqslant i\leqslant N}\mathtt{m}_{t,i}^\eps(r) \leqslant \frac{\eps^{5+\nu}}{r^6},
\end{equation*}
where $\mathtt{m}_{t,i}^{\varepsilon}(r)$ is defined in \eqref{defttm}.
\end{lem}
\begin{proof}
    We set
    $$\delta^*\triangleq \frac{2-4\beta}{4}\cdot$$ 
    Fix $t \leqslant \min\big(\tau_{\varepsilon,\beta},\alpha |\ln \eps|\big)$. Invoking Lemma \ref{lem:moments_generaux}, we get by compared growth as $\varepsilon\to0$ there exists $\alpha_n>0$ small enough that if $\alpha \leqslant \alpha_n$, then
    $$m_{n,i}^\eps(t) \leqslant C_n \eps^{(2-\delta^*)n}.$$ 
    Therefore, for any $r\geqslant\frac{\varepsilon^{\beta'}}{2}$ and any $\nu>0$, we have
\begin{align*}
    \frac{1}{\Gamma_i}\int_{\mathbb{S}^2 \setminus \mathcal{C}(c_i^\eps(t),r)} \zeta_i^\eps(t,\mathbf{x}) \dd \boldsymbol{\sigma}(\mathbf{x}) & =  \frac{1}{\Gamma_i}\int_{\mathbb{S}^2 \setminus \mathcal{C}(c_i^\eps(t),r)} \frac{|\mathbf{x} - c_i^\eps(t)|_{\mathbb{R}^3}^{4n}}{|\mathbf{x} - c_i^\eps(t)|_{\mathbb{R}^3}^{4n}} \zeta_i^\eps(t,\mathbf{x}) \dd \boldsymbol{\sigma}(\mathbf{x}) \\
         & \leqslant \frac{m_{n,i}^\eps(t)}{r^{4n}} \\
         & \leqslant \frac{C_n \eps^{n(2-\delta^*)}}{r^{4n}}\\
         & \leqslant C_n\frac{\eps^{5+\nu}}{r^6} \frac{\eps^{(2-\delta^*)n-5-\nu}}{\left(\frac{\eps^{\beta'}}{2}\right)^{4n-6}} \\
         & \leqslant C_n 2^{4n-6} \frac{\eps^{5+\nu}}{r^6} \eps^{(2-\delta^*-4\beta')n - 5 - \nu + 6\beta'}.
\end{align*}
Since $2-\delta^* - 4\beta' = \tfrac{1}{2}-\beta$, taking $n = \left\lfloor \frac{5+\nu}{\frac{1}{2}-\beta}\right\rfloor +2$ there only remains to take $\eps$ small enough such that
\begin{equation*}
    \eps C_n 2^{4n-6} < 1.
\end{equation*}
Since $n$ depends only on $\beta$ and $\nu$, so does $\eps_0$ such that this holds for every $ \eps \in (0,\eps_0)$.
\end{proof}

We are now ready to prove Theorem~\ref{theo:general_conf}. The idea is to use the previous estimates to show that for every time $t \leqslant \min( \tau_{\eps,\beta},\alpha |\ln \eps|) $ and for any $1\leqslant i\leqslant N$, we have that $\supp\big(\zeta_i^\eps(t)\big) \subset B\left(c_i^{\varepsilon}(t),\tfrac{3\eps^\beta}{4}\right)$. Thus, necessarily, $\tau_{\eps,\beta} \geqslant\alpha |\ln \eps|$.

\subsubsection*{Conclusion of the proof of Theorem \ref{theo:general_conf}}

Fix $i\in\{1,\ldots,N\}.$ For every $\eps > 0$ and $t \geqslant0$, recall the notations of Section~\ref{sec:growth} concerning the existence of a point $X_i^\eps$ and a trajectory $s\mapsto X_{t,i}^\eps(s)$ such that relations~\eqref{formula:radmax} and~\eqref{def:flowt} hold. Applying Lemma~\ref{lem:moments_generaux} to Lemma~\ref{lem:vitesse_radiale}, we have that 
\begin{equation}\label{control ecartvar}
    \der{}{s}\left|X_{t,i}^{\varepsilon}(s)-c_i^\eps(s)\right|_{\mathbb{R}^3}\Big|_{s=t} \leqslant D R_i^{\varepsilon}(t) + C_2\frac{\eps^{2-\delta}}{(R_i^\eps(t))^3} + \left(\frac{M\eps^{-\eta}|\Gamma_i|}{\pi}\mathtt{m}_{t,i}^\eps\left(\frac{R_i^{\varepsilon}(t)}{2}\right)\right)^\frac{1}{2}.
\end{equation}
Let $f_i$ be the solution of the ODE:
\begin{equation}\label{EDOfi}
    f_i'(t) =  2 D f_i(t) + 2 C_2\frac{\eps^{2-\delta}}{f_i^3(t)} + 2\left(\frac{M\eps^{-\eta}|\Gamma_i|}{\pi}\mathtt{m}_{t,i}^\eps\left(\frac{f_i(t)}{2}\right)\right)^\frac{1}{2},
\end{equation}
with initial data 
\begin{equation}\label{init fi}
    f_i(0) = 2 R_i^{\varepsilon}(0).
\end{equation} 
First observe that \eqref{control ecartvar} and \eqref{EDOfi} imply
\begin{equation}\label{maj0402}
    f_i'(t)>\der{}{s}\left|X_{t,i}^{\varepsilon}(s)-c_i^\eps(s)\right|_{\mathbb{R}^3}\Big|_{s=t} .
\end{equation}
Then, we claim that
\begin{equation}\label{claim-fi}
    \forall t\leqslant\min(\tau_{\varepsilon,\beta},\alpha|\ln(\varepsilon)|), \quad f_i(t) > R_i^{\varepsilon}(t).
\end{equation}
Indeed, assume that the converse is true and take 
$$t'\triangleq\min\{s\in[0,\min(\tau_{\varepsilon,\beta},\alpha|\ln(\varepsilon)|)]\quad\textnormal{s.t.}\quad f_i(s)=R_i^{\varepsilon}(s)\}.$$
According to \eqref{init fi}, we have $t'>0.$ In view of the definition \eqref{formula:radmax}, one has
\begin{equation}\label{eq0402}
    \left|X_i^{\varepsilon}(t')-c_i^{\varepsilon}(t')\right|_{\mathbb{R}^3}=R_i^{\varepsilon}(t')=f_i(t').
\end{equation}
By construction, for any $0<h<t'$
\begin{equation}\label{bnd0402}
    \left|X_{t',i}^{\varepsilon}(t'-h)-c_i^{\varepsilon}(t'-h)\right|_{\mathbb{R}^3}\leqslant R_i^{\varepsilon}(t'-h)<f_i(t'-h).
\end{equation}
Combining \eqref{eq0402} and \eqref{bnd0402}, we obtain
\begin{equation}\label{taux accroissement}
    \frac{\left|X_{t',i}^{\varepsilon}(t'-h)-c_i^{\varepsilon}(t'-h)\right|_{\mathbb{R}^3}-\left|X_i^{\varepsilon}(t')-c_i^{\varepsilon}(t')\right|_{\mathbb{R}^3}}{-h}>\frac{f_i(t'-h)-f_i(t')}{-h}\cdot
\end{equation}
Passing to the limit $h\to0$ in \eqref{taux accroissement}, we get a contradiction with \eqref{maj0402}. This proves the claim \eqref{claim-fi}.

Let us now prove that there does not exist an index $i \in \{1,\ldots,N\}$ and a time $t_2 \leqslant \alpha|\ln \eps|$ such that $f_i(t_2) \geqslant \eps^{\beta}/2$. We proceed by contradiction: assume there exists such time. Let $\beta' = \frac{\beta + 1/2}{2} \in (\beta,1/2)$ and let $t_1$ be the last time prior to $t_2$ such that for every $t \in [t_1,t_2]$, we have that $\eps^{\beta'} \leqslant f_i(t) \leqslant \eps^{\beta}/2$.
Since $f_i(t)/2 \geqslant \eps^{\beta'}/2$, we can apply Lemma~\ref{lem:4n_estExt} with $r = f_i(t)$ and any positive $\nu$ such that $\nu \geqslant \eta +  6\beta' -4$ to get that for all $t \in [t_1,t_2]$,
\begin{align*}
    f_i'(t) & \leqslant  2 D f_i(t) + 2 C_2\frac{\eps^{2-\delta}}{f_i^3(t)} + 2\left(\frac{M\eps^{-\eta}|\Gamma_i|}{\pi}2^6\frac{\eps^{5+\nu}}{f_i^6(t)}\right)^\frac{1}{2} \\
    & \leqslant  2 D f_i(t) + 2 C_2\frac{\eps^{2-\delta}}{f_i^3(t)} + C\eps^{(5+\nu-\eta-6\beta')/2} \\
    & \leqslant C \left(f_i(t) +\frac{\eps^{2-\delta}}{f_i^3(t)}\right),
\end{align*}
where in the last inequality we used that $\tfrac{1}{2}(5+\nu-\eta-6\beta')> \beta'$ and thus $\eps^{\frac{1}{2}(5+\nu-\eta-6\beta')} \ll \eps^\beta \leqslant f_i(t)$ as $\eps \to 0$.
Multiplying by $f_i^3(t)$ gives that
\begin{equation*}
    (f_i^4)'(t) \leqslant C (f_i^4(t) + \eps^{2-\delta}).
\end{equation*}
This in turns gives that
\begin{equation*}
    f_i^4(t_2) \leqslant f_i^4(t_1) e^{C(t_2-t_1)} + C \eps^{2-\delta} (e^{C(t_2-t_1)}-1) = \eps^{4\beta' - C\alpha} + C \eps^{2-\delta} (\eps^{-C\alpha}-1).
\end{equation*}
Provided $\alpha$ is small enough so that $\beta' - C\alpha/4 > \beta$, and $\eps$ is small enough, we have that
\begin{equation*}
    f_i(t_2) < \tfrac{\eps^\beta}{2},
\end{equation*}
which is a contradiction. No such time $t_2$ exists for any index $i$, and thus provided $\alpha$ and $\eps$ are small enough, for every $i \in\{1,\ldots,N\}$, for every $t \leqslant \min(\tau_{\eps,\beta},\alpha|\ln \eps|),$ $R_i^{\varepsilon}(t) < f_i(t) < \tfrac{\eps^\beta}{2}$. Therefore, for every such $t$ and $i$, for every point $\mathbf{x} \in \supp \big(\zeta_i^\eps(t,\cdot)\big)$, we have that
\begin{equation*}
    |\mathbf{x}-\mathbf{x}_i(t)|_{\mathbb{R}^3} \leqslant |\mathbf{x}-c_i^\eps(t)|_{\mathbb{R}^3} + |\mathbf{x}_i(t) - c_i^\eps(t)|_{\mathbb{R}^3} \leqslant R_i^{\varepsilon}(t) + |\mathbf{x}_i(t) - c_i^\eps(t)|_{\mathbb{R}^3} \leqslant \frac{3}{4}\eps^\beta,
\end{equation*}
where we used Lemma~\ref{lem:moments_generaux} provided $\eps$ and $\alpha$ are small enough. Applying this last equality in $\tau_{\eps,\beta}$ would then be a contradiction as by definition of $\tau_{\eps,\beta}$, there exists $i_0\in\{1,\ldots,N\}$ and $\overline{\mathbf{x}} \in \supp \big(\zeta_{i_0}^\eps(\tau_{\eps,\beta},\cdot)\big)$ such that 
\begin{equation*}
    |\overline{\mathbf{x}}-\mathbf{x}_{i_0}(t)|_{\mathbb{R}^3} = \eps^\beta.
\end{equation*}
Therefore, this inequality cannot be applied in $\tau_{\eps,\beta}$, so $\min(\tau_{\eps,\beta}, \alpha |\ln \eps|) = \alpha|\ln \eps|$, meaning that $\tau_{\eps,\beta} \geqslant \alpha|\ln\eps|$.
Theorem~\ref{theo:general_conf} is now proved.

\subsection{Optimality of the bound}\label{sec:optimalityproof}

We now prove that the logarithmic bound is optimal conditionally to the existence of a proper configuration of point-vortices.
Let  $(\mathbf{x}_1^0,\ldots,\mathbf{x}_N^0)$ be pairwise distinct points on the sphere $\mathbb{S}^2$, intensities $\Gamma_1,\ldots,\Gamma_N$ satisfying~\eqref{Gauss:PV}, and $\big(t\mapsto \mathbf{x}_i(t)\big)_{1\leqslant i \leqslant N}$ be the solution to the point-vortex dynamics~\eqref{PVS}. We assume the following.
\begin{hypothesis}\label{hyp:instable}
    There exists a constant $\mu_\beta > 0$ and $\eps_0 > 0$ such that for all $\eps \in (0,\eps_0)$, there exist a time $\tau_{\eps,\beta}^* \leqslant \mu_\beta|\ln\eps|$ and a set of points $\mathbf{x}_1^\eps,\ldots,\mathbf{x}_N^\eps$
    such that the solution $\big(t \mapsto \mathbf{x}_i^\eps(t) \big)_{1\leqslant i \leqslant N}$ of the point-vortex dynamics~\eqref{PVS} with initial data $\big(\mathbf{x}_{i,0}^\eps\big)_ {1\leqslant i \leqslant N}$ and intensities $\Gamma_1,\ldots,\Gamma_N$ satisfies
\begin{equation*}
    \begin{cases}
        |\mathbf{x}_i^0 - \mathbf{x}_{i,0}^\eps|_{\mathbb{R}^3} \leqslant\tfrac{\eps}{2}, \\
        |\mathbf{x}_i(\tau_{\eps,\beta}^*) - \mathbf{x}_{i}^\eps(\tau_{\eps,\beta}^*)|_{\mathbb{R}^3} \geqslant 4\eps^\beta.
    \end{cases}
\end{equation*}
\end{hypothesis}
 We prove the following.
\begin{theo}\label{theo:conditional_optimal}
    Assuming Hypothesis~\ref{hyp:instable}, there exists $\beta_0 < 1/2$, $\eta \geqslant 2$ such that for any $\beta \in (\beta_0,1)$, there exists $\alpha_0 > 0$ such that for any $\eps > 0$ small enough, there exists $\zeta_0^\eps$ satisfying Hypothesis~\ref{hyp:zeta} such that
    \begin{equation*}
        \tau_{\eps,\beta} \leqslant \alpha_0 |\ln \eps|.
    \end{equation*}
\end{theo}

\begin{proof}
For $\eta$ to be chosen later, consider
\begin{equation*}
    \zeta_{0,i}^\eps = \frac{\Gamma_i}{\pi\varepsilon^{2\eta}}\Ind_{\mathcal{C} \left(\mathbf{x}_{i,0}^\eps,\eps^{\frac{\eta}{2}}\right)}.
\end{equation*}
In view of \eqref{areacap}, we have
$$\int_{\mathbb{S}^2} \zeta_{0,i}^\eps(\mathbf{y})\dd\boldsymbol{\sigma}(\mathbf{y})=\Gamma_i.$$
One can check by direct computation that the maps $\zeta_{0,i}^\eps$ satisfy Hypothesis~\eqref{hyp:zeta} provided $\eps$ is small enough. But in addition, it satisfies that
\begin{equation*}
    I_i^\eps(0) \leqslant C\eps^{\eta}.
\end{equation*}
Therefore, applying Lemma~\ref{lem:control_I} and integrating by Gronwall's Lemma, recalling relation~\eqref{maj:F_Lip}, we obtain that for every $t \leqslant \tau_{\eps,\beta}$,
\begin{equation*}
    I_i^\eps(t) \leqslant C\eps^{\eta} e^{2Dt}.
\end{equation*}
Then, using Lemma~\ref{lem:c}, we have that
    \begin{equation*}
        \left|\der{}{t} \Big(c_i^\eps(t) - \mathcal{F}_i\big(\vec{c}^{\,\eps}(t)\big)\Big)\right|_{\mathbb{R}^3} \leqslant C\eps^{\frac{\eta}{2}} e^{Dt}.
    \end{equation*}
Up to renaming $D$, assume that $D$ also bounds the Lipschitz norm of $\mathcal{F}$ over $\mathscr{D}_N$. Using the variant of Gronwall's Lemma~\ref{lem:gronwall}, we obtain that
\begin{align*}
    \Big|c_i^\eps(t) - \mathbf{x}_i^\eps(t)\Big|_{\mathbb{R}^3} & \leqslant  \left( \int_0^t C\eps^{\frac{\eta}{2}} e^{Ds} \dd s + |c_i^\eps(0) - \mathbf{x}_i^\eps(0)|_{\mathbb{R}^3}\right)e^{D t} \\
    & \leqslant C\left( \eps^{\frac{\eta}{2}} e^{Dt}+ \eps^{\eta}\right)e^{D t}
\end{align*}
and thus provided $\eps$ is small enough,
\begin{equation}\label{eq:majCinstable}
     \Big|c_i^\eps(t) - \mathbf{x}_i^\eps(t)\Big|_{\mathbb{R}^3}  \leqslant C \eps^{\frac{\eta}{2} -2 D t |\ln \eps|^{-1}}
\end{equation}
Recall that by hypothesis~\eqref{hyp:instable}, there exists $\tau_{\eps,\beta}^* \leqslant \mu_\beta |\ln \eps|$ such that
\begin{equation*}
    \Big|\mathbf{x}_i(\tau_{\eps,\beta}^*) - \mathbf{x}_i^\eps(\tau_{\eps,\beta}^*)\Big| \geqslant4\eps^\beta.
\end{equation*}
Let $\eta$ be large enough such that
\begin{equation*}
    \frac{\frac{\eta}{2}-\beta}{2D} > \mu_\beta.
\end{equation*}
Then,
\begin{equation*}
    \frac{\eta}{2} - 2D \tau_{\eps,\beta}^* |\ln \eps|^{-1} > \beta.
\end{equation*}
Assume now that $\tau_{\eps,\beta} \geqslant \tau_{\eps,\beta}^*$. We deduce from the previous relation and from \eqref{eq:majCinstable} that for $\eps$ small enough,
\begin{equation*}
    \Big|c_i^\eps(\tau_{\eps,\beta}^*) - \mathbf{x}_i^\eps(\tau_{\eps,\beta}^*)\Big| \leqslant \eps^\beta
\end{equation*}
and thus by triangular inequality,
\begin{equation*}
    \Big|c_i^\eps(\tau_{\eps,\beta}^*) - \mathbf{x}_i(\tau_{\eps,\beta}^*)\Big| \geqslant2\eps^\beta.
\end{equation*}
This implies that $\tau_{\eps,\beta} < \tau_{\eps,\beta}^*$, which is a contradiction, proving indeed that 
\begin{equation*}
     \tau_{\eps,\beta} < \tau_{\eps,\beta}^* \leqslant \mu_\beta |\ln \eps|,
\end{equation*}
which concludes the proof of Theorem~\ref{theo:conditional_optimal}
\end{proof}

\section{Power-law confinement result}\label{sec:confinement_ameliore}

   Similarly to what we did in Section~\ref{sec:optimalityproof}, we prove Theorem~\ref{theo:strong_conf} conditionnally to the existence of suitable point-vortex configurations.

    \subsection{Super-stability hypotheses}

    Recall the notation \eqref{BS-kernel}, and the fact that a priori the function $K_{\mathbb{S}^2}$ is well-defnied on
    $$\mathcal{D}_K^{\mathbb{S}^2}\triangleq\Big\{(\mathbf{x},\mathbf{y})\in\mathbb{S}^2\times\mathbb{S}^2\quad\textnormal{s.t.}\quad\mathbf{x}\neq\mathbf{y}\Big\}.$$
    Fixing $\mathbf{y}\in\mathbb{S}^2$, the partial application $\mathbf{x}\mapsto K_{\mathbb{S}^2}(\mathbf{x},\mathbf{y)}$ is an application from $\mathbb{S}^2\setminus\{\mathbf{y}\}$ into $\mathbb{R}^3$ that is smooth. Given $\mathbf{x}\in\mathbb{S}^2\setminus\{\mathbf{y}\},$ the associated tangent linear map is denoted
$$\Diff_1K_{\mathbb{S}^2}(\mathbf{x},\mathbf{y}):\begin{array}[t]{rcl}
     T_{\mathbf{x}}\mathbb{S}^2 & \rightarrow & \mathbb{R}^3\\
     \mathbf{z} & \mapsto & \Diff_1K_{\mathbb{S}^2}(\mathbf{x},\mathbf{y})[\mathbf{z}]. 
\end{array}$$
Observe that the notation \eqref{BS-kernel} also makes sense in the ambient Euclidean space so that actually $K_{\mathbb{S}^2}$ is well-defined on
$$\mathcal{D}_{K}^{\mathbb{R}^3}\triangleq\Big\{(x,y)\in\mathbb{R}^3\times\mathbb{R}^3\quad\textnormal{s.t.}\quad x\neq y\Big\}.$$
If one denote by $K_{\mathbb{R}^3}$ the extension, since $\mathbb{S}^2$ is a submanifold of $\mathbb{R}^3$, we have 
$$\Diff_1K_{\mathbb{S}^2}(\mathbf{x},\mathbf{y})=\Diff_1 K_{\mathbb{R}^3}(\mathbf{x},\mathbf{y})|_{T_{\mathbf{x}}\mathbb{S}^2}.$$
We then continue to keep the notation $K_{\mathbb{S}^2}$ for its $\R^3$ extension.
Let us introduce the following Hypothesis, similar to Hypothesis 3.1 of \cite{Donati_2025_Crystal} in the planar case, but somewhat relaxed to a weaker constraint, as discussed below.
    \begin{hypothesis}\label{hyp:supstable}
        Assume that  $(\mathbf{x}_1^0,\ldots,\mathbf{x}_N^0)$ is such that the solution of the point-vortex dynamics satisfies
        \begin{equation*}
            \forall t \geqslant 0,\quad \forall i \in \{1,\ldots,N\}, \quad \forall h \in T_{\mathbf{x}_i(t)}\mathbb{S}^2, \quad \sum_{j=1\atop j \neq i}^{N} \Gamma_j \Diff_1K_{\mathbb{S}^2}\big(\mathbf{x}_i(t),\mathbf{x}_j(t)\big)[h]\cdot h = 0 .
        \end{equation*}
    \end{hypothesis}
    Looking for instance to vortex crystal configurations as in \cite{Donati_2025_Crystal}, one can see that Hypothesis~\ref{hyp:supstable} is actually equivalent to Hypothesis~3.1 of \cite{Donati_2025_Crystal}, as the required equality can only happen if the first vector in the scalar product is vanishing. But on the sphere, due to the particular geometry, it is not the case, and this scalar product can vanish in a non-trivial way, as seen in the constructed example in Section~\ref{sec:config_stable}. We then introduce a second stability hypothesis, which, as will be observed in Section~\ref{sec:prop_stable}, is implied by the linear stability of the configuration, as required in \cite{Donati_2025_Crystal}.
    \begin{hypothesis}\label{hyp:stable}
        Assume that  $(\mathbf{x}_1^0,\ldots,\mathbf{x}_N^0)$ satisfies that for every $\eps > 0$ small enough, and every family of functions $\big(t\mapsto\mathbf{x}_i^\eps(t))_{1\leqslant i \leqslant N}$  defined on a time interval $[0,T_\eps]$, such that
        \begin{equation}\label{eq:hyp:stable1}
            \forall i \in \{1,\ldots,N\}, \quad |\mathbf{x}_i^0 - \mathbf{x}_{i}^\eps(0) |_{\mathbb{R}^3} \leqslant \eps
        \end{equation}
        and satisfying the existence of a constant $C$ such that
        \begin{equation}\label{eq:hyp:stable2}
            \forall i\in\{1,\ldots,N\},\quad\left| \der{}{t} \mathbf{x}_i(t) - \mathcal{F}_i\big(\mathbf{x}_1^\eps(t),\ldots,\mathbf{x}_N^\eps(t) \big)\right|_{\mathbb{R}^3} \leqslant C \eps,
        \end{equation}
        then this family of trajectories satisfies that for every $\beta < 1/2$ and every $\alpha < \beta$,
        \begin{equation*}
            \bar{\tau}_{\eps,\beta} \triangleq \inf \Big\{t \in [0,T_\eps]\quad \text{s.t.}\quad\exists i \in \{1,\ldots,N\},\quad |\mathbf{x}_i(t) - \mathbf{x}_i^\eps(t)|_{\mathbb{R}^3} = \tfrac{\eps^\beta}{2} \Big\} \geqslant \min(\eps^{-\alpha},T_\eps).
        \end{equation*}

\end{hypothesis}

\subsection{The conditional theorem}

    \begin{theo}\label{theo:strong_conf_conditional}\textbf{(Improved confinement time for special configurations)}\\
        There exists a choice of $N$, of $(\mathbf{x}_1,\ldots,\mathbf{x}_N)$ and intensities $\Gamma_1,\ldots,\Gamma_N$ satisfying \eqref{Gauss:PV} and \eqref{dist-min-points} as well as Hypotheses~\ref{hyp:supstable} and~\ref{hyp:stable}, such that for every $\beta<1/2$ there exists $\eps_0>0$ and $\alpha > 0$ such that for every $\eps \in (0,\eps_0)$, the solution $\zeta^\eps$ of \eqref{EulerAV} with initial condition $\zeta_0^\eps$ subjected to the Hypothesis \ref{hyp:zeta} near the points $(\mathbf{x}_i^0)_{1\leqslant i\leqslant N}$ satisfies
        \begin{equation*}
            \tau_{\eps,\beta} \geqslant \eps^{-\alpha}.
        \end{equation*}
    \end{theo}

    \paragraph{Proof of Theorem~\ref{theo:strong_conf_conditional}.}
    \text{~}
    
        The proof follows the same outline as that of Theorem~\ref{theo:general_conf}, except that we use Hypotheses~\ref{hyp:supstable} and~\ref{hyp:stable} to improve the estimates of the vorticity moments and obtain these estimates for longer times. We start by observing how Hypothesis~\ref{hyp:supstable} allows to bound the constant $D_\eps$ as follows.
        \begin{lem}\label{lem:D'eps}
            There exists a constant $D$ such that for every $\eps$ small enough,
            \begin{equation*}
                D_\eps \leqslant D \eps^\beta.
            \end{equation*}
        \end{lem}
        
    \begin{proof}
        Let $i\in\{1,\ldots,N\}$ and $t\leqslant\tau_{\varepsilon,\beta}$. Take $x,y\in B\big(\mathbf{x}_i(t),\varepsilon^{\beta}\big)$, and let $h = x-y$.
        By definition,
        $$F_i^{\varepsilon}(t,x)-F_i^{\varepsilon}(t,y)=\frac{1}{2\pi}\sum_{j=1\atop j\neq i}^{N}\int_{\mathbb{S}^2}\big(K_{\mathbb{S}^2}(x,\mathbf{z})-K_{\mathbb{S}^2}(y,\mathbf{z})\big)\zeta_j^{\varepsilon}(t,\mathbf{z})d\boldsymbol{\sigma}(\mathbf{z}) +\gamma\,\mathbf{e}_3\wedge h.$$
        Then, fix $j\in\{1,\ldots,N\}\setminus\{i\}$ and $\mathbf{z}\in\supp\big(\zeta_{j}^{\varepsilon}(t,\cdot)\big)\subset B\big(\mathbf{x}_j(t),\varepsilon^{\beta}\big)$. We compute by Taylor expanding the last expression that
        \begin{align*}
            K_{\mathbb{S}^2}(x,\mathbf{z})-K_{\mathbb{S}^2}(y,\mathbf{z})&=\Diff_1K_{\mathbb{S}^2}(y,\mathbf{z})[h]+\frac{1}{2}\int_{0}^{1}(1-\tau)\Diff_1^2K_{\mathbb{S}^2}(y+\tau h,\mathbf{z})[h,h]\dd \tau\\
            &=\Diff_1K_{\mathbb{S}^2}\big(\mathbf{x}_i(t),\mathbf{z}\big)[h]+\Big(\Diff_1K_{\mathbb{S}^2}(y,\mathbf{z})-\Diff_1K_{\mathbb{S}^2}\big(\mathbf{x}_i(t),\mathbf{z}\big)\Big)[h]\\
            &\quad+\frac{1}{2}\int_{0}^{1}(1-\tau)\Diff_1^2 K_{\mathbb{S}^2}(y+\tau h,\mathbf{z})[h,h]\dd \tau.
        \end{align*}
        By the Mean Value Theorem, we get that
        \begin{align*}
            \left|\Big(\Diff_1K_{\mathbb{S}^2}(y,\mathbf{z})-\Diff_1K_{\mathbb{S}^2}\big(\mathbf{x}_i(t),\mathbf{z}\big)\Big)[h]\right|_{\mathbb{R}^3}&\leqslant C|y-\mathbf{x}_i(t)|_{\mathbb{R}^3}|h|_{\mathbb{R}^3}\\
            &\leqslant C\varepsilon^{\beta}|h|_{\mathbb{R}^3}.
        \end{align*}
        Then, we compute that
        \begin{align*}
            \left|\frac{1}{2}\int_{0}^{1}(1-\tau)\Diff_1^2 K_{\mathbb{S}^2}(y+\tau h,\mathbf{z})[h,h]\dd \tau\right|&\leqslant C|h|_{\mathbb{R}^3}^2\\
        &\leqslant C\varepsilon^{\beta}|h|_{\mathbb{R}^3}.
        \end{align*}
        Besides,
        $$\Diff_1K_{\mathbb{S}^2}\big(\mathbf{x}_i(t),\mathbf{z}\big)(h)=\Diff_1K_{\mathbb{S}^2}\big(\mathbf{x}_i(t),\mathbf{x}_j(t)\big)h+\Big(\Diff_1K_{\mathbb{S}^2}\big(\mathbf{x}_i(t),\mathbf{z}\big)-\Diff_1K_{\mathbb{S}^2}\big(\mathbf{x}_i(t),\mathbf{x}_j(t)\big)\Big)h.$$
        By the Mean Value Theorem,
        \begin{align*}
            \left|\Big(\Diff_1K_{\mathbb{S}^2}\big(\mathbf{x}_i(t),\mathbf{z}\big)-\Diff_1K_{\mathbb{S}^2}\big(\mathbf{x}_i(t),\mathbf{x}_j(t)\big)\Big)[h]\right|&\leqslant C\left|\mathbf{z}-\mathbf{x}_j(t)\right|_{\mathbb{R}^3}|h|_{\mathbb{R}^3}\\
            &\leqslant C\varepsilon^{\beta}|h|_{\mathbb{R}^3}.
        \end{align*}
        Combining the foregoing calculations yields
        
        \begin{equation}\label{eq:devF}
            F_i^{\varepsilon}(t,x)-F_i^{\varepsilon}(t,y) =   \frac{1}{2\pi}\sum_{j=1\atop j \neq i}^{N} \Gamma_j \Diff_1K_{\mathbb{S}^2}(\mathbf{x}_i(t),\mathbf{x}_j(t))[h] +\gamma \mathbf{e}_3 \wedge h + \mathcal{O}(\eps^\beta|h|_{\mathbb{R}^3}).
        \end{equation} 
        Let $\Pi_{\mathbf{x}_i(t)}$ be the $\R^3$ orthogonal projection onto the tangent plane $T_{\mathbf{x}_i(t)}\mathbb{S}^2$. Then, we have that
        \begin{multline*}
            \Diff_1K_{\mathbb{S}^2}(\mathbf{x}_i(t),\mathbf{x}_j(t))[h] \cdot h = \Diff_1K_{\mathbb{S}^2}(\mathbf{x}_i(t),\mathbf{x}_j(t))[\Pi_{\mathbf{x}_i(t)}h] \cdot \Pi_{\mathbf{x}_i(t)}h \\
            + \Diff_1K_{\mathbb{S}^2}(\mathbf{x}_i(t),\mathbf{x}_j(t))[h- \Pi_{\mathbf{x}_i(t)}h] \cdot h + \Diff_1K_{\mathbb{S}^2}(\mathbf{x}_i(t),\mathbf{x}_j(t))[\Pi_{\mathbf{x}_i(t)}h] \cdot (h-\Pi_{\mathbf{x}_i(t)}h),
        \end{multline*}
        but we now claim that the two last terms are vanishing, so that in the end,
        \begin{equation}\label{eq:projetee}
            \Diff_1K_{\mathbb{S}^2}(\mathbf{x}_i(t),\mathbf{x}_j(t))[h] \cdot h = \Diff_1K_{\mathbb{S}^2}(\mathbf{x}_i(t),\mathbf{x}_j(t))[\Pi_{\mathbf{x}_i(t)}h] \cdot \Pi_{\mathbf{x}_i(t)}h.
        \end{equation}
        Indeed, by using that $K_{\mathbb{S}^2}(\mathbf{x}_i(t),\mathbf{y}_i(t)) \in T_{\mathbf{x}_i(t)}\mathbb{S}^2$, so the differential $\Diff_1K_{\mathbb{S}^2}(\mathbf{x}_i(t),\mathbf{x}_j(t))$ maps $T_{\mathbf{x}_i(t)}\mathbb{S}^2$ to itself, and thus
        \begin{equation*}
            \Diff_1K_{\mathbb{S}^2}(\mathbf{x}_i(t),\mathbf{x}_j(t))[\Pi_{\mathbf{x}_i(t)}h] \cdot (h-\Pi_{\mathbf{x}_i(t)}h) = 0.
        \end{equation*}
        In addition, the second term also vanishes, which can be seen by direct computation: using the explicit formula of the differential matrix, obtained in details at relation~\eqref{sec:differential}, we have, since $h-\Pi_{\mathbf{x}_i(t)}h = \mathbf{x}_i(t) (h \cdot \mathbf{x}_i(t) )$, that
        \begin{align*}
            \Diff_1K_{\mathbb{S}^2}(\mathbf{x}_i(t),\mathbf{x}_j(t))[h- \Pi_{\mathbf{x}_i(t)}h] \cdot h & = h \cdot \mathbf{x}_i(t) \frac{\mathbf{x}_i(t) \wedge\mathbf{x}_j(t)}{|\mathbf{x}_i(t)-\mathbf{x}_j(t)|_{\R^3}^2} \left( 1 - 2 \frac{\mathbf{x}_i(t) \cdot (\mathbf{x}_i(t) - \mathbf{x}_j(t))}{|\mathbf{x}_i(t)-\mathbf{x}_j(t)|_{\R^3}^2} \right) \\
            & =  h \cdot \mathbf{x}_i(t) \frac{\mathbf{x}_i(t) \wedge\mathbf{x}_j(t)}{|\mathbf{x}_i(t)-\mathbf{x}_j(t)|_{\R^3}^2} \left( 1 - 2 \frac{1 -\mathbf{x}_i(t) \cdot \mathbf{x}_j(t)}{2-2 \mathbf{x}_i(t) \cdot \mathbf{x}_j(t)} \right) \\
            & = 0,
        \end{align*}
        so that the relation~\eqref{eq:projetee} is now proved. Taking the scalar product of equation~\eqref{eq:devF} with $h = x-y$, then using relation~\eqref{eq:projetee} we get that
        \begin{equation}\label{eq:devF}
            \big(F_i^{\varepsilon}(t,x)-F_i^{\varepsilon}(t,y)\big)\cdot (x-y) =   \frac{1}{2\pi}\sum_{j=1\atop j \neq i}^{N} \Gamma_j \Diff_1K_{\mathbb{S}^2}(\mathbf{x}_i(t),\mathbf{x}_j(t))[\Pi_{\mathbf{x}_i(t)}h]\cdot \Pi_{\mathbf{x}_i(t)}h + \mathcal{O}(\eps^\beta|h|_{\mathbb{R}^3}),
        \end{equation}
        and using now Hypothesis~\ref{hyp:supstable} to $h' = \Pi_{\mathbf{x}_i(t)}h \in T_{\mathbf{x}_i(t)}\mathbb{S}^2$, we proved that
        \begin{equation}\label{eq:Flip_a}
            \Big|\big(F_i^{\varepsilon}(t,x)-F_i^{\varepsilon}(t,y)\big)\cdot (x-y) \Big| \leqslant C\eps^\beta |x-y|_{\mathbb{R}^3}^2.
        \end{equation}
        Coming back to~\eqref{eq:devF}, integrating in $\mathbf{y}$ against the measure $\frac{1}{\Gamma_i}\zeta_i^\eps(t,\cdot)$, then taking the scalar product with $x-c_i^\eps(t)$, we have that
        \begin{align*}
            &\left(F_i^\eps(t,x) -\int_{\mathbb{S}^2} F_i^\eps(t,\mathbf{y}) \frac{\zeta_i^\eps(t,\mathbf{y})}{\Gamma_i}\dd\boldsymbol{\sigma}(\mathbf{y})\right) \cdot(x-c_i^\eps(t)\big) \\
            &= \frac{1}{2\pi}\sum_{j=1\atop j \neq i}^{N} \Gamma_j \Diff_1K_{\mathbb{S}^2}\big(\mathbf{x}_i(t),\mathbf{x}_j(t)\big)[x-c_i^\eps(t)]\cdot\big(x-c_i^\eps(t)\big) 
            + \mathcal{O}\left(\eps^\beta \int_{\mathbb{S}^2}|x-\mathbf{y}|_{\mathbb{R}^3}\frac{\zeta_i^\eps(t,\mathbf{y})}{\Gamma_i}\dd\boldsymbol{\sigma}(\mathbf{y})|x-c_i^\eps(t)|_{\mathbb{R}^3}\right).
        \end{align*}
       Using again relation~\eqref{eq:projetee} with $h = x-c_i^\eps(t)$, then Hypothesis~\ref{hyp:supstable}, we conclude that
        \begin{equation}\label{eq:Flip_b}
            \left| \left(F_i^\eps(t,x) -\int_{\mathbb{S}^2} F_i^\eps(t,\mathbf{y}) \frac{\zeta_i^\eps(t,\mathbf{y})}{\Gamma_i}\dd\boldsymbol{\sigma}(\mathbf{y})\right) \cdot\big(x-c_i^\eps(t)\big)\right| \leqslant C\eps^\beta |x-c_i^\eps(t)|_{\mathbb{R}^3} \int_{\mathbb{S}^2} |x-\mathbf{y}|_{\mathbb{R}^3}\frac{\zeta_i^\eps(t,\mathbf{y})}{\Gamma_i}\dd\boldsymbol{\sigma}(\mathbf{y}).
        \end{equation}
        Up to renaming $D$, equations~\eqref{eq:Flip_a} and~\eqref{eq:Flip_b} prove, referring to Definition~\ref{def:Dprime}, that $D_\eps \leqslant D\eps^\beta$.
    \end{proof}
    Then we bound the vorticity moments. Let us denote by $\alpha > 0$ a positive number to be chosen later.
    \begin{lem}\label{lem:moments_improved}
        There exists a constant $C$ such that for $\eps>0$ small enough, for any $i \in \{1,\ldots,N\}$ and any $t \leqslant \min(\tau_{\eps,\beta},\eps^{-\alpha})$,
    \begin{equation*}
        I_i^\eps(t) \leqslant 5 \eps^{2}, \qquad |c_i^\eps(t)-\mathbf{x}_i(t)|_{\mathbb{R}^3} \leqslant \tfrac{\eps^\beta}{2}, \qquad m_{n,i}^\eps(t) \leqslant C_n \eps^{(2-\alpha)n}.
    \end{equation*}
    \end{lem}
    \begin{proof}
    $\blacktriangleright$ \textit{First estimate :} using Lemma~\ref{lem:control_I}, with the bound on $D_\eps$ given by Lemma~\ref{lem:D'eps}, we have that for all $i \in \{1,\ldots,N\}$ and for any $t\leqslant \tau_{\eps,\beta}$,
        \begin{equation*}
            \left| \der{}{t} I_i^{\varepsilon}(t)\right|\leqslant 2D \eps^\beta I_i^{\varepsilon}(t),
        \end{equation*}
        which we integrate, recalling Lemma~\ref{lem:initestim}, as
        \begin{equation*}
            I_i^\eps(t) \leqslant 4 \eps^2 e^{2D \eps^\beta t}.
        \end{equation*}
        Therefore, provided $\eps$ is small enough and $\alpha < \beta$, for every $t \leqslant \min(\tau_{\eps,\beta},\eps^{-\alpha})$, we have that
        \begin{equation}\label{eq:Iimproved}
            I_i^\eps(t) \leqslant 5 \eps^2.
        \end{equation}
        $\blacktriangleright$ \textit{Second estimate :} this one differs significantly from Lemma~\ref{lem:moments_generaux}. We use Lemma~\ref{lem:c} to get that
    \begin{equation*}
        \left|\der{}{t}c_i^\eps(t) - \mathcal{F}_i\big(\vec{c}^{\,\eps}(t)\big)\right|_{\mathbb{R}^3} \leqslant C \eps ,
    \end{equation*}
    and recalling from Lemma~\ref{lem:initestim} that $|c_i^\eps(t) - \mathbf{x}_i^0|_{\mathbb{R}^3} \leqslant \eps $, we deduce that the trajectories $(t \mapsto c_i^\eps(t) \big)_{1 \leqslant i \leqslant N}$ defined until $T_\eps = \min(\tau_{\eps,\beta},\eps^{-\alpha})$ satisfy relations~\eqref{eq:hyp:stable1} and~\eqref{eq:hyp:stable2}, and thus from Hypothesis~\ref{hyp:stable}, we have, at the condition that $\alpha < \beta/2$, that
    \begin{equation*}
        \forall t \leqslant \min(\tau_{\eps,\beta},\eps^{-\alpha}), \quad \left|\der{}{t}c_i^\eps(t) - x_i(t)\right|_{\mathbb{R}^3} \leqslant \tfrac{\eps^\beta}{2}\cdot
    \end{equation*}
    $\blacktriangleright$ \textit{Third estimate :} plugging estimate~\eqref{eq:Iimproved} and Lemma~\ref{lem:D'eps} into Lemma~\ref{lem:4n}, we have that
        \begin{equation*}
        \der{}{t} m_{n,i}^\eps(t) \leqslant C n^2\big(m_{n,i}^{\eps}(t)\big)^\frac{n-1}{n} \left(\eps^2 + \eps^\beta  \left( \big(m_{n,i}^{\eps}(t)\big)^\frac{1}{n} +\eps^{1+3\beta} \right) \right),
    \end{equation*}
    provided $\eps$ is small enough. If $1+4\beta > 2$, namely $\beta > 1/4$ which we recall is an assumption we can make as proving Theorem~\ref{theo:strong_conf_conditional} for a certain $\beta < 1/2$ also proves it for every $\beta'<\beta$, then we are reduced to
    \begin{equation*}
        \der{}{t} m_{n,i}^\eps(t) \leqslant \eps^\beta C n^2\big(m_{n,i}^{\eps}(t)\big)^\frac{n-1}{n} \left(\eps^{2-\beta} +  \big(m_{n,i}^{\eps}(t)\big)^\frac{1}{n}  \right).
    \end{equation*}
    Integrating this inequality using Lemma~\ref{lem:4nODE} and Lemma~\ref{lem:initestim}, we obtain that
    \begin{equation*}
        m_{n,i}^\eps(t) \leqslant \left( -\eps^{2-\beta} + \left( \eps^{2-\beta} + 16\eps^4 \right) e^{C n\eps^\beta t} \right)^n.
    \end{equation*}
    For $\eps$ small enough, for $\alpha < \beta$ and $t \leqslant \eps^{-\alpha}$, we thus have that
    \begin{equation*}
        e^{Cn \eps^\beta t} \leqslant 1+2Cn \eps^{\beta-\alpha}
    \end{equation*}
    and therefore
    \begin{equation*}
        m_{n,i}^\eps(t) \leqslant (2C n \eps^{2-\alpha})^n.
    \end{equation*}
    
    \end{proof}

We now obtain the corollary of the estimate of the higher order moments in terms of a control on $\mathtt{m}_{t,i}^{\varepsilon}$.

\begin{lem}\label{lem:4n_ext_improved}
    For every $\nu>0$, provided $\alpha<2-4\beta$, there exists $\eps_0 > 0$ such that for all $\eps \in (0,\eps_0)$, and for any $t\leqslant  \min(\tau_{\eps,\beta},\eps^{-\alpha})$ and any $r \geqslant \tfrac{\eps^{\beta}}{4}$,
    \begin{equation}\label{estim outside_improved}
    \max_{1\leqslant i\leqslant N}\mathtt{m}_{t,i}^\eps(r) \leqslant \frac{\eps^{5+\nu}}{r^6},
\end{equation}
where $\mathtt{m}_{t,i}^{\varepsilon}(r)$ is defined in \eqref{defttm}.
\end{lem}
\begin{proof}
    Fix $t \leqslant \min\big(\tau_{\varepsilon,\beta},\eps^{-\alpha}\big)$. Invoking Lemma \ref{lem:moments_improved}, for any $r\geqslant\frac{\varepsilon^{\beta}}{4}$ and any $\nu>0$, we have
\begin{align*}
    \frac{1}{\Gamma_i}\int_{\mathbb{S}^2 \setminus \mathcal{C}(c_i^\eps(t),r)} \zeta_i^\eps(t,\mathbf{x}) \dd \boldsymbol{\sigma}(\mathbf{x}) & =  \frac{1}{\Gamma_i}\int_{\mathbb{S}^2 \setminus \mathcal{C}(c_i^\eps(t),r)} \frac{|\mathbf{x} - c_i^\eps(t)|_{\mathbb{R}^3}^{4n}}{|\mathbf{x} - c_i^\eps(t)|_{\mathbb{R}^3}^{4n}} \zeta_i^\eps(t,\mathbf{x}) \dd \boldsymbol{\sigma}(\mathbf{x}) \\
         & \leqslant \frac{m_{n,i}^\eps(t)}{r^{4n}} \\
         & \leqslant \frac{C_n \eps^{n(2-\alpha)}}{r^{4n}}\\
         & \leqslant C_n\frac{\eps^{5+\nu}}{r^6} \frac{\eps^{(2-\alpha)n-5-\nu}}{\left(\frac{\eps^{\beta}}{4}\right)^{4n-6}} \\
         & \leqslant C_n 4^{4n-6} \frac{\eps^{5+\nu}}{r^6} \eps^{(2-\alpha-4\beta)n - 5 - \nu + 6\beta}.
\end{align*}
Taking $n = \left\lfloor \frac{5+\nu}{2-4\beta-\alpha}\right\rfloor +2$ we get that there exists $\eps_0$ depending only on $\beta$, $\nu$ and $\alpha$ such that for every $\eps \in (0,\eps_0)$, \eqref{estim outside_improved} holds.
\end{proof}
We are now ready to prove Theorem~\ref{theo:strong_conf_conditional} in a similar way to the end of the proof of Theorem~\ref{theo:general_conf}.

\subsubsection*{Conclusion of the proof of Theorem \ref{theo:strong_conf_conditional}}

Fix $i\in\{1,\ldots, N\}.$ For every $\eps > 0$ and $t \geqslant0$, recall the notations of Section~\ref{sec:growth} concerning the existence of a point $X_i^\eps$ and a trajectory $s\mapsto X_{t,i}^\eps(s)$ such that relations~\eqref{formula:radmax} and~\eqref{def:flowt} hold. Applying Lemmas~\ref{lem:D'eps} and~\ref{lem:moments_improved} to Lemma~\ref{lem:vitesse_radiale}, we have that for all $t \leqslant \min(\tau_{\eps,\beta},\eps^{-\alpha})$,
\begin{equation*}
    \der{}{s}\left|X_{t,i}^{\varepsilon}(s)-c_i^\eps(s)\right|_{\mathbb{R}^3}\Big|_{s=t} \leqslant D\eps^\beta R_i^{\varepsilon}(t) + C_2\frac{\eps^{2}}{(R_i^\eps(t))^3} + \left(\frac{M\eps^{-\eta}|\Gamma_i|}{\pi}\mathtt{m}_{t,i}^\eps\left(\frac{R_i^{\varepsilon}(t)}{2}\right)\right)^\frac{1}{2}.
\end{equation*}
Let $f_i$ be the solution of the ODE:
\begin{equation*}
    f_i'(t) =  2 D \eps^\beta f_i(t) + 2 C_2\frac{\eps^{2}}{f_i^3(t)} + 2\left(\frac{M\eps^{-\eta}|\Gamma_i|}{\pi}\mathtt{m}_{t,i}^\eps\left(\frac{f_i(t)}{2}\right)\right)^\frac{1}{2},
\end{equation*}
with initial data 
\begin{equation*}
    f_i(0) = 2 R_i^{\varepsilon}(0).
\end{equation*} 
Then by the same argument that was performed in Section~\ref{sec:proofgenerale}, we have that
\begin{equation*}
    \forall t\leqslant\min(\tau_{\varepsilon,\beta},\varepsilon^{-\alpha}), \quad f_i(t) > R_i^{\varepsilon}(t).
\end{equation*}

Let us now prove that there does not exist an index $i \in \{1,\ldots,N\}$ and a time $t_2 \leqslant \eps^{-\alpha}$ such that $f_i(t_2) \geqslant \eps^{\beta}/2$. We proceed by contradiction: assume there exists such time. Let $t_1$ be the last time prior to $t_2$ such that for every $t \in [t_1,t_2]$, we have that $\eps^{\beta}/4 \leqslant f_i(t) \leqslant \eps^{\beta}/2$.
Since $f_i(t)/2 \geqslant \eps^{\beta}/8$, we can apply Lemma~\ref{lem:4n_ext_improved} with $r = f_i(t)$ and any positive $\nu$ such that $\nu \geqslant \eta +  6\beta -3$ to get that for all $t \in [t_1,t_2]$,
\begin{align*}
    f_i'(t) & \leqslant  2 D \eps^\beta f_i(t) + 2 C_2\frac{\eps^{2}}{f_i^3(t)} + 2\left(\frac{M\eps^{-\eta}|\Gamma_i|}{\pi}2^6\frac{\eps^{5+\nu}}{f_i^6(t)}\right)^\frac{1}{2} \\
    & \leqslant  2 D \eps^\beta f_i(t) + 2 C_2\frac{\eps^{2}}{f_i^3(t)} + C\eps^{\tfrac{1}{2}(5+\nu-\eta-6\beta)} \\
    & \leqslant C \left(\eps^\beta f_i(t) +\frac{\eps^{2}}{f_i^3(t)}\right),
\end{align*}
where in the last inequality we used that $\tfrac{1}{2}(5+\nu-\eta-6\beta) > \beta$ and thus $\eps^{\tfrac{1}{2}(5+\nu-\eta-6\beta)} \ll \eps^{2\beta} \leqslant \eps^\beta f_i(t)$ as $\eps \to 0$.
Multiplying by $f_i^3(t)$ gives that
\begin{equation*}
    (f_i^4)'(t) \leqslant C \eps^\beta( f_i^4(t) + \eps^{2-\beta}).
\end{equation*}
This in turns gives that
\begin{equation*}
    f_i^4(t_2) \leqslant f_i^4(t_1) e^{C\eps^\beta(t_2-t_1)} + \eps^{2-\beta} (e^{C\eps^\beta(t_2-t_1)}-1) \leqslant \frac{1}{4^4}\eps^{4\beta} e^{C\eps^{\beta-\alpha}} + \eps^{2-\beta}2C\eps^{\beta-\alpha}.
\end{equation*}
Hence, provided $2-\alpha > 4\beta$, and $\beta > \alpha$, for $\eps$ small enough,
\begin{equation*}
    f_i(t_2) < \tfrac{\eps^\beta}{2},
\end{equation*}
which is a contradiction. No such time $t_2$ exists for any index $i$, and thus for every $i \in\{1,\ldots,N\}$, for every $t \leqslant \min(\tau_{\eps,\beta},\eps^{-\alpha}),$ $R_i^{\varepsilon}(t) < f_i(t) < \tfrac{\eps^\beta}{2}$. Therefore, for every such $t$ and $i$, for every point $\mathbf{x} \in \supp \big(\zeta_i^\eps(t,\cdot)\big)$, we have that
\begin{equation*}
    |\mathbf{x}-\mathbf{x}_i(t)|_{\mathbb{R}^3} \leqslant |\mathbf{x}-c_i^\eps(t)|_{\mathbb{R}^3} + |\mathbf{x}_i(t) - c_i^\eps(t)|_{\mathbb{R}^3} \leqslant R_i^{\varepsilon}(t) + |\mathbf{x}_i(t) - c_i^\eps(t)|_{\mathbb{R}^3} \leqslant \frac{3}{4}\eps^\beta,
\end{equation*}
where we used Lemma~\ref{lem:moments_improved} provided $\eps$ and $\alpha$ is small enough. Applying this last equality in $\tau_{\eps,\beta}$ would then be a contradiction as by definition of $\tau_{\eps,\beta}$, there exists $i_0\in\{1,\ldots,N\}$ and $\overline{\mathbf{x}} \in \supp \big(\zeta_{i_0}^\eps(\tau_{\eps,\beta},\cdot)\big)$ such that 
\begin{equation*}
    |\overline{\mathbf{x}}-\mathbf{x}_{i_0}(\tau_{\eps,\beta})|_{\mathbb{R}^3} = \eps^\beta.
\end{equation*}
Therefore, this inequality cannot be applied in $\tau_{\eps,\beta}$, and thus $\tau_{\eps,\beta} \geqslant \eps^{-\alpha}$.
Theorem~\ref{theo:strong_conf_conditional} is now proved.

\section{Existence of point-vortex configurations leading to each stability hypothesis}\label{sec:configurations}
In this section, we construct configurations satisfying Hypothesis~\ref{hyp:instable}, as well as Hypotheses~\ref{hyp:supstable} and~\ref{hyp:stable} and conclude that the conditional Theorem~\ref{theo:conditional_optimal} and Theorem~\ref{theo:strong_conf_conditional} prove Theorem~\ref{theo:optimal} and Theorem~\ref{theo:strong_conf}.

We start with the sufficient conditions to satisfy Hypotheses~\ref{hyp:stable} and~\ref{hyp:instable} in terms of linear stability and instability.
\subsection{Link between Hypotheses \ref{hyp:instable}-\ref{hyp:stable} and linear stability}
    \subsubsection{Linear stability or neutrality implies Hypothesis~\ref{hyp:stable}}\label{sec:prop_stable}
    Let $\mathbf{X^0} = (\mathbf{x}_1^0,\ldots,\mathbf{x}_N^0)$ be a rotative relative equilibrium of the point vortex dynamics on the rotating sphere, which is aligned with the sphere axis of rotation, namely that the solution $t\mapsto \mathbf{X}(t)$ satisfies that for all $i \in \{1,\ldots,N\}$,
    \begin{equation*}
        \mathbf{x}_i(t) = \mathcal{R}(\Omega t) \mathbf{x}_i^0,
    \end{equation*}
    for some $\Omega \in \R$, where $\mathcal{R}(\theta)$ is defined at relation~\eqref{def:Rtheta}, a fact that we write
    \begin{equation}\label{eq:VC}
        \mathbf{X}(t) = \mathcal{\mathbf{R}}(\Omega t) \mathbf{X}^0,
    \end{equation}
    with the appropriate definition of the multidimensional rotation $\mathcal{\mathbf{R}}(\theta)$.
    Then, we have the following.
    \begin{prop}\label{prop:stable_implique_hypstable} 
        If $t\mapsto \mathbf{X}(t)$ is a vortex crystal solution aligned with the rotation axis of the sphere, namely satisfies~\eqref{eq:VC} for some $\nu \in \R$, and that $\Diff \mathcal{F}(\mathbf{X}^0)$ satisfies that for every $H \in \mathbb{T}_{\mathbf{x}_1^0} \mathbb{S}^2 \times \ldots \times \mathbb{T}_{\mathbf{x}_N^0} \mathbb{S}^2$ that
        \begin{equation}\label{eq:dissipatif}
            \Diff \mathcal{F}(\mathbf{X}^0)[H]\cdot H \leqslant 0,
        \end{equation}
        then $\mathbf{X}^0$ satisfies Hypothesis~\ref{hyp:stable}.
    \end{prop}
    \begin{proof}
    First, let us notice that since $t\mapsto \mathbf{X}(t)$ is a vortex crystal solution, and that relation~\eqref{eq:dissipatif} is invariant by the action of the rotations $\mathcal{\mathbf{R}}(\theta)$, then for all $t\geqslant 0$,
    \begin{equation}\label{eq:dissipatif2}
            \Diff \mathcal{F}\big(\mathbf{X}(t)\big)[H]\cdot H \leqslant 0.
        \end{equation}
        The interested reader can also see this fact through the formula \eqref{rep-diffF} given later. 
        Throughout the proof, we denote
        $$|\mathbf{X}|_{N}^2\triangleq\sum_{i=1}^{N}|\mathbf{x}_i|_{\mathbb{R}^3}^2.$$
        Let us compute
        \begin{align*}
            \der{}{t} |\mathbf{X}^\eps(t)-\mathbf{X}(t)|_N^2 & = 2\der{}{t} \big(\mathbf{X}^\eps(t)-\mathbf{X}(t) \big)\cdot \big(\mathbf{X}^\eps(t)-\mathbf{X}(t)\big) \\
            & = 2\Big(\mathcal{F}\big(\mathbf{X}^\eps(t)\big) - \mathcal{F}\big(\mathbf{X}(t)\big)\Big) \cdot \big(\mathbf{X}^\eps(t)-\mathbf{X}(t)\big).
        \end{align*}
        We then split the solution 
        $$\mathbf{X}^\eps(t) = \mathbf{X}^\eps_\perp(t) + \mathbf{X}^\eps_\top(t),\qquad with \qquad\mathbf{X}^\eps_\top-\mathbf{X}(t) \in T_{\mathbf{x}_1^\eps(t)}\mathbb{S}^2 \times\ldots\times T_{\mathbf{x}_N^\eps(t)}\mathbb{S}^2,$$ 
        so that
        \begin{align*}
            \der{}{t} |\mathbf{X}^\eps(t)-\mathbf{X}(t)|_N^2 & = 2\Big(\mathcal{F}\big(\mathbf{X}^\eps_\top(t)\big) - \mathcal{F}\big(\mathbf{X}(t)\big)\Big) \cdot \big(\mathbf{X}_\top^\eps(t)-\mathbf{X}(t)\big) \\
            & \quad+2\Big(\mathcal{F}\big(\mathbf{X}^\eps(t)\big)-\mathcal{F}\big(\mathbf{X}(t)\big)\Big) \cdot \mathbf{X}_\perp^\eps(t)\\
            & \quad+2\Big(\mathcal{F}\big(\mathbf{X}^\eps(t)\big)-\mathcal{F}\big(\mathbf{X}^\eps_\top(t)\big)\Big) \cdot \big(\mathbf{X}_\top^\eps(t)-\mathbf{X}(t)\big).
        \end{align*}
        Using relation~\eqref{eq:dissipatif2}, we find
        \begin{align*}
            \Big(\mathcal{F}\big(\mathbf{X}^\eps_\top(t)\big)-\mathcal{F}\big(\mathbf{X}(t)\big)\Big)\cdot \big(\mathbf{X}_\top^\eps(t)-\mathbf{X}(t)\big) & = \Diff \mathcal{F}\big(\mathbf{X}(t)\big)[\mathbf{X}_\top^\eps(t)-\mathbf{X}(t)]\cdot\big(\mathbf{X}_\top^\eps(t)-\mathbf{X}(t)\big)+C|\mathbf{X}_\top^\eps(t)-\mathbf{X}(t)|_N^3\\
            & \leqslant C|\mathbf{X}_\top^\eps(t)-\mathbf{X}(t)|_N^3.
        \end{align*}
        Then, by Lipschitz and Cauchy-Schwarz estimates, we get
        \begin{align*}
            \der{}{t} |\mathbf{X}^\eps(t)-\mathbf{X}(t)|_N^2 & \leqslant  C|\mathbf{X}^\eps(t)-\mathbf{X}(t)|_N|\mathbf{X}^\eps_\perp(t)|_N + C |\mathbf{X}^\eps(t)-\mathbf{X}^\eps_\top(t)|_N|\mathbf{X}^\eps_\top(t)-\mathbf{X}(t)|_N + C|\mathbf{X}_\top^\eps(t)-\mathbf{X}(t)|_N^3.
        \end{align*}
        Now recall that $\mathbf{X}_{\perp}^{\varepsilon}(t)=\mathbf{X}^{\varepsilon}(t)-\mathbf{X}_{\top}^{\varepsilon}(t)$ and that, by definition of the orthogonal projection, we have
        $$|\mathbf{X}_{\top}^{\varepsilon}(t)-\mathbf{X}(t)|_N\leqslant|\mathbf{X}^{\varepsilon}(t)-\mathbf{X}(t)|_N.$$
        With this in hand, we obtain
\begin{align*}
    \der{}{t} |\mathbf{X}^\eps(t)-\mathbf{X}(t)|_N^2\leqslant C|\mathbf{X}^\eps(t)-\mathbf{X}(t)|_N|\mathbf{X}^\eps_\perp(t)|_N + C|\mathbf{X}_\top^\eps(t)-\mathbf{X}(t)|_N^3.
\end{align*}
Additionally, since for all $i \in \{1,\ldots,N\}$, $\mathbf{x}_i^\eps(t),\mathbf{x}_i(t) \in \mathbb{S}^2$, we have by virtue of \eqref{norm-diff-carre} that
        \begin{equation*}
            |\mathbf{x}_i^\eps(t) - \mathbf{x}_i(t) |_{\mathbb{R}^3}^2 = 2\big(1-\mathbf{x}_i^\eps(t)\cdot \mathbf{x}_i(t)\big),
        \end{equation*}
        and by definition, $\mathbf{x}_i^\eps(t) = \mathbf{x}_{i,\bot}^\eps(t)+\mathbf{x}_{i,\top}^\eps(t)$, with 
        $$\big(\mathbf{x}_{i,\top}^\eps(t)-\mathbf{x}_i(t) \big) \cdot \mathbf{x}_i(t) = 0,\qquad\textnormal{i.e.}\qquad \mathbf{x}_{i,\top}^\eps(t)\cdot\mathbf{x}_{i}(t)=1.$$
        Therefore, combining the last three identities and using the fact that $\mathbf{x}_{i,\perp}^{\varepsilon}(t)$ is colinear to $\mathbf{x}_i(t),$ we obtain
        \begin{equation*}
            |\mathbf{x}_i^\eps(t) - \mathbf{x}_i(t) |_{\mathbb{R}^3}^2 = -2 \mathbf{x}_{i,\perp}^\eps(t)\cdot\mathbf{x}_i(t) = 2  |\mathbf{x}_{i,\perp}^\eps(t)|_{\mathbb{R}^3}.
        \end{equation*}
        In conclusion, the normal perturbation is quadratic
        \begin{equation*}
            |\mathbf{X}^\eps_\perp(t)|_N \leqslant C |\mathbf{X}^\eps(t) - \mathbf{X}(t) |_N^2.
        \end{equation*}
        We have proved that
        \begin{equation*}
            \der{}{t} |\mathbf{X}^\eps(t)-\mathbf{X}(t)|_N^2 \leqslant C |\mathbf{X}^\eps(t)-\mathbf{X}(t)|_N^3.
        \end{equation*}
        We then make a quick bootstrap argument. For all $t \leqslant \bar{\tau}_{\eps,\beta}$, $|\mathbf{x}_i(t) - \mathbf{x}_i^\eps(t)| < \eps^\beta/2$, thus $|\mathbf{X}^\eps(t)-\mathbf{X}(t)|_2 \leqslant \frac{1}{2}\sqrt{N} \eps^\beta$ so that
        \begin{equation*}
            \der{}{t} |\mathbf{X}^\eps(t)-\mathbf{X}(t)|_N\leqslant C \eps^\beta |\mathbf{X}^\eps(t)-\mathbf{X}(t)|_N,
        \end{equation*}
        which yields that
        \begin{equation*}
            |\mathbf{X}^\eps(t)-\mathbf{X}(t)|_N\leqslant |\mathbf{X}^\eps(0)-\mathbf{X}^0|_N e^{C\eps^\beta t} \leqslant \eps e^{C\eps^\beta t}.
        \end{equation*}
        Therefore, for any $\alpha < \beta$, for every $t \leqslant \eps^{-\alpha}$ and provided $\eps$ is small enough, we have that
        \begin{equation*}
            \max_{i\in\{1,\ldots,N\}}|\mathbf{x}_i^\eps(t)-\mathbf{x}_i(t)|_\infty \leqslant  |\mathbf{X}^\eps(t)-\mathbf{X}(t)|_N \leqslant 2\eps.
        \end{equation*}
        which proves by continuity of the trajectories that $\bar{\tau}_{\eps,\beta} \geqslant \eps^{-\alpha}.$
    \end{proof}

    \subsubsection{Linear Instability implies Hypothesis~\ref{hyp:instable}}

    Let us recall Theorem 6.1, Chapter 9 of \cite{hartman1982ordinary}.
\begin{theo}\label{theo:center_manifold}
Let $f : (\mathbb{S}^2)^N \to (\mathbb{S}^2)^N$. We consider the differential equation 
\begin{equation}\label{eq:diff_eq_generale}
    \der{}{t} \widetilde{\mathbf{X}}(t) = \mathcal{F}\big(\widetilde{\mathbf{X}}(t)\big).
\end{equation}
Assume that there exists $\mathbf{X}^* \in (\R^2)^N$ is such that $f(\mathbf{X}^*) = 0$. Assume furthermore that $\Diff f(Z^*)$ has an eigenvalue with positive real part $\lambda_0 > 0$. Then there exists a solution $\tilde{\mathbf{X}}$ of \eqref{eq:diff_eq_generale} such that $\tilde{\mathbf{X}}(t)$ exists some fixed neighborhood of $\mathbf{X}^*$, that
\begin{equation*}
    \widetilde{\mathbf{X}}(t) \tend{t}{-\infty} \mathbf{X}^*,
\end{equation*}
and that
\begin{equation*}
    \frac{1}{t}\ln |\widetilde{\mathbf{X}}(t)-\mathbf{X}^*|_{\infty,N} \tend{t}{-\infty} \lambda_0.
\end{equation*}
where $|\cdot|_{\infty,N}$ is the norm defined by
$$|\mathbf{X}|_{\infty,N}\triangleq\max_{i\in\{1,\ldots,N\}}|\mathbf{x}_i|_{\mathbb{R}^3}.$$
\end{theo}
We claim that since Theorem~\ref{theo:center_manifold} is purely local, it is also true on $(\mathbb{S}^2)^N$ by writing in a local chart, and with $|Z(t)-Z^*|$ replaced with the $\mathbb{S}^2$ distance, or equivalently and suiting better our notations, the $\R^3$ distance. Let us now prove the following.
\begin{prop}\label{prop:instable_implique_hyp}
    Assume that $\mathbf{X}^*$ is an equilibrium of the point-vortex dynamics such that $\Diff \mathcal{F}(X^*)$ has an eigenvalue with positive real part. Then Hypothesis~\ref{hyp:instable} is satisfied.
\end{prop}
\begin{proof}
    Let $\beta \in (0,1)$ and let $\tilde{\mathbf{X}}$ a solution of the point-vortex dynamics given by Theorem \ref{theo:center_manifold}. Since $\tilde{\mathbf{X}}(t) \tend{t}{-\infty} \mathbf{X}^*$ and since $\tilde{\mathbf{X}}$ exits some fixed neighborhood of $\mathbf{X}^*$, for $\eps$ small enough, there exist $t_0$ and $t_1$ such that
\begin{equation*}
    \begin{cases}
    -\infty < t_0 < t_1 \\
    t_1 \to -\infty & \text{ as } \eps \to 0\\
    |\tilde{\mathbf{X}}(t_1)-\mathbf{X}^*|_\infty = 4\eps^\beta \\
    |\tilde{\mathbf{X}}(t_0)-\mathbf{X}^*|_\infty = \tfrac{\eps}{2}.
    \end{cases}
\end{equation*}
Let $\mathbf{X}^\eps(t) = \tilde{\mathbf{X}}(t + t_0)$, we have that $|\mathbf{X}^\eps(0) - \mathbf{X}^*|_\infty = \tfrac{\eps}{2}$ and that $\tau^*_{\eps,\beta} \leqslant t_1-t_0$ since $\mathbf{X}^{\varepsilon}(t_1-t_0) = 4 \eps^\beta$. Moreover, since 
\begin{equation*}
    \ln |\mathbf{X}(t)-\mathbf{X}^*|_\infty = \lambda_0 t + o_{t \to -\infty}(t),
\end{equation*}
then for any $\kappa \in (0,1)$, for $-t$ big enough we have that
\begin{equation*}
    1 - \kappa < \frac{\ln |\mathbf{X}^{\varepsilon}(t)-\mathbf{X}^*|_\infty}{\lambda_0 t} < 1 + \kappa .
\end{equation*}
Therefore, for $\eps$ small enough, applying in $t_0$ and $t_1$ (we recall that $t_1 \to -\infty$ as $\eps \to 0$) we have that 
\begin{equation*}
    t_1 < \frac{\ln |\mathbf{X}^{\varepsilon}(t_1)-\mathbf{X}^*|_\infty}{\lambda_0(1+\kappa)} = \frac{\beta\ln \eps + \ln 2}{\lambda_0 (1+\kappa)} = \frac{- \beta|\ln \eps|+\ln 2}{\lambda_0 (1+\kappa)}
\end{equation*}
and
\begin{equation*}
    -t_0 < -\frac{\ln |\mathbf{X}^{\varepsilon}(t_0)-\mathbf{X}^*|_\infty}{\lambda_0(1-\kappa)} = \frac{|\ln\eps|+\ln 2}{\lambda_0 (1-\kappa)}\cdot
\end{equation*}
Thus,
\begin{equation*}
    t_1 - t_0 < |\ln\eps| \left( \frac{1}{\lambda_0(1-\kappa)} - \frac{\beta}{\lambda_0(1+\kappa)} + \frac{\ln 2 }{\lambda_0|\ln \eps| } \left( \frac{1}{1+\kappa} + \frac{1}{1-\kappa}\right) \right).
\end{equation*}
Therefore, by letting $\kappa \to 0$, for any $\lambda < \lambda_0$, for $\eps$ small enough, 
\begin{equation*}
    t_1 - t_0 \leqslant \frac{1-\beta}{\lambda} |\ln \eps|.
\end{equation*}
By definition, $\tau_{\eps,\beta}^* \leqslant t_1 - t_0 \leqslant \frac{1-\beta}{\lambda} |\ln \eps|$.
This concludes the proof.
\end{proof}

\subsection{Computing the differential matrix}\label{sec:differential}
The goal of this subsection is to provide the explicit general computation of the matrix $\Diff\mathcal{F}(\mathbf{X})$ associated with the point vortex functional in \eqref{def:PVvectorfield}. Here one must deal with difficulties coming from the non-Euclidean geometry that appear through a projection analysis on the tangent bundle.\\

Let us compute the Jacobian matrix $\Diff_1 K_{\mathbb{S}^2}(\mathbf{x},\mathbf{y})$. Let $p \in \R^3$. We have that
\begin{align*}
    K_{\mathbb{S}^2}(\mathbf{x}+p,\mathbf{y}) - K_{\mathbb{S}^2}(\mathbf{x},\mathbf{y}) & = \frac{(\mathbf{x}+p) \wedge \mathbf{y}}{|\mathbf{x}+p-\mathbf{y}|_{\mathbb{R}^3}^2}- \frac{\mathbf{x} \wedge \mathbf{y}}{|\mathbf{x}-\mathbf{y}|_{\mathbb{R}^3}^2} \\
    & = \frac{p \wedge \mathbf{y}}{|\mathbf{x}+p-\mathbf{y}|_{\mathbb{R}^3}^2} + \frac{\mathbf{x} \wedge \mathbf{y}}{|\mathbf{x}-\mathbf{y}|_{\mathbb{R}^3}^2}\left(\frac{1}{1+ 2p \cdot\frac{\mathbf{x}-\mathbf{y}}{|\mathbf{x}-\mathbf{y}|_{\mathbb{R}^3}^2} + o(|p|)} -1 \right) \\
    & = \frac{p \wedge \mathbf{y}}{|\mathbf{x}-\mathbf{y}|_{\mathbb{R}^3}^2} -2 p\cdot(\mathbf{x}-\mathbf{y}) \frac{\mathbf{x} \wedge \mathbf{y}}{|\mathbf{x}-\mathbf{y}|_{\mathbb{R}^3}^4} + o(|p|_{\mathbb{R}^3}).
\end{align*}
Therefore, for every $h \in T_{\mathbf{x}}\mathbb{S}^2$,
\begin{equation}\label{eq:differential}
    \Diff_1 K_{\mathbb{S}^2}(\mathbf{x},\mathbf{y})[h] = \frac{h \wedge \mathbf{y}}{|\mathbf{x}-\mathbf{y}|_{\mathbb{R}^3}^2} -2 h\cdot(\mathbf{x}-\mathbf{y}) \frac{\mathbf{x} \wedge \mathbf{y}}{|\mathbf{x}-\mathbf{y}|_{\mathbb{R}^3}^4}\cdot
\end{equation}
By symmetry, the differential with respect to the second variable can be obtained
by exchanging the roles of $\mathbf x$ and $\mathbf y$ and using the antisymmetry
of the wedge product. More precisely, since
\[
K_{\mathbb S^2}(\mathbf x,\mathbf y)
=-K_{\mathbb{S}^2}(\mathbf{y},\mathbf{x}),
\]
we obtain, for every $k\in T_{\mathbf y}\mathbb S^2$,
\begin{equation}\label{eq:differential_D2}
\Diff_25 K_{\mathbb S^2}(\mathbf x,\mathbf y)[k]
=\frac{\mathbf x\wedge k}{|\mathbf x-\mathbf y|_{\mathbb{R}^3}^2}
+2\,k\cdot(\mathbf x-\mathbf y)\,
\frac{\mathbf x\wedge\mathbf y}{|\mathbf x-\mathbf y|_{\mathbb{R}^3}^4}\cdot
\end{equation}
The Jacobian matrix of the function $\mathcal{F}$ is by definition
$$\Diff\mathcal{F}(\mathbf{X})=\left(\frac{\partial\mathcal{F}_i}{\partial\mathbf{x}_j}(\mathbf{X})\right)_{1\leqslant i,j\leqslant N}\triangleq(A_{ij})_{1\leqslant i,j\leqslant N},$$
where each component is the projection on the tangent space of the differential in the ambient space, namely,
\begin{equation*}
    A_{ij}=
\begin{cases}
\displaystyle
\Pi_{\mathbf x_i}\left[\sum_{k=1\atop k\neq i}^{N}\frac{\Gamma_k}{2\pi}\,\Diff_1K(\mathbf x_i,\mathbf x_k)
+\gamma(\mathbf{e}_3\wedge\cdot)\right]_{T_{\mathbf x_i}\mathbb{S}^2\to\R^3},
& \textnormal{if }i=j,\\[0.8em]
\displaystyle
\Pi_{\mathbf x_i}\left[\frac{\Gamma_j}{2\pi}\,\Diff_25K(\mathbf x_i,\mathbf x_j)\right]_{T_{\mathbf x_j}\mathbb{S}^2\to\R^3},
& \textnormal{if } i\neq j,
\end{cases}
\end{equation*}
where recall that $\Pi_{\mathbf{x}}\triangleq\textnormal{Id}-\mathbf{x}\otimes\mathbf{x}$ is the projection onto the tangent plane $T_{\mathbf{x}}\mathbb{S}^2$. 
Combining the exact expression \eqref{def:PVvectorfield} with \eqref{eq:differential} and \eqref{eq:differential_D2}, given a direction $H=(h_1,\ldots,h_N)\in T_{\mathbf{x}_1}\mathbb{S}^2\times\ldots\times T_{\mathbf{x}_N}\mathbb{S}^2$, we have that for any $i,j\in\{1,\ldots,N\},$
\begin{equation}\label{rep-diffF}
    A_{ij}[h_j]=\begin{cases}
         \displaystyle\Pi_{\mathbf{x}_i}\left[\sum_{k=1\atop k\neq i}^{N}\frac{\Gamma_k}{2\pi}\left(\frac{h_i\wedge\mathbf{x}_k}{|\mathbf{x}_i-\mathbf{x}_k|_{\mathbb{R}^3}^2}-2h_i\cdot(\mathbf{x}_i-\mathbf{x}_k)\frac{\mathbf{x}_i\wedge\mathbf{x}_k}{|\mathbf{x}_i-\mathbf{x}_k|_{\mathbb{R}^3}^4}\right)+\gamma\mathbf{e}_3\wedge h_i\right], & \textnormal{if } i=j, \vspace{2mm}\\
         \displaystyle \Pi_{\mathbf{x}_i}\left[\frac{\Gamma_j}{2\pi}\left(\frac{\mathbf{x}_i\wedge h_j}{|\mathbf{x}_i-\mathbf{x}_j|_{\mathbb{R}^3}^2}+2h_j\cdot(\mathbf{x}_i-\mathbf{x}_j)\frac{\mathbf{x}_i\wedge\mathbf{x}_j}{|\mathbf{x}_i-\mathbf{x}_j|_{\mathbb{R}^3}^4}\right)\right], & \textnormal{if }i\neq j.
    \end{cases}
\end{equation}
By a quick computation, one can check that for any $\mathbf{x},\mathbf{y} \in \mathbb{S}^2$, and any $h \in T_{\mathbf{x}}\mathbb{S}^2$, we have that
\begin{equation*}
    \Pi_{\mathbf{x}} (\mathbf{y} \wedge h) = (\mathbf{y} \cdot \mathbf{x}) (\mathbf{x} \wedge h).
\end{equation*}
Also, for any $\mathbf{x}\in\mathbb{S}^2$ and any $h\in\mathbb{R}^3$,
\begin{equation*}
    \Pi_{\mathbf{x}}(\mathbf{x} \wedge h) = \mathbf{x}\wedge \Pi_\mathbf{x} h.
\end{equation*}
Noticing that the application
\begin{equation*}
    \mathtt{R}_{\mathbf{x}}:\begin{array}[t]{rcl}
        T_\mathbf{x}\mathbb{S}^2 &  \rightarrow & T_\mathbf{x}\mathbb{S}^2\\
         h & \mapsto & \mathbf{x} \wedge h 
    \end{array}
\end{equation*}
is exactly the rotation of angle $\tfrac{\pi}{2}$ in $T_\mathbf{x}\mathbb{S}^2$. Given $B=(b_1,b_2)$ an orthonormal basis of $T_{\mathbf{x}}\mathbb{S}^2$ such that $(\mathbf{x},b_1,b_2)$ is a direct orthonormal basis of $\mathbb{R}^3$, the matrix of $\mathtt{R}_{\mathbf{x}}$ in the basis $B$ is
$$\underset{B}{\textnormal{Mat}}(\mathtt{R}_{\mathbf{x}})=\begin{pmatrix}
    0 & -1\\
    1 & 0
\end{pmatrix}\triangleq J.$$
We now choose once for all for any $i\in\{1,\ldots,N\}$ an orthonormal basis $B_i\triangleq(b_{1,i},b_{2,i})$ such that the triplet $(\mathbf{x}_i,b_{1,i},b_{2,i})$ is a direct orthonormal basis of $\mathbb{R}^3$. Next, denote by
\begin{equation*}
    M_{ik}\triangleq\underset{B_i}{\textnormal{Mat}}(\mathtt{f}_{ik}),\qquad \mathtt{f}_{ik}:\begin{array}[t]{rcl}
        T_{\mathbf{x}_i}\mathbb{S}^2 &  \rightarrow & T_{\mathbf{x}_i}\mathbb{S}^2\\
         h & \mapsto &  \displaystyle 2h\cdot(\mathbf{x}_i-\mathbf{x}_k)\frac{\mathbf{x}_i\wedge\mathbf{x}_k}{|\mathbf{x}_i-\mathbf{x}_k|_{\mathbb{R}^3}^2}, 
    \end{array}
\end{equation*}

\begin{equation*}
    N_{ij}\triangleq\underset{B_j,B_i}{\textnormal{Mat}}(\mathtt{g}_{ij}),\qquad\mathtt{g}_{ij}:\begin{array}[t]{rcl}
        T_{\mathbf{x}_j}\mathbb{S}^2 &  \rightarrow & T_{\mathbf{x}_i}\mathbb{S}^2\\
         h & \mapsto &  \displaystyle 2h\cdot(\mathbf{x}_i-\mathbf{x}_j)\frac{\mathbf{x}_i\wedge\mathbf{x}_j}{|\mathbf{x}_i-\mathbf{x}_j|_{\mathbb{R}^3}^2}
    \end{array}
\end{equation*}
and 
\begin{equation*}
    P_{ij}\triangleq\underset{B_j,B_i}{\textnormal{Mat}}(\mathtt{p}_{ij}),\qquad \mathtt{p}_{ij}:\begin{array}[t]{rcl}
        T_{\mathbf{x}_j}\mathbb{S}^2 &  \rightarrow & T_{\mathbf{x}_i}\mathbb{S}^2\\
         h & \mapsto &  \Pi_{\mathbf{x}_i}(h)=\displaystyle h-(\mathbf{x}_i\cdot h)\mathbf{x}_i. 
    \end{array}
\end{equation*}
In conclusion, in the basis $B_1 \times\ldots \times B_N$ of $T_{\mathbf{x}_1}\mathbb{S}^2\times\ldots\times T_{\mathbf{x}_N}\mathbb{S}^2$, the block $A_{ij}$ is represented by the matrix
\begin{equation}\label{eq:differentielle_complete}
        A_{ij}=\begin{cases}
         \displaystyle\sum_{k=1\atop k\neq i}^{N}\frac{-\Gamma_k}{2\pi|\mathbf{x}_i-\mathbf{x}_k|_{\mathbb{R}^3}^2}\left(\mathbf{x}_i\cdot\mathbf{x}_k J + M_{ik}\right)+\gamma \mathbf{e}_3 \cdot \mathbf{x}_i J, & \textnormal{if } i=j, \vspace{2mm}\\
         \displaystyle \frac{\Gamma_j}{2\pi|\mathbf{x}_i-\mathbf{x}_j|_{\mathbb{R}^3}^2}( J P_{ij} +  N_{ij})  , & \textnormal{if }i\neq j.
    \end{cases}
\end{equation}

\subsection{Proof of Theorem~\ref{theo:strong_conf}}\label{sec:config_stable}

    Let us recall that in the planar case, without boundaries, \cite{BM18} obtained two configurations that realize the improved bound $\tau_{\eps,\beta} \geqslant \eps^{-\alpha}$: by taking $N=1$, or by taking self-similar expanding configurations of point-vortices. Both are not possible on the sphere. Indeed, due to the Gauss constraint, the point-vortex problem cannot be stated with $N=1$. As for the second situation, the compactness of the sphere prevents expansion.\\

    Let us consider the case of polar counter-rotating vortices, namely 
    \begin{equation}\label{conf2pt}
        N=2,\qquad (\mathbf{x}_1,\mathbf{x}_2) = (\mathbf{e}_3,-\mathbf{e}_3),\qquad(\Gamma_1,\Gamma_2) = (\Gamma,-\Gamma).
    \end{equation}
    By construction, the Gauss constraint \eqref{Gauss:PV} is satisfied. In addition, in view of Lemma \ref{lem rot sph}, the configuration \eqref{conf2pt} is stationary and \eqref{dist-min-points} holds with $d_0=2.$\\
    $\blacktriangleright$ \textit{Hypothesis \ref{hyp:supstable} :} Using the relation~\eqref{eq:differential}, we compute that for any $h\in\mathbb{R}^3,$
\begin{equation*}
    \Diff_1 K_{\mathbb{S}^2}(\mathbf{e}_3,-\mathbf{e}_3)[h] = -\frac{h \wedge \mathbf{e}_3}{4}=-\Diff_1 K_{\mathbb{S}^2}(-\mathbf{e}_3,\mathbf{e}_3)[h].
\end{equation*}
In particular, that for any $h \in \R^3$, we have
\begin{equation*}
    \Diff_1 K_{\mathbb{S}^2}(\mathbf{e}_3,-\mathbf{e}_3)[h] \cdot h = 0=\Diff_1 K_{\mathbb{S}^2}(-\mathbf{e}_3,\mathbf{e}_3)[h] \cdot h.
\end{equation*}
Therefore, this configuration satisfies Hypothesis~\ref{hyp:supstable}.\\
$\blacktriangleright$ \textit{Hypothesis \ref{hyp:stable} :} With the notations of Section \ref{sec:differential}, since $\mathbf{e}_3 \wedge (-\mathbf{e}_3) = 0$, we readily get 
$$M_{12} = M_{21} = N_{12} = N_{21} = 0.$$ 
Also, due to the orientation convention, one has 
$$P_{12} = P_{21} = \begin{pmatrix}
    1 & 0 \\  0 & -1\end{pmatrix}.$$
Therefore, for any $(h_1,h_2)\in T_{\mathbf{e}_3}\mathbb{S}^2\times T_{-\mathbf{e}_3}\mathbb{S}^2$ we can write matricially
\begin{equation}\label{diffF2pt}
    \Diff \mathcal{F}(\mathbf{e}_3,-\mathbf{e}_3)[h_1,h_2] = \begin{pmatrix} \vspace{1mm} (-\frac{\Gamma}{8\pi}+\gamma) J & -\frac{\Gamma}{8\pi}S \\ \frac{\Gamma}{8\pi}S & (\frac{\Gamma}{8\pi}-\gamma)J \end{pmatrix}\begin{pmatrix}h_1\\h_2\end{pmatrix},
\end{equation}
where 
$$S \triangleq JP_{12} = \begin{pmatrix} 0 & -1 \\ -1 & 0\end{pmatrix}.$$
The matrix in \eqref{diffF2pt} is skew-symmetric, hence satisfies
\begin{equation*}
    \Diff \mathcal{F}(\mathbf{x},-\mathbf{x})[h_1,h_2] \cdot (h_1,h_2) = 0.
\end{equation*}
Applying Proposition~\ref{prop:stable_implique_hypstable}, we conclude that the configuration \eqref{conf2pt} satisfies Hypothesis~\ref{hyp:stable}.\\

In conclusion, for any couple $\Gamma, \gamma \in \R$, the configuration of point-vortices $\mathbf{X}^0 = (\mathbf{e}_3,-\mathbf{e}_3)$ with intensities $(\Gamma_1,\Gamma_2) = (\Gamma,-\Gamma)$ satisfies Hypotheses~\ref{hyp:supstable} and~\ref{hyp:stable}. We can thus apply conditional Theorem~\ref{theo:strong_conf_conditional} to this configuration, which proves Theorem~\ref{theo:strong_conf}.

\subsection{Proof of Theorem~\ref{theo:optimal}}\label{sec:config_instable}

We now consider the following configuration: let $N =4$, $a \in (0,1)$, $\Gamma\neq 0$, $\gamma \in \R$ and let us consider the configuration:
\begin{equation*}
    \begin{alignedat}{3}
\mathbf{x}_1^0 = \mathbf{e}_3, \quad 
\mathbf{x}_2^0 = -\mathbf{e}_3, \quad
\mathbf{x}_3^0 = \begin{pmatrix}a \\ 0 \\ \sqrt{1-a^2}\end{pmatrix}, \quad 
\mathbf{x}_4^0 = \begin{pmatrix}-a \\ 0 \\ \sqrt{1-a^2}\end{pmatrix},
\end{alignedat}
\end{equation*}
where we impose that $\Gamma_3=\Gamma_4 \triangleq \Gamma \neq 0$, that $\Gamma_1 = \kappa \Gamma$ where $\kappa$ is to be determined and $\Gamma_2$ is such that $\Gamma_1 + \Gamma_2 + \Gamma_3 + \Gamma_4 = 0$ so that the Gauss constraint in satisfied.

\begin{lem}\label{lem:stationary}
    There exists a unique choice of $\kappa$ such that the previously described configuration is a stationary solution of the point-vortex dynamics~\eqref{PVS} for any choice of $\Gamma \neq 0$ and $\gamma \in \R$.
\end{lem}
\begin{proof}
    We need to prove that for all $i \in \{1,\ldots,4\}$,
\begin{equation*}
     0=\sum_{j=1\atop j\neq i}^{4}\frac{\Gamma_j}{2\pi}\frac{\mathbf{x}_i^0\wedge\mathbf{x}_j^0}{|\mathbf{x}_i^0-\mathbf{x}_j^0|_{\R^3}^2}+\gamma\,\mathbf{e}_3\wedge\mathbf{x}_i^0.
\end{equation*}
For $i = 1,2$, we have that $|\pm \mathbf{e}_3 -\mathbf{x}_3^0| = |\pm \mathbf{e}_3 -\mathbf{x}_4^0|$ and $\mathbf{x}_3^0 +\mathbf{x}_4^0 = 2\sqrt{1-a^2} \mathbf{e}_3$, so that
\begin{align*}
    \sum_{j=1\atop j\neq i}^{4}\frac{\Gamma_j}{2\pi}\frac{\mathbf{x}_i^0\wedge\mathbf{x}_j^0}{|\mathbf{x}_i^0-\mathbf{x}_j^0|_{\R^3}^2}+\gamma\,\mathbf{e}_3\wedge\mathbf{x}_i^0 = \frac{\Gamma}{2\pi} \frac{\pm \mathbf{e}_3 \wedge(\mathbf{x}_3^0 +\mathbf{x}_4^0)}{|\pm\mathbf{e}_3 -\mathbf{x}_3^0|_{\mathbb{R}^3}^2} = 0.
\end{align*}
For $i=3$ we compute first that $\mathbf{x}_3^0 \wedge \mathbf{x}_4^0 = 2\sqrt{1-a^2} \mathbf{x}_3^0 \wedge \mathbf{e}_3$, and thus
\begin{align*}
    \sum_{j=1\atop j\neq 3}^{4}\frac{\Gamma_j}{2\pi}\frac{\mathbf{x}_3^0\wedge\mathbf{x}_j^0}{|\mathbf{x}_3^0-\mathbf{x}_j^0|_{\R^3}^2}+\gamma\,\mathbf{e}_3\wedge\mathbf{x}_3^0 & = \mathbf{x}_3^0 \wedge \mathbf{e}_3 \left( \frac{\Gamma_1}{2\pi|\mathbf{x}_3^0 - \mathbf{e}_3|_{\mathbb{R}^3}^2} - \frac{\Gamma_2}{2\pi|\mathbf{x}_3^0 + \mathbf{e}_3|_{\mathbb{R}^3}^2} + \frac{\Gamma_4 2\sqrt{1-a^2}}{2\pi |\mathbf{x}_3^0-\mathbf{x}_4^0|_{\mathbb{R}^3}^2} -\gamma\right).
\end{align*}
We introduce the notation
\[
s = \sqrt{1-a^2}, \qquad 
\alpha_-=\frac{1}{|\mathbf x_3^0-\mathbf{e}_3|_{\mathbb{R}^3}^2} = \frac{1}{{2(1-s)}},\qquad
\alpha_+=\frac{1}{|\mathbf x_3^0+\mathbf{e}_3|_{\mathbb{R}^3}^2} =  \frac{1}{2(1+s)},\qquad
\Upsilon=\frac{1}{|\mathbf x_3^0-\mathbf x_4^0|_{\mathbb{R}^3}^2}=\frac{1}{4a^2}\cdot
\]
Therefore, this is 0 if and only if
\begin{equation*}
    \Gamma_1\alpha_- - \Gamma_2\alpha_+ + \Gamma_4 \Upsilon 2 \sqrt{1-a^2} -2\pi\gamma = 0,
\end{equation*}
which is equivalent to
\begin{equation*}
    \kappa \left( \alpha_- +\alpha_+\right) = - 2s\Upsilon-2\alpha_+ + 2\pi\frac{\gamma}{\Gamma}\cdot
\end{equation*}
This gives a unique choice of $\kappa$ such that this equation holds. By symmetry, we have the same result for $i = 4$.
\end{proof}
Let us make this choice of $\kappa$. Since only the ratio $\gamma/\Gamma$ is relevant, up to renaming $\gamma$, we can choose $\Gamma=1$, leading to $\Gamma_2 = -(2+\kappa)$. Let us denote the resulting equilibrium by $\mathbf{X}^*$, and let us drop the $0$ superscript so that we consider the configuration  $(\mathbf{x}_1,\ldots,\mathbf{x}_4)$.
With the method presented in Section~\ref{sec:differential}, we can construct $\Diff \mathcal{F}(\mathbf{X}^*)$ for this configuration. Serving as an example, let us compute $A_{11}$.
Relation~\eqref{eq:differentielle_complete} yields
\begin{equation}\label{eq:A11_general}
A_{11}
=
\sum_{k\in\{2,3,4\}}
\frac{-\Gamma_k}{2\pi|\mathbf x_1-\mathbf x_k|_{\mathbb{R}^3}^2}
\Bigl((\mathbf x_1\cdot \mathbf x_k)J + M_{1k}\Bigr)
+\gamma\,(\mathbf{e}_3\cdot\mathbf x_1)\,J.
\end{equation}
Since $\mathbf x_1=\mathbf{e}_3$, we have $\mathbf{e}_3\cdot\mathbf x_1=1$, hence
\[
\gamma(\mathbf{e}_3\cdot\mathbf x_1)J=\gamma J.
\]
We now turn to the term $k=2$ in the sum.
We have
\[
|\mathbf x_1-\mathbf x_2|_{\mathbb{R}^3}^2=|\mathbf{e}_3-(-\mathbf{e}_3)|_{\mathbb{R}^3}^2=4,
\qquad
\mathbf x_1\cdot\mathbf x_2=\mathbf{e}_3\cdot(-\mathbf{e}_3)=-1,
\qquad
\mathbf x_1\wedge\mathbf x_2=0.
\]
In particular, $\mathbf x_1\wedge\mathbf x_2=0$ implies $M_{12}=0$.
Therefore the $k=2$ term in~\eqref{eq:A11_general} is
\[
\frac{-\Gamma_2}{2\pi\cdot 4}\Bigl((\mathbf x_1\cdot\mathbf x_2)J+M_{12}\Bigr)
=
\frac{-\Gamma_2}{8\pi}\,(-J)
=
\frac{-(\kappa+2)}{8\pi}\,J.
\]
Let us now look at the term $k=3$.
We first collect the geometric quantities
\[
\mathbf x_1-\mathbf x_3=( -a,0,1-s),
\qquad
|\mathbf x_1-\mathbf x_3|_{\mathbb{R}^3}^2
=
a^2+(1-s)^2
=
2(1-s),
\qquad
\mathbf x_1\cdot\mathbf x_3=s,
\]
\[
\mathbf x_1\wedge\mathbf x_3=\mathbf{e}_3\wedge(a,0,s)=(0,a,0)=a\,\mathbf e_2.
\]
Let $h\in T_{\mathbf{e}_3}\mathbb S^2$, we have that $h\cdot \mathbf{e}_3=0$.
Then
\[
h\cdot(\mathbf x_1-\mathbf x_3)
=
-a\,h\cdot\mathbf e_1
\]
(where $\mathbf e_1$ is the ambient unit vector).
Hence the vector-valued map defining $M_{13}$ reads
\[
h\longmapsto
2\,h\cdot(\mathbf x_1-\mathbf x_3)\,
\frac{\mathbf x_1\wedge\mathbf x_3}{|\mathbf x_1-\mathbf x_3|_{\mathbb{R}^3}^2}
=
2\,(-a\,h\cdot\mathbf e_1)\,
\frac{a\,\mathbf e_2}{2(1-s)}
=
-\frac{a^2}{1-s}\,(h\cdot\mathbf e_1)\,\mathbf e_2.
\]
Choosing $B_1=(\mathbf e_1,\mathbf e_2)$ (the simplest oriented choice),
then the above map is represented by the matrix
\[
M_{13}=
\begin{pmatrix}
0&0\\[0.2em]
-\dfrac{a^2}{1-s}&0
\end{pmatrix}.
\]
Using now~\eqref{eq:A11_general}, the full $k=3$ contribution is
\[
\frac{-\Gamma_3}{2\pi|\mathbf x_1-\mathbf x_3|_{\mathbb{R}^3}^2}\Bigl((\mathbf x_1\cdot\mathbf x_3)J+M_{13}\Bigr)
=
\frac{-1}{2\pi\cdot 2(1-s)}\Bigl(sJ+M_{13}\Bigr).
\]
We turn to the term $k=4$. The computation is identical up to the sign change $a\mapsto -a$ in $\mathbf x_4$.
In particular,
\[
|\mathbf x_1-\mathbf x_4|_{\mathbb{R}^3}^2=2(1-s),\qquad \mathbf x_1\cdot\mathbf x_4=s,
\qquad \mathbf x_1\wedge\mathbf x_4=\mathbf{e}_3\wedge(-a,0,s)=(0,-a,0)=-a\,\mathbf e_2.
\]
Moreover $h\cdot(\mathbf x_1-\mathbf x_4)=+a\,h\cdot\mathbf e_1$, so that the two
signs compensate and one obtains the \emph{same} rank-one operator as for $k=3$.
In particular, for $B_1=(\mathbf e_1,\mathbf e_2)$ we again have
\[
M_{14}=M_{13},
\qquad\text{hence}\qquad
\frac{-\Gamma_4}{2\pi|\mathbf x_1-\mathbf x_4|_{\mathbb{R}^3}^2}\Bigl((\mathbf x_1\cdot\mathbf x_4)J+M_{14}\Bigr)
=
\frac{-1}{2\pi\cdot 2(1-s)}\Bigl(sJ+M_{13}\Bigr).
\]
Summing the contributions of $k=2,3,4$ and adding the Coriolis term, we obtain
\begin{align*}
A_{11} & =\left(\gamma-\frac{\kappa+2}{8\pi}-\frac{s}{2\pi(1-s)}\right)J-\frac{1}{2\pi(1-s)}M_{13} \\
& = \begin{pmatrix}
    0 &  -\gamma+\frac{\kappa+2}{8\pi}+\frac{s}{2\pi(1-s)} \\ \gamma-\frac{\kappa+2}{8\pi}-\frac{s}{2\pi(1-s)} + \frac{a^2}{2\pi(1-s)^2} & 0
\end{pmatrix}.
\end{align*}

We now see that the full $8\times 8$ matrix will be very hard to manipulate since it depends on the two paramters $a$ and $\gamma$. Instead of looking to the albebraic properties of this matrix, we look instead for a set of parameters simplifying it enough so that it is computable. We claim the following.
\begin{lem}
    Take $a = 1$ and $\gamma = 1/2$. Then $\Diff \mathcal{F}(\mathbf{X}^*)$ has a positive eigenvalue.
\end{lem}
\begin{proof}
    With this set of parameters, we compute that
    \[
\Diff\mathcal F(\mathbf X^*)
=
\begin{pmatrix}
0 & \dfrac{1-3\pi}{8\pi} & 0 & -\dfrac{\pi+1}{8\pi} & -\dfrac{1}{4\pi} & 0 & -\dfrac{1}{4\pi} & 0\\[0.9em]
\dfrac{3(\pi+1)}{8\pi} & 0 & -\dfrac{\pi+1}{8\pi} & 0 & 0 & \dfrac{1}{4\pi} & 0 & \dfrac{1}{4\pi}\\[0.9em]
0 & \dfrac{\pi-1}{8\pi} & 0 & \dfrac{3\pi+1}{8\pi} & \dfrac{1}{4\pi} & 0 & \dfrac{1}{4\pi} & 0\\[0.9em]
\dfrac{\pi-1}{8\pi} & 0 & \dfrac{3(1-\pi)}{8\pi} & 0 & 0 & -\dfrac{1}{4\pi} & 0 & -\dfrac{1}{4\pi}\\[0.9em]
\dfrac{1-\pi}{4\pi} & 0 & -\dfrac{\pi+1}{4\pi} & 0 & 0 & \dfrac{3}{8\pi} & 0 & \dfrac{1}{8\pi}\\[0.9em]
0 & \dfrac{\pi-1}{4\pi} & 0 & \dfrac{\pi+1}{4\pi} & \dfrac{1}{8\pi} & 0 & \dfrac{1}{8\pi} & 0\\[0.9em]
\dfrac{1-\pi}{4\pi} & 0 & -\dfrac{\pi+1}{4\pi} & 0 & 0 & \dfrac{1}{8\pi} & 0 & \dfrac{3}{8\pi}\\[0.9em]
0 & \dfrac{\pi-1}{4\pi} & 0 & \dfrac{\pi+1}{4\pi} & \dfrac{1}{8\pi} & 0 & \dfrac{1}{8\pi} & 0
\end{pmatrix}.
\]
One can then compute that
\[
\chi(\lambda)=\det\big(\lambda I_8-\Diff\mathcal F(\mathbf X^*)\big) 
=\frac{1}{512\pi^2}\,\lambda^2\,(4\lambda^2+1)\,\Bigl(128\pi^2\lambda^4+(32+8\pi^2)\lambda^2+3\Bigr)
\]
and conclude that
\begin{equation*}
\operatorname{spec}\!\left(\Diff\mathcal F(\mathbf X^*)\big|_{\gamma=\frac12}\right)
=
\left\{
0,\ \pm \frac{i}{2},\ 
\pm \sqrt{\mu_{+}},\ \pm i\sqrt{-\mu_{-}}
\right\},
\qquad
\mu_{\pm}=\frac{-(\pi^2+4)\pm\sqrt{\pi^4-16\pi^2+16}}{32\pi^2},
\end{equation*}
where $\mu_+ \approx 0.049 > 0$.
Therefore, by Proposition~\ref{prop:instable_implique_hyp}, we can apply the conditional Theorem~\ref{theo:conditional_optimal} to prove Theorem~\ref{theo:optimal}.

\end{proof}

\begin{remark}
    By continuity of the spectrum, one also finds instability for $a\approx1$ and $\gamma\approx1/2.$
\end{remark}

\subsection{Polar vortex crystals}\label{sec:polar_VC}

On the poles of Jupiter (see \cite{Adriani2018Nature}, and see \cite{SieYouIng22} for numerical simulation on how those structure of vortices can emerge from turbulence, and \cite{Modin26} for a mathematical study of this phenomenon), are evolving rather stable large vortices arranged in a relative equilibrium of vortices, consisting in a polar vortex surrounded by equally distributed vortices around it. The stability properties of these vortex crystals have been studied, see for instance \cite{Cabral_Schmidt_1999_Stability_N+1_vortex}, and for some of them with well-chosen intensities, \cite{Donati_2025_Crystal} proved the improved confinement bound. 

On the sphere, we can easily construct a similar configuration by adding more vortices and taking $a$ close to 0, while ensuring the Gauss constraint with a vortex on the opposite pole. More precisely, consider the configuration consisting
\begin{align}\label{eq:VCN}
    \mathbf{x}_i^0 = \mathcal{R}_{\frac{2\pi}{N-2}i} \begin{pmatrix} a \\ 0 \\ s\end{pmatrix}, \quad \forall i \in \{1,\ldots,N-2\}, \quad \mathbf{x}_{N-1}^0 = \mathbf{e}_3, \quad \mathbf{x}_N^0 = \mathbf{-\mathbf{e}_3}
\end{align}
with $\Gamma_1 = \ldots = \Gamma_{N-2} = 1$, $\Gamma_{N-1} = \kappa$ and $\Gamma_N = -(N-2) - \kappa$. One can check that this configuration is a relative equilibrium of the point-vortex dynamics. In Figure~\ref{fig:sim}, a numerical simulation of this configuration with $N=8$ was made of the fluid equation, where very small viscosity is added for numerical purposes, which illustrates Theorem~\ref{theo:general_conf}: the vortices remain concentrated around the vortex crystal solution of the point-vortex dynamics.

\begin{figure}
    \centering
    \includegraphics[width=0.7\linewidth]{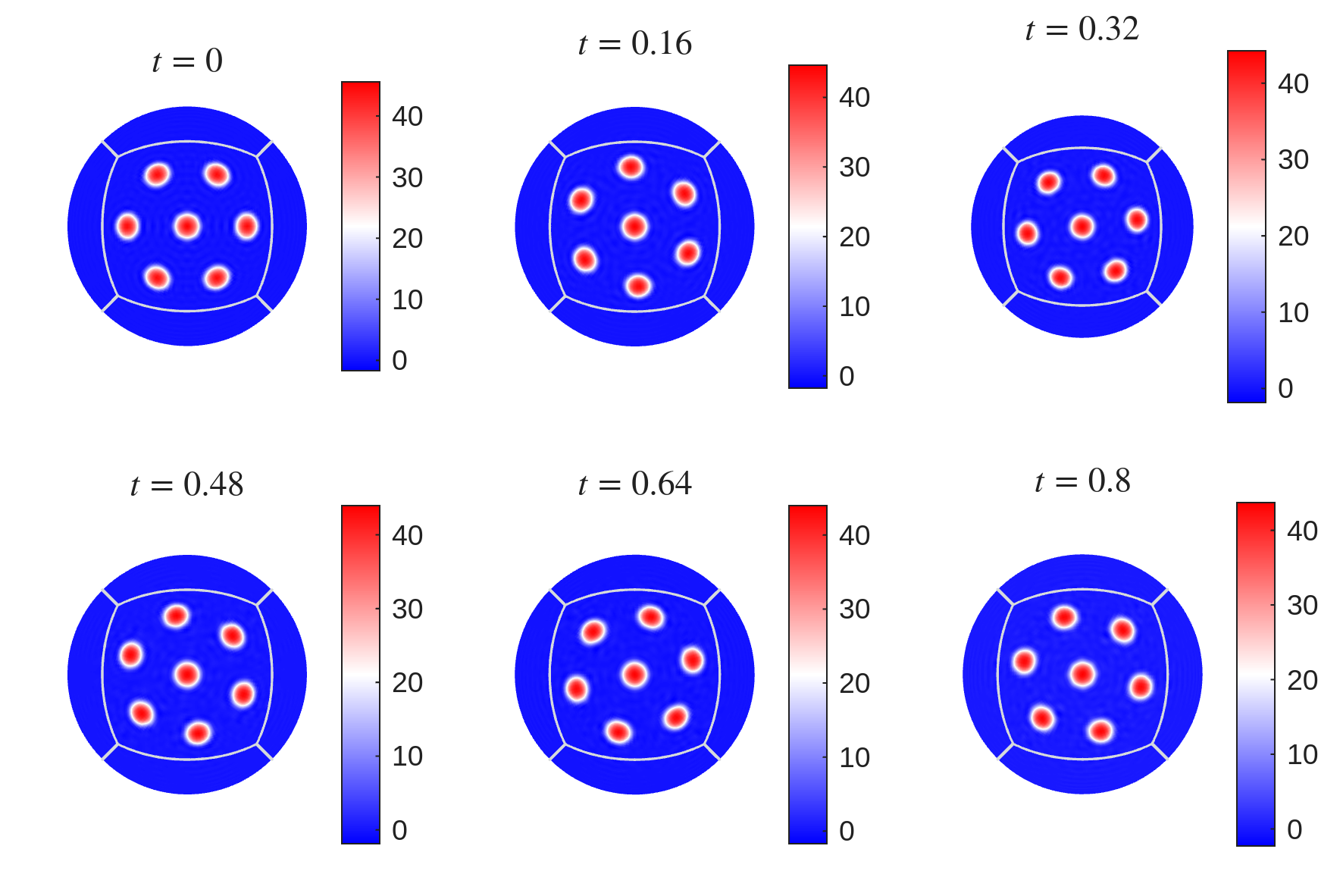}
    \caption{Numerical simulation of the Navier-Stokes equations on the non-rotating sphere with Reynold's number $10^5$, and an initial data concentrated near the vortex crystal~\eqref{eq:VCN} with $N=8$ and $\kappa =1$. Credit: Matthieu Brachet, Université de Poitiers, CNRS, LMA, Poitiers, France.}
    \label{fig:sim}
\end{figure}

Taking $N=4$, we recover the configuration constructed in Section~\ref{sec:config_instable}. We observe that in that case, with the stationary condition computed in Lemma~\ref{lem:stationary}, letting $a\to 0$, we find that $\kappa \to -\frac{1}{2}$, which leads to the well-known equilibria of three aligned vortices in the plane, which is unstable. One can check numerically that for small values of $a$, the configuration with the choice of $\kappa$ made in Section~\ref{sec:config_instable} is unstable for every value of $\gamma$ of order 1, when for larger values of $a$, non-trivial behavior appears, as observed in Figure~\ref{fig:a03}. The problem degenerates near $a=1$, see Figure~\ref{fig:a1}.
\begin{figure}
    \centering
    \includegraphics[width=0.45\linewidth]{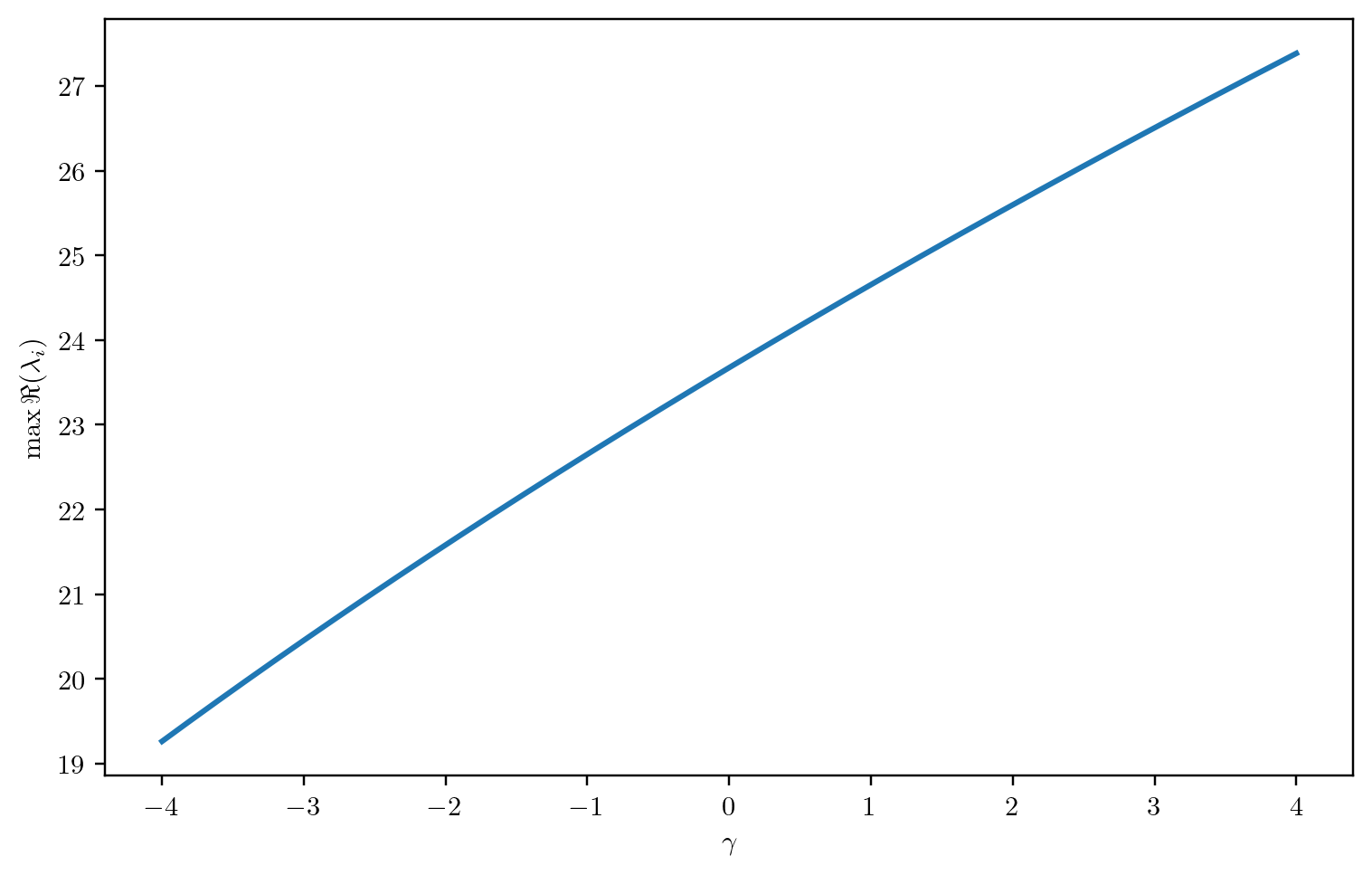} \hspace{2mm}
    \includegraphics[width=0.45\linewidth]{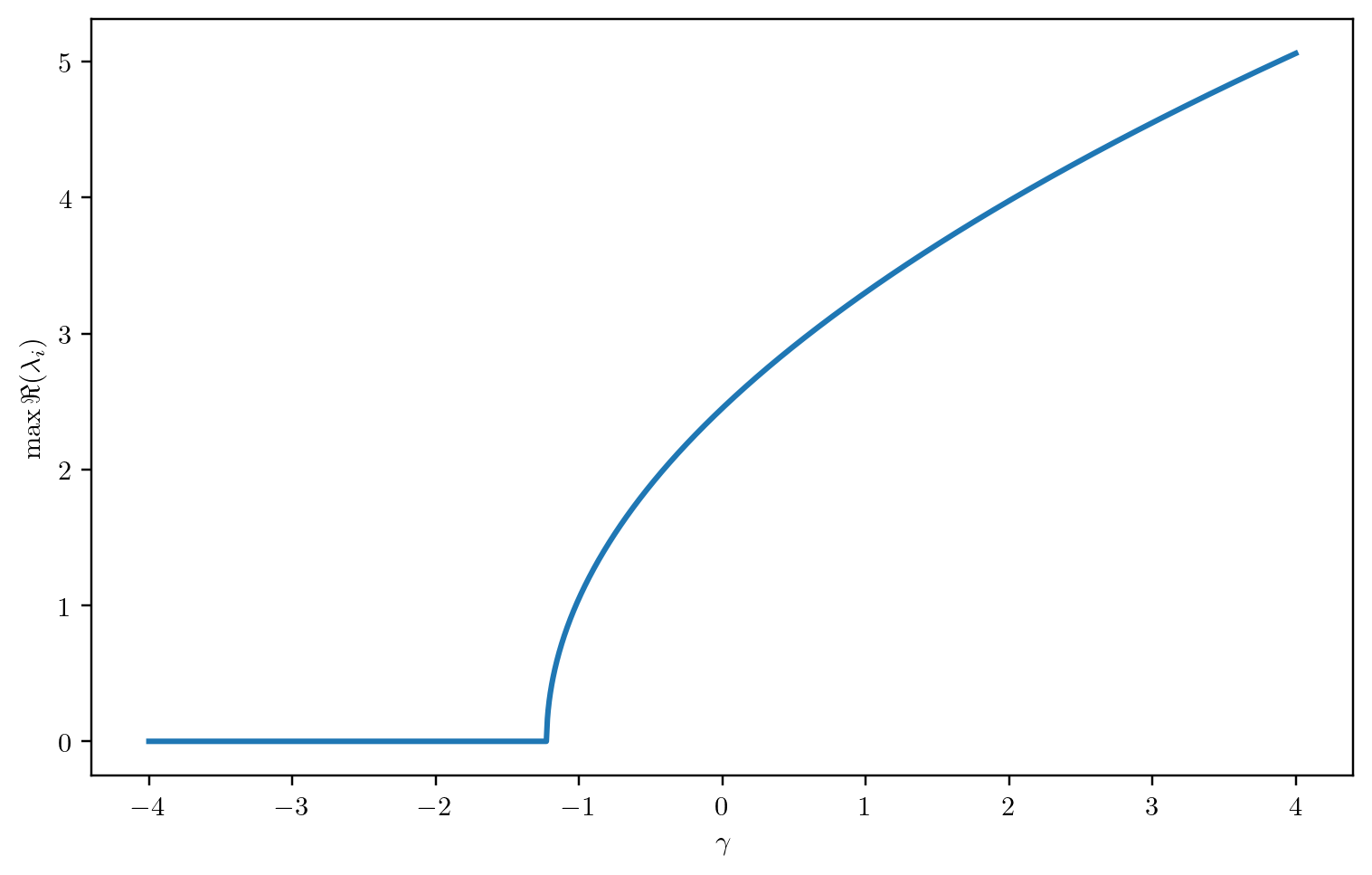} \\ 
    \caption{Maximum real part of the eigenvalues of $\Diff \mathcal{F}(\mathbf{X}^*)$ for the $N=4$ stationary configuration in the case $a=0.1$ (left), $a=0.3$ (right), depending on $\gamma$.}
    \label{fig:a03}
\end{figure}
\begin{figure}
    \centering
    \includegraphics[width=0.95\linewidth]{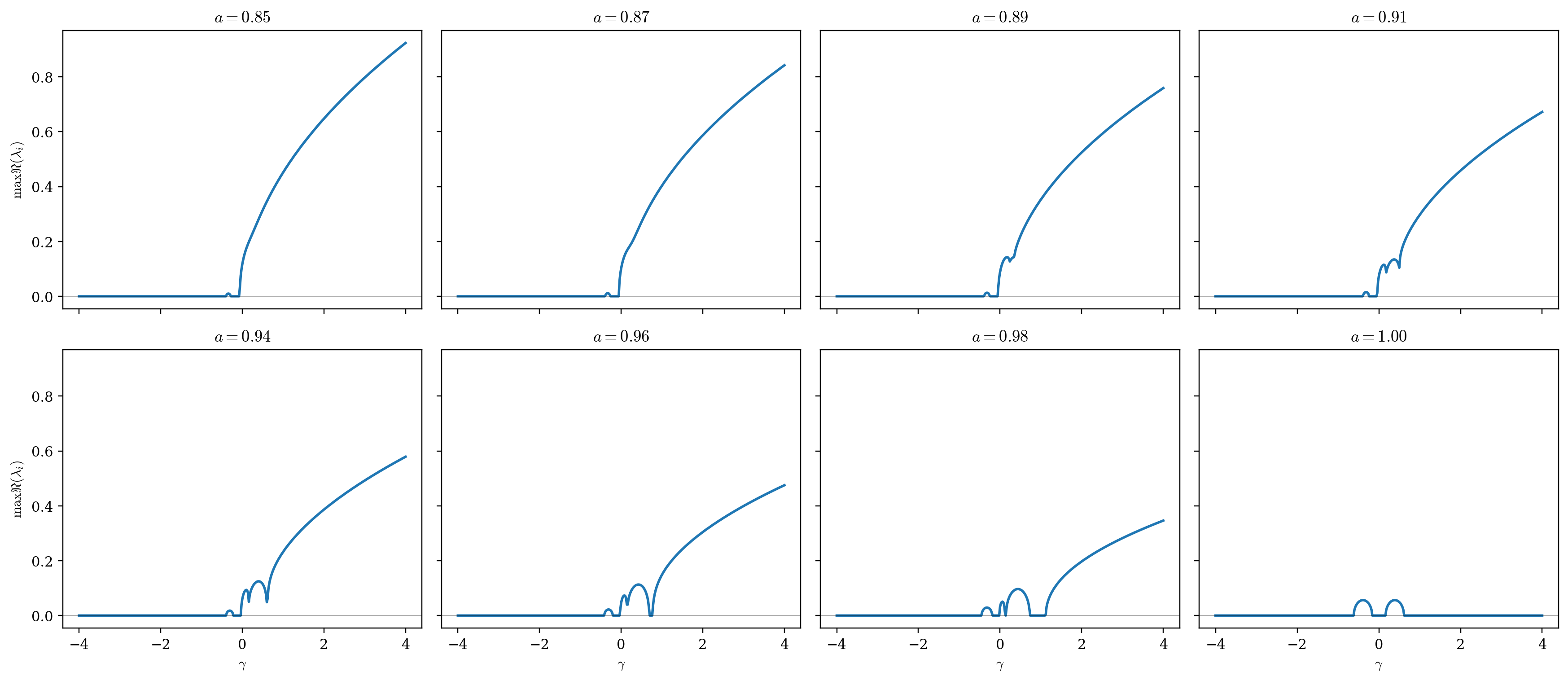}
    \caption{Maximum real part of the eigenvalues of $\Diff \mathcal{F}(\mathbf{X}^*)$ for the $N=4$ stationary configuration for values of $a$ close to 1, depending on $\gamma$.}
    \label{fig:a1}
\end{figure}
On the contrary, one can wonder whether for every value of $\gamma$, it is possible to find $\kappa$ such that Hypothesis~\ref{hyp:supstable} and~\ref{hyp:stable} are satisfied. Let us explore whether, for this polar vortex configuration, there exists a choice of parameters such that Hypothesis~\ref{hyp:supstable} holds. 
\begin{prop}
    The configuration~\eqref{eq:VCN} satisfies Hypothesis~\ref{hyp:supstable} if and only if
\begin{equation*}
    \kappa = \frac{1}{s} \left[ \frac{(N-3)(N-7)}{12} + (1-s^2) \frac{(N-3)}{4}  + (1-s)^2\right].
\end{equation*}
\end{prop}
\begin{proof}
Let us denote by
\begin{equation*}
    \mathcal{A}_i(h,t) = \sum_{j=1\atop j \neq i}^{N} \Gamma_j \Diff_1K_{\mathbb{S}^2}\big(\mathbf{x}_i(t),\mathbf{x}_j(t)\big)[h]\cdot h = 0.
\end{equation*}
Hypothesis~\ref{hyp:supstable} reads that for every $i \in \{1,\ldots,N\}$, for every  $ h \in T_{\mathbf{x}_i}\mathbb{S}^2$ and for every $t \geqslant 0$, 
\begin{equation*}
    \mathcal{A}_i(h,t) = 0.
\end{equation*}
Since the configuration is a relative equilibrium, it satisfies relation~\eqref{eq:VC}, and thus
\begin{align*}
    \mathcal{A}_i(h,t) & = \sum_{j=1\atop j \neq i}^{N} \Gamma_j \Diff_1K_{\mathbb{S}^2}\big(\mathbf{x}_i(t),\mathbf{x}_j(t)\big)[h]\cdot h \\
    & = \sum_{j=1\atop j \neq i}^{N} \Gamma_j\Diff_1K_{\mathbb{S}^2}\big(\mathcal{R}(\Omega t)\mathbf{x}_i^0,\mathcal{R}(\Omega t)\mathbf{x}_j^0\big)[h]\cdot h \\
    & = \sum_{j=1\atop j \neq i}^{N} \Gamma_j \mathcal{R}(\Omega t)\Diff_1K_{\mathbb{S}^2}\big(\mathcal{R}(\Omega t)\mathbf{x}_i^0,\mathcal{R}(\Omega t)\mathbf{x}_j^0\big)[ \mathcal{R}(-\Omega t)h]\cdot h \\
    & = \sum_{j=1\atop j \neq i}^{N} \Gamma_j \Diff_1K_{\mathbb{S}^2}\big(\mathcal{R}(\Omega t)\mathbf{x}_i^0,\mathcal{R}(\Omega t)\mathbf{x}_j^0\big)[ \mathcal{R}(-\Omega t)h]\cdot \mathcal{R}(-\Omega t) h \\
    & = \mathcal{A}_i(\mathcal{R}(-\Omega t) h,0),
\end{align*}
where we used that
\begin{equation*}
    \Diff_1K_{\mathbb{S}^2}\big(\mathcal{R}(\theta)\mathbf{x},\mathcal{R}(\theta)\mathbf{y}\big) = \mathcal{R}(\theta)\Diff_1K_{\mathbb{S}^2}(\mathbf{x},\mathbf{y})\mathcal{R}(-\theta),
\end{equation*}
which is obtained by differentiating the relation $K_{\mathbb{S}^2}(\mathcal{R}(\theta)\mathbf{x},\mathcal{R}(\theta)\mathbf{y}) = \mathcal{R}(\theta) K_{\mathbb{S}^2}(\mathbf{x},\mathbf{y}) $, and that
\begin{equation*}
    \mathcal{R}(\theta) h \cdot h' =  h \cdot \mathcal{R}(-\theta) h'.
\end{equation*}
In conclusion, we have that
\begin{equation*}
     \Big( \forall h \in T_{\mathbf{x}_i}\mathbb{S}^2, \; \forall t \geqslant 0, \; \mathcal{A}_i(h,t) = 0 \Big) \Longleftrightarrow  \Big( \forall h \in T_{\mathbf{x}_i}\mathbb{S}^2, \; \mathcal{A}_i(h,0) = 0 \Big).
\end{equation*}
Let us now consider the cases $i \in \{N-1,N\}$. Recall that $\mathbf{x}_i^0 = \pm \mathbf{e}_3$.  
Define the scalar potential
\begin{equation*}
    \Phi_i(\mathbf{x}) \triangleq \Gamma \sum_{j=1}^{N-2} G_{\mathbb{S}^2}(\mathbf{x},\mathbf{x}_j^0)+C_0,
\end{equation*}
where $G_{\mathbb{S}^2}(\mathbf{x},\mathbf{y}) = \ln ( |\mathbf{x}-\mathbf{y}|_{\R^3})$ is (up to a multiplication and addition of constants) the Green function of $\Delta_{\mathbb{S}^2}$ and $C_0\triangleq C_0(\Gamma,N,\mathbf{x}_1^0,...,\mathbf{x}_{N-2}^0)\in\mathbb{R}$ is a constant chosen so that
$$\int_{\mathbb{S}^2}\Phi_i(\mathbf{x})d\boldsymbol{\sigma}(\mathbf{x})=0.$$ In a neighbourhood of $\mathbf{x}_i^0$, the function $\Phi_i$ is smooth and satisfies
\begin{equation*}
    \Delta_{\mathbb{S}^2} \Phi_i = 0,
\end{equation*}
so that $\Phi_i$ is harmonic near $\mathbf{x}_i^0$.
Using now the cyclic symmetry of the configuration, we have
\begin{equation*}
    \mathcal{R}\Big(\frac{2\pi}{N-2}\Big)\mathbf{x}_j^0 =
    \begin{cases}
        \mathbf{x}_{j+1}^0, & \forall j \in \{1,\ldots,N-3\}, \\
        \mathbf{x}_1^0, & \text{ if } j = N-2
    \end{cases}
\end{equation*}
and
\begin{equation*}
    \mathcal{R}\Big(\frac{2\pi}{N-2}\Big)\mathbf{x}_i^0 = \mathbf{x}_i^0.
\end{equation*}
Therefore the potential $\Phi_i$ is invariant under this rotation, namely
\begin{equation*}
    \Phi_i\Big(\mathcal{R}\Big(\frac{2\pi}{N-2}\Big)x\Big) = \Phi_i(x).
\end{equation*}
Since $\mathbf{x}_i^0$ is a fixed point of the rotation, this invariance implies
\begin{equation*}
    \nabla_{\mathbb{S}^2} \Phi_i(\mathbf{x}_i^0) = 0 .
\end{equation*}
Let
\begin{equation*}
    H_i\triangleq\Diff^2_{\mathbb{S}^2}\Phi_i(\mathbf{x}_i^0)
\end{equation*}
denote the Hessian restricted to the tangent space $T_{\mathbf{x}_i^0}\mathbb{S}^2$.  
Since $\Phi_i$ is invariant under the rotation $\mathcal{R}\big(\frac{2\pi}{N-2}\big)$, the matrix $H_i$ commutes with this rotation acting on $T_{\mathbf{x}_i^0}\mathbb{S}^2$. As $H_i$ is symmetric, this implies that
\begin{equation*}
    H_i = \mu_i I_2
\end{equation*}
for some $\mu_i \in \mathbb{R}$.
Using the fact that $\Phi_i$ is harmonic, we obtain
\begin{equation*}
    0 = \Delta_{\mathbb{S}^2} \Phi_i(\mathbf{x}_i^0) = \mathrm{tr}(H_i) = 2\mu_i.
\end{equation*}
Hence,
\begin{equation*}
    H_i = 0.
\end{equation*}
Recall now that
\begin{equation*}
    K_{\mathbb{S}^2}(\mathbf{x},\mathbf{y}) = \nabla_{\mathbf{x}}^\perp G_{\mathbb{S}^2}(\mathbf{x},\mathbf{y}),
\end{equation*}
so that
\begin{equation*}
    \Gamma \sum_{j=1}^{N-2} K_{\mathbb{S}^2}(\mathbf{x},\mathbf{x}_j^0) = \nabla_x^\perp \Phi_i(x).
\end{equation*}
Differentiating with respect to $\mathbf{x}$ at $\mathbf{x}=\mathbf{x}_i^0$ gives
\begin{equation*}
    \Gamma \sum_{j=1}^{N-2} \Diff_1 K_{\mathbb{S}^2}(\mathbf{x}_i^0,\mathbf{x}_j^0)[h] = \Diff \big(\nabla^\perp \Phi_i\big)(\mathbf{x}_i^0)[h].
\end{equation*}
Since $\nabla\Phi_i(\mathbf{x}_i^0)=0$, as the configuration is a relative equilibrium where $\mathbf{x}_i^0$ lies on the axis of rotation, hence is constant, the above differential reduces to the rotated Hessian and therefore vanishes because $H_i=0$. Consequently,
\begin{equation*}
    \Gamma \sum_{j=1}^{N-2}\Diff_1 K_{\mathbb{S}^2}(\mathbf{x}_i^0,\mathbf{x}_j^0)[h] = 0
\end{equation*}
for all $h \in T_{\mathbf{x}_i^0}\mathbb{S}^2$, and thus\footnote{The previous argument is missing in the equivalent proof in \cite{Donati_2025_Crystal}. It also holds in the planar case.}
\begin{equation*}
    \mathcal{A}_i(h,0) = 0.
\end{equation*}
We now turn to the case $i \in \{1,\ldots,N-2\}$. Set
\begin{equation*}
    n\triangleq N-2, \qquad \varphi \triangleq \frac{2\pi}{n}\cdot
\end{equation*}
For every $i \in \{1,\ldots,n\}$, we have
\begin{equation*}
    \mathbf{x}_i^0 = \mathcal{R}\big((i-1)\varphi\big)\mathbf{x}_1^0.
\end{equation*}
Since the two polar vortices are fixed by every rotation around $\mathbf{e}_3$, and since
\begin{equation*}
    \mathbf{x}_{j+i-1}^0 = \mathcal{R}\big((i-1)\varphi\big)\mathbf{x}_j^0
\end{equation*}
for every $j \in \{1,\ldots,n\}$, the same covariance argument as before gives, for every
$h \in T_{\mathbf{x}_i^0}\mathbb{S}^2$,
\begin{align*}
    \mathcal{A}_i(h,0)
    &= \sum_{j=1\atop j\neq i}^{N} \Gamma_j
    \Diff_1K_{\mathbb{S}^2}\big(\mathbf{x}_i^0,\mathbf{x}_j^0\big)[h]\cdot h \\
    &= \sum_{j=1\atop j\neq 1}^{N} \Gamma_j
    \Diff_1K_{\mathbb{S}^2}\big(\mathbf{x}_1^0,\mathbf{x}_j^0\big)
    \big[\mathcal{R}(-(i-1)\varphi)h\big]\cdot \mathcal{R}(-(i-1)\varphi)h \\
    &= \mathcal{A}_1\big(\mathcal{R}(-(i-1)\varphi)h,0\big).
\end{align*}
Therefore,
\begin{equation*}
    \Big( \forall i \in \{1,\ldots,n\},\ \forall h \in T_{\mathbf{x}_i^0}\mathbb{S}^2, \quad \mathcal{A}_i(h,0)=0 \Big) \Longleftrightarrow \Big( \forall h \in T_{\mathbf{x}_1^0}\mathbb{S}^2,\quad \mathcal{A}_1(h,0)=0 \Big).
\end{equation*}
It remains to compute $\mathcal{A}_1(h,0)$.  
We introduce the orthonormal basis of $T_{\mathbf{x}_1^0}\mathbb{S}^2$
\begin{equation*}
    \tau_1\triangleq\frac{\mathbf{e}_3\wedge \mathbf{x}_1^0}{|\mathbf{e}_3\wedge \mathbf{x}_1^0|_{\mathbb{R}^3}} = \begin{pmatrix} 0 \\ 1 \\ 0 \end{pmatrix}, \qquad \nu_1 \triangleq \mathbf{x}_1^0 \wedge \tau_1 = \begin{pmatrix} -s \\ 0 \\ a \end{pmatrix}.
\end{equation*}
Thus every $h \in T_{\mathbf{x}_1^0}\mathbb{S}^2$ can be written uniquely as
\begin{equation*}
    h = u \tau_1 + v \nu_1, \qquad u,v \in \mathbb{R}.
\end{equation*}
We split
\begin{equation*}
    \mathcal{A}_1(h,0) = \sum_{j=2}^{n}\Diff_1K_{\mathbb{S}^2}(\mathbf{x}_1^0,\mathbf{x}_j^0)[h]\cdot h + \kappa\,\Diff_1K_{\mathbb{S}^2}(\mathbf{x}_1^0,\mathbf{e}_3)[h]\cdot h - (n+\kappa)\,\Diff_1K_{\mathbb{S}^2}(\mathbf{x}_1^0,-\mathbf{e}_3)[h]\cdot h.
\end{equation*}
Using relation~\eqref{eq:differential}, we have
\begin{equation*}
    \Diff_1K_{\mathbb{S}^2}(\mathbf{x},\mathbf{y})[h]\cdot h = -2\,\frac{h\cdot(\mathbf{x}-\mathbf{y})\; h\cdot(\mathbf{x}\wedge \mathbf{y})}{|\mathbf{x}-\mathbf{y}|_{\mathbb{R}^3}^4}.
\end{equation*}

We first compute the contributions of the two poles. Since
\begin{equation*}
    \mathbf{x}_1^0-\mathbf{e}_3 = \begin{pmatrix} a \\ 0 \\ s-1 \end{pmatrix},
    \qquad  \mathbf{x}_1^0\wedge \mathbf{e}_3 = \begin{pmatrix} 0 \\ -a \\ 0 \end{pmatrix},
\end{equation*}
we obtain
\begin{equation*}
    h\cdot(\mathbf{x}_1^0-\mathbf{e}_3)=-av,
    \qquad
    h\cdot(\mathbf{x}_1^0\wedge \mathbf{e}_3)=-au.
\end{equation*}
Hence (here and in the sequel, we use the notations $\alpha_-$ and $\alpha_+$ as in the previous subsection),
\begin{equation*}
    \Diff_1K_{\mathbb{S}^2}(\mathbf{x}_1^0,\mathbf{e}_3)[h]\cdot h = -2a^2\alpha_-^2\,uv.
\end{equation*}
Similarly, since
\begin{equation*}
    \mathbf{x}_1^0+\mathbf{e}_3
    = \begin{pmatrix} a \\ 0 \\ s+1 \end{pmatrix},
    \qquad
    \mathbf{x}_1^0\wedge (-\mathbf{e}_3) = \begin{pmatrix} 0 \\ a \\ 0 \end{pmatrix},
\end{equation*}
we get
\begin{equation*}
    h\cdot(\mathbf{x}_1^0+\mathbf{e}_3)=av, \qquad h\cdot\big(\mathbf{x}_1^0\wedge (-\mathbf{e}_3)\big)=au,
\end{equation*}
and therefore
\begin{equation*}
    \Diff_1K_{\mathbb{S}^2}(\mathbf{x}_1^0,-\mathbf{e}_3)[h]\cdot h = -2a^2\alpha_+^2\,uv.
\end{equation*}
Thus the polar contribution is
\begin{equation*}
    -2a^2\Big(\kappa\alpha_-^2-(n+\kappa)\alpha_+^2\Big)\,uv.
\end{equation*}
We now consider the contribution of the $n-1$ vortices on the ring. For every
$m \in \{1,\ldots,n-1\}$, let
\begin{equation*}
    \theta_m \triangleq m\varphi = \frac{2\pi m}{n}\cdot
\end{equation*}
Then
\begin{equation*}
    \mathbf{x}_{1+m}^0 = \begin{pmatrix} a\cos\theta_m \\ a\sin\theta_m \\ s \end{pmatrix},
\end{equation*}
where the index $m$ is understood modulo $n$. We get that
\begin{equation*}
    h\cdot(\mathbf{x}_1^0-\mathbf{x}_{1+m}^0) = -a\Big(sv(1-\cos\theta_m) + u\sin\theta_m\Big),
\end{equation*}
and
\begin{equation*}
    h\cdot(\mathbf{x}_1^0\wedge \mathbf{x}_{1+m}^0) = a\Big(-su(1-\cos\theta_m)+v\sin\theta_m\Big).
\end{equation*}
Moreover,
\begin{equation*}
    |\mathbf{x}_1^0-\mathbf{x}_{1+m}^0|_{\mathbb{R}^3}^2 = 2a^2(1-\cos\theta_m).
\end{equation*}
Therefore,
\begin{equation*}
    \Diff_1K_{\mathbb{S}^2}(\mathbf{x}_1^0,\mathbf{x}_{1+m}^0)[h]\cdot h = \frac{1}{2a^2} \frac{uv \big(-s^2(1-\cos \theta_m)^2 + \sin^2 \theta_m \big) + s(v^2-u^2)(1-\cos\theta_m)\sin(\theta_m)}{(1-\cos\theta_m)^2}\cdot
\end{equation*}
The idea is now to pair the terms corresponding to $m$ and $n-m$. Since $\cos \theta_m = \cos\theta_{-m}$ and $\sin \theta_m = - \sin(-\theta_m)$, one obtains
\begin{align*}
    \Diff_1K_{\mathbb{S}^2}(\mathbf{x}_1^0,\mathbf{x}_{1+m}^0)[h]\cdot h + \Diff_1K_{\mathbb{S}^2}(\mathbf{x}_1^0,\mathbf{x}_{1-m}^0)[h]\cdot h & = \frac{uv}{a^2}\frac{-s^2(1-\cos\theta_m)^2 + \sin^2 \theta_m}{(1-\cos\theta_m)^2}\\
    & = \frac{1}{a^2}\left(\cot^2\Big(\frac{m\pi}{n}\Big)-s^2\right)uv.
\end{align*}
Summing over all $m \in \{1,\ldots,n-1\}$, this yields
\begin{equation*}
    \sum_{j=2}^{n}\Diff_1K_{\mathbb{S}^2}(\mathbf{x}_1^0,\mathbf{x}_j^0)[h]\cdot h = \frac{uv}{2a^2}\sum_{m=1}^{n-1}\left(\cot^2\Big(\frac{m\pi}{n}\Big)-s^2\right).
\end{equation*}
Using the classical identity
\begin{equation*}
    \sum_{m=1}^{n-1}\cot^2\Big(\frac{m\pi}{n}\Big) = \frac{(n-1)(n-2)}{3},
\end{equation*}
we obtain
\begin{equation*}
    \sum_{j=2}^{n}\Diff_1K_{\mathbb{S}^2}(\mathbf{x}_1^0,\mathbf{x}_j^0)[h]\cdot h = \frac{n-1}{2a^2}\left(\frac{n-2}{3}-s^2\right)uv.
\end{equation*}
Combining the ring and polar contributions, we finally get
\begin{equation*}
    \mathcal{A}_1(h,0) = \left[\frac{n-1}{2a^2}\left(\frac{n-2}{3}-s^2\right) + 2a^2\Big(\kappa\alpha_-^2+(n+\kappa)\alpha_+^2\Big)
    \right]uv.
\end{equation*}
Therefore,
\begin{equation*}
    \forall h \in T_{\mathbf{x}_1^0}\mathbb{S}^2,\quad \mathcal{A}_1(h,0)=0
\end{equation*}
if and only if
\begin{equation*}
    \frac{n-1}{2a^2}\left(s^2-\frac{n-2}{3}\right) - 2a^2\Big(\kappa\alpha_-^2-(n+\kappa)\alpha_+^2\Big)=0.
\end{equation*}
Since $\alpha_-^2 - \alpha_+^2 = -\frac{s}{a^4}$, equivalently,
\begin{equation*}
 2\kappa\frac{s}{a^2} = \frac{n-1}{2a^2}\left(\frac{n-2}{3}-s^2\right) - \frac{1}{2}\frac{a^2n}{(1+s)^2}\cdot
\end{equation*}
Therefore, recalling that $a^2 = 1-s^2$, we infer
\begin{equation*}
    \kappa = \frac{1}{4s} \left[ (n-1)\left(\frac{n-2}{3}-s^2\right) - (1-s)^2 \right].
\end{equation*}
In order to make to limit $s \to 1$ well visible, we factor and obtain that 
\begin{equation*}
    \kappa = \frac{1}{s} \left[ \frac{(n-1)(n-5)}{12} + (1-s^2) \frac{(n-1)}{4}  + (1-s)^2\right],
\end{equation*}
which gives the desired result since $n=N-2$.
\end{proof}
When $s \to 1$, we recover the condition
\begin{equation*}
    \kappa = \frac{(n-1)(n-5)}{12}
\end{equation*}
which corresponds to the result of \cite{Donati_2025_Crystal} for the vortex crystal with $n$ outer vortices. To then prove that with this choice of $\kappa$, the configuration satisfies the power-law confinement bound, from Proposition~\ref{prop:stable_implique_hypstable}, it is enough to check that the stability matrix $\Diff \mathcal{F} (\mathbf{X}^0)$ has no eigenvalue with positive real part. In particular, it does not depend on $\gamma$, so the results of \cite{LPMR11} on the non-rotating sphere could apply. However, for this specific configuration, \cite{LPMR11} does not provide a sufficient and necessary condition for the spectral instability, so one can not conclude, but the interested reader can find in \cite{LPMR11} other related details on the stability of vortex equilibria on the sphere. We also refer to \cite{Roberts_2013_Stability_of_Relative_Eq} for a general method on how to compute the stability of point-vortex configurations (in the plane).

\appendix
\section{Technical lemmas}
In this appendix, we gather some technical results used throughout the manuscript.
\begin{lem}
For every $\mathbf{x},\mathbf{y} \in \mathbb{S}^2$, there holds that
        \begin{equation}\label{eq:maj_x_wedge_y}
        |\mathbf{x} \wedge \mathbf{y}|_{\R^3} \leqslant |\mathbf{x} - \mathbf{y}|_{\R^3}.
    \end{equation}
\end{lem}
\begin{proof}
    First of all, take $\mathbf{x},\mathbf{y}\in\mathbb{S}^2.$ On one hand
    $$|\mathbf{x} \wedge \mathbf{y}|_{\R^3}^2=|\mathbf{x}|_{\R^3}^2|\mathbf{y}|_{\R^3}^2-(\mathbf{x}\cdot\mathbf{y})^2=1-(\mathbf{x}\cdot\mathbf{y})^2=(1-\mathbf{x}\cdot\mathbf{y})(1+\mathbf{x}\cdot\mathbf{y}).$$
    On the other hand
    \begin{equation}\label{norm-diff-carre}
        |\mathbf{x}-\mathbf{y}|_{\R^3}^2=|\mathbf{x}|_{\R^3}^2+|\mathbf{y}|_{\R^3}^2-2\mathbf{x}\cdot\mathbf{y}=2(1-\mathbf{x}\cdot\mathbf{y}).
    \end{equation}
    In addition, by the Cauchy-Schwarz inequality,
    $$1+\mathbf{x}\cdot\mathbf{y}\leqslant1+|\mathbf{x}|_{\R^3}|\mathbf{y}|_{\R^3}=2.$$
    Combining the foregoing calculations leads to relation~\eqref{eq:maj_x_wedge_y}.
\end{proof}
\begin{lem}[Biot-Savart law on the rotating sphere]\label{lem:gradlog}
For any $\mathbf{x},\mathbf{y}\in\mathbb{S}^2,$ we have
    $$\nabla_{\mathbf{x}}^{\perp}(\mathbf{x}\cdot\mathbf{y})=-\mathbf{x}\wedge\mathbf{y}\qquad\textnormal{and}\qquad\nabla_{\mathbf{x}}^{\perp}\big(\ln|\mathbf{x}-\mathbf{y}|_{\mathbb{R}^3}\big)=\frac{\mathbf{x}\wedge\mathbf{y}}{\,\,|\mathbf{x}-\mathbf{y}|_{\mathbb{R}^3}^2}\cdot$$
    Consequently, one obtains the following Biot-Savart law on the rotating sphere at speed $\gamma$
    $$\forall\mathbf{x}\in\mathbb{S}^2,\quad u(t,\mathbf{x})=\frac{1}{2\pi}\int_{\mathbb{S}^2}\frac{\mathbf{x}\wedge\mathbf{y}}{\,\,|\mathbf{x}-\mathbf{y}|_{\mathbb{R}^3}^2}\zeta(t,\mathbf{y})d\boldsymbol{\sigma}(\mathbf{y})+\gamma\mathbf{e}_3\,\wedge\mathbf{x}.$$
\end{lem}
\begin{proof}
Let us first recall from \eqref{norm-diff-carre} that for $\mathbf{x},\mathbf{y}\in\mathbb{S}^2,$
$$|\mathbf{x}-\mathbf{y}|_{\mathbb{R}^3}^2=|\mathbf{x}|_{\mathbb{R}^3}^2+|\mathbf{y}|_{\mathbb{R}^3}^2-2\mathbf{x}\cdot\mathbf{y}=2(1-\mathbf{x}\cdot\mathbf{y}).$$
Therefore,
$$\nabla_{\mathbf{x}}^{\perp}\left(\ln|\mathbf{x}-\mathbf{y}|_{\mathbb{R}^3}\right)=\frac{1}{2}\nabla_{\mathbf{x}}^{\perp}\left(\ln|\mathbf{x}-\mathbf{y}|_{\mathbb{R}^3}^2\right)=-\frac{\nabla_{\mathbf{x}}^{\perp}(\mathbf{x}\cdot\mathbf{y})}{\,\,|\mathbf{x}-\mathbf{y}|_{\mathbb{R}^3}^2}\cdot$$
The unit sphere is a manifold with principal co-latitude/longitude local chart
$$\psi_1(\theta,\varphi)\triangleq\big(\sin(\theta)\cos(\varphi),\sin(\theta)\sin(\varphi),\cos(\theta)\big),\qquad\theta\in(0,\pi),\qquad\varphi\in(0,2\pi).$$
The interested reader can perform the same computations working in an other local chart, covering the missing points, and get the same result.
On each point $\mathbf{x}\in\mathbb{S}^2$ in the form $\mathbf{x}=\psi_1(\theta,\varphi)$ for some $(\theta,\varphi)\in(0,\pi)\times(0,2\pi),$ the tangent space $T_{\mathbf{x}}\mathbb{S}^2$ has an orthonormal basis $\mathbf{e}=\big(\mathbf{e}_{\theta}(\mathbf{x}),\mathbf{e}_{\varphi}(\mathbf{x})\big)$ given by
$$\mathbf{e}_{\theta}(\mathbf{x})\triangleq\partial_{\theta}\psi_1(\theta,\varphi)=\begin{pmatrix}
    \cos(\theta)\cos(\varphi)\\
    \cos(\theta)\sin(\varphi)\\
    -\sin(\theta)
\end{pmatrix},\qquad\mathbf{e}_{\varphi}(\mathbf{x})\triangleq\frac{\partial_{\varphi}\psi_1(\theta,\varphi)}{\sin(\theta)}=\begin{pmatrix}
    -\sin(\varphi)\\
    \cos(\varphi)\\
    0
\end{pmatrix}.$$
The operator $\nabla^{\perp}$ is defined  by 
$$\nabla^{\perp}f(\theta,\varphi)=\frac{\partial_{\varphi}f(\theta,\varphi)}{\sin(\theta)}\mathbf{e}_{\theta}-\partial_{\theta}f(\theta,\varphi)\mathbf{e}_{\varphi}.$$
We denote $\mathbf{y}=\psi_1(\theta',\varphi'),$ then
\begin{align*}
    \nabla^{\perp}_{\mathbf{x}}(\mathbf{x}\cdot\mathbf{y})&=\nabla^{\perp}\big(\psi_1(\theta,\varphi)\cdot\mathbf{y}\big)\\
    &=-\big(\mathbf{e}_{\theta}(\mathbf{x})\cdot\mathbf{y}\big)\mathbf{e}_{\varphi}(\mathbf{x})+\big(\mathbf{e}_{\varphi}(\mathbf{x})\cdot\mathbf{y}\big)\mathbf{e}_{\theta}(\mathbf{x}).
\end{align*}
But
\begin{align*}
    \mathbf{e}_{\varphi}(\mathbf{x})\cdot\mathbf{y}=\begin{pmatrix}
        -\sin(\varphi)\\
        \cos(\varphi)\\
        0
    \end{pmatrix}\cdot\begin{pmatrix}
        \sin(\theta')\cos(\varphi')\\
        \sin(\theta')\sin(\varphi')\\
        \cos(\theta')
    \end{pmatrix}&=\cos(\varphi)\sin(\theta')\sin(\varphi')-\sin(\varphi)\sin(\theta')\cos(\varphi')
\end{align*}
and
\begin{align*}
    \mathbf{e}_{\theta}(\mathbf{x})\cdot\mathbf{y}&=\begin{pmatrix}
    \cos(\theta)\cos(\varphi)\\
    \cos(\theta)\sin(\varphi)\\
    -\sin(\theta)
\end{pmatrix}\cdot\begin{pmatrix}
        \sin(\theta')\cos(\varphi')\\
        \sin(\theta')\sin(\varphi')\\
        \cos(\theta')
    \end{pmatrix}\\
    &=\cos(\theta)\cos(\varphi)\sin(\theta')\cos(\varphi')+\cos(\theta)\sin(\varphi)\sin(\theta')\sin(\varphi')-\sin(\theta)\cos(\theta').
\end{align*}
Therefore, after simplifications, we find
\begin{align*}
    \nabla_{\mathbf{x}}^{\perp}(\mathbf{x}\cdot\mathbf{y})&=-\begin{pmatrix}
        \sin(\theta)\sin(\varphi)\cos(\theta')-\cos(\theta)\sin(\theta')\sin(\varphi')\\
        \cos(\theta)\sin(\theta')\cos(\varphi')-\sin(\theta)\cos(\varphi)\cos(\theta')\\
        \sin(\theta)\cos(\varphi)\sin(\theta')\sin(\varphi')-\sin(\theta)\sin(\varphi)\sin(\theta')\cos(\varphi')
    \end{pmatrix}.
\end{align*}
Besides,
\begin{align*}
    \mathbf{x}\wedge\mathbf{y}&=\begin{pmatrix}
        \sin(\theta)\cos(\varphi)\\
        \sin(\theta)\sin(\varphi)\\
        \cos(\theta)
    \end{pmatrix}\wedge\begin{pmatrix}
        \sin(\theta')\cos(\varphi')\\
        \sin(\theta')\sin(\varphi')\\
        \cos(\theta')
    \end{pmatrix}\\
    &=\begin{pmatrix}
        \sin(\theta)\sin(\varphi)\cos(\theta')-\cos(\theta)\sin(\theta')\sin(\varphi')\\
        \cos(\theta)\sin(\theta')\cos(\varphi')-\sin(\theta)\cos(\varphi)\cos(\theta')\\
        \sin(\theta)\cos(\varphi)\sin(\theta')\sin(\varphi')-\sin(\theta)\sin(\varphi)\sin(\theta')\cos(\varphi')
    \end{pmatrix}.
\end{align*}
Thus, we have proven
\begin{equation}\label{gradperp:scalarprod}
    \nabla_{\mathbf{x}}^{\perp}(\mathbf{x}\cdot\mathbf{y})=-\mathbf{x}\wedge\mathbf{y}.
\end{equation}
Consequently,
$$\nabla_{\mathbf{x}}^{\perp}\big(\ln|\mathbf{x}-\mathbf{y}|_{\mathbb{R}^3}\big)=\frac{\mathbf{x}\wedge\mathbf{y}}{\,\,|\mathbf{x}-\mathbf{y}|_{\mathbb{R}^3}^2}\cdot$$
The proof of Lemma \ref{lem:gradlog} is now complete.
\end{proof}

Throughout the document, give $\theta\in\mathbb{R}$, the direct rotation of angle $\theta$ around the vertical axis is denoted
\begin{equation}\label{def:Rtheta}
    \mathcal{R}(\theta)\triangleq\begin{pmatrix}
            \cos(\theta) & -\sin(\theta) & 0\\
            \sin(\theta) & \cos(\theta) & 0\\
            0 & 0 & 1
        \end{pmatrix}.
\end{equation}
We have the following characterization of uniformly rotating point vortex configurations (vortex crystals).
 \begin{lem}\label{lem rot sph}
        Let $\big((\mathbf{x}_1(t),,\Gamma_1)\ldots,(\mathbf{x}_N(t),\Gamma_N)\big)$ be a point vortex dynamical system on $\mathbb{S}^2$. Then, this system performs a uniform rotation around the vertical axis at constant speed $\Omega\in\mathbb{R}$ if and only if for any $i\in\{1,\ldots,N\},$
        $$\big(\Omega-\gamma)\mathbf{e}_3\wedge\mathbf{x}_i(0)=\sum_{j=1\atop j\neq i}^{N}\frac{\Gamma_j}{2\pi}\frac{\mathbf{x}_i(0)\wedge\mathbf{x}_{j}(0)}{|\mathbf{x}_i(0)-\mathbf{x}_j(0)|_{\mathbb{R}^3}^2}\cdot$$
    \end{lem}
    \begin{remark}
        The previous result gives in particular that two antipodal points with opposite circulations might rotate at the sphere rotation speed $\gamma$. 
    \end{remark}
    \begin{proof}
        Observe that a point $\mathbf{x}(t)$ on $\mathbb{S}^2$ performs a uniform rotation around the vertical axis at constant speed $\Omega$ if and only if
    $$\mathbf{x}(t)=\mathcal{R}(\Omega t)\mathbf{x}(0).$$
    Therefore,
    $$\frac{d}{dt}\mathbf{x}(t)=\partial_{t}\big(\mathcal{R}(\Omega t)\big)\mathbf{x}(0)=\Omega\begin{pmatrix}
        -\sin(\Omega t) & -\cos(\Omega t) & 0\\
        \cos(\Omega t) & -\sin(\Omega t) & 0\\
        0 & 0 & 0
    \end{pmatrix}\mathbf{x}(0).$$
    Denoting $\mathbf{x}(0)\triangleq(x_1^{0},x_2^0,x_3^0)^{\top},$ we get
    $$\frac{d}{dt}\mathbf{x}(t)=-\Omega\begin{pmatrix}
        x_1^0\sin(\Omega t)+x_2^0\cos(\Omega t)\\
        x_2^0\sin(\Omega t)-x_1^0\cos(\Omega t)\\
        0
    \end{pmatrix}.$$
    Now notice that
    $$\mathcal{R}(\Omega t)\big(\mathbf{e}_3\wedge\mathbf{x}(0)\big)=\mathbf{e}_3\wedge\mathcal{R}(\Omega t)\mathbf{x}(0)=\begin{pmatrix}
        0\\
        0\\
        1
    \end{pmatrix}\wedge\begin{pmatrix}
        x_1^0\cos(\Omega t)-x_2^0\sin(\Omega t)\\
        x_1^0\sin(\Omega t)+x_2^0\cos(\Omega t)\\
        x_3^0
    \end{pmatrix}=-\begin{pmatrix}
        x_1^0\sin(\Omega t)+x_2^0\cos(\Omega t)\\
        x_2^0\sin(\Omega t)-x_1^0\cos(\Omega t)\\
        0
    \end{pmatrix}.$$
    Thus, we get
    $$\frac{d}{dt}\mathbf{x}(t)=\Omega\,\mathcal{R}(\Omega t)\big(\mathbf{e}_3\wedge\mathbf{x}(0)\big).$$
    Inserting this information into the point vortex system gives that for any $i\in\{1,\ldots,N\},$
    \begin{align*}
        \Omega\mathcal{R}(\Omega t)\big(\mathbf{e}_3\wedge\mathbf{x}_i(0)\big)&=\sum_{j=1\atop j\neq i}^{N}\frac{\Gamma_j}{2\pi}\frac{\mathcal{R}(\Omega t)\mathbf{x}_i(0)\wedge\mathcal{R}(\Omega t)\mathbf{x}_j(0)}{\left|\mathcal{R}(\Omega t)\mathbf{x}_i(0)-\mathcal{R}(\Omega t)\mathbf{x}_j(0)\right|_{\mathbb{R}^3}^2}+\gamma\,\mathbf{e}_3\wedge\mathcal{R}(\Omega t)\mathbf{x}_i(0)\\
        &=\mathcal{R}(\Omega t)\left(\sum_{j=1\atop j\neq i}^{N}\frac{\Gamma_j}{2\pi}\frac{\mathbf{x}_i(0)\wedge\mathbf{x}_j(0)}{\left|\mathbf{x}_i(0)-\mathbf{x}_j(0)\right|_{\mathbb{R}^3}^2}+\gamma\,\mathbf{e}_3\wedge\mathbf{x}_i(0)\right).
    \end{align*}
    We have use the fact that $\mathcal{R}(\Omega t)\in SO_3(\mathbb{R}).$ In particular it is invertible, so composing on the left by its inverse gives the desired result. This concludes the proof of Lemma \ref{lem rot sph}
    \end{proof}

\begin{lem}[A variant of Gronwall's Lemma]\label{lem:gronwall}
Let $f : \R^n \to \R^n$ such that there exists $\kappa>0$ such that
\begin{equation*}
 \forall \, x,y \in \R^n, \quad \big|f(x)-f(y)\big| \leqslant \kappa|x-y|.
\end{equation*}
Let $g \in L^1(\R_+,\R_+)$ and $T \geqslant0$. We assume that $z : \R_+ \to \R^n$ satisfies
\begin{equation*}
 \forall t \in [0,T], \quad z'(t) = f\big(z(t)\big)
\end{equation*}
and that $y : \R_+ \to \R^n$ satisfies
\begin{equation*}
 \forall t \in [0,T], \quad \left|y'(t)-f\big(y(t)\big)\right| \leqslant g(t).
\end{equation*}
Then,
\begin{equation*}
 \forall t \in [0,T], \quad |y(t) - z(t)| \leqslant \left( \int_0^t g(s) \dd s + |y(0)-z(0)|\right)e^{\kappa t}.
\end{equation*}
\end{lem}

\begin{proof}
The proof was already provided in \cite[Lemma B.2]{D24a} and is recalled here for the sake of completeness. Fix $t \in [0,T]$. One readily has
\begin{align*}
 |y(t)-z(t)| & \leqslant \left|\int_0^t \big(y'(s) - z'(s) \big)\dd s \right| + |y(0)-z(0)|\\
 & \leqslant \int_0^t g(s)\dd s + \left| \int_0^t \big[f\big(y(s)\big) - f\big(z(s)\big)\big] \dd s \right| + |y(0)-z(0)|\\
 & \leqslant \int_0^t g(s)\dd s + |y(0)-z(0)| + \kappa \int_0^t |y(s)-z(s)|\dd s.
\end{align*}
Applying the classical Gronwall's inequality, since $t\mapsto \int_0^t g(s)\dd s + |y(0)-z(0)|$ is non-negative and differentiable, we obtain
\begin{equation*}
 |y(t)-z(t)| \leqslant \left( \int_0^t g(s) \dd s + |y(0)-z(0)|\right)e^{\kappa t}.
\end{equation*}
This concludes the proof of Lemma \ref{lem:gronwall}.
\end{proof}

\begin{lem}[Solving the high-order moments ODE]
\label{lem:4nODE}
Let $a > 0$, $b \geqslant 0$, $n\in\mathbb{N}^*$. Consider $y: [0, +\infty) \to \mathbb{R}_+$ a differentiable function satisfying the following differential inequality
\begin{equation*}
\forall t\geqslant0,\quad y'(t) \leqslant a y^{\frac{n-1}{n}}(t) \left( b + y^{\frac{1}{n}}(t) \right).
\end{equation*}
Then, the following upper bound holds
\begin{equation}
\label{eq:modified_upper_bound}
\forall t\geqslant0,\quad y(t) \leqslant \left( -b + \left( b + y^{\frac{1}{n}}(0) \right) e^{\frac{a}{n}t} \right)^n.
\end{equation}
\end{lem}

\begin{proof}
Consider the associated differential equation
\begin{equation}
\label{eq:associated_ode_modified}
z'(t) = a z^{\frac{n-1}{n}}(t) \left( b + z^{\frac{1}{n}}(t) \right), \qquad z(0) = y(0).
\end{equation}
By the comparison lemma for differential equations, we have 
$$\forall t \geqslant 0,\quad y(t) \leqslant z(t).$$
In order to solve~\eqref{eq:associated_ode_modified}, we set $w\triangleq z^{\frac{1}{n}}$. Then,
\begin{equation*}
\forall t\geqslant0,\quad w'(t) = \frac{z'(t)}{n z^{\frac{n-1}{n}}(t)} = \frac{a z^{\frac{n-1}{n}}(t) \left( b + z^{\frac{1}{n}}(t) \right)}{n z^{\frac{n-1}{n}}(t)} = \frac{a}{n} \big( b + w(t) \big).
\end{equation*}
This is a linear differential equation of order 1 with constant coefficients whose solution is given by
\begin{equation*}
\forall t\geqslant0,\quad w(t) = -b + \big( b + w(0) \big) e^{\frac{a}{n} t}=-b + \left( b + z^{\frac{1}{n}}(0) \right) e^{\frac{a}{n} t}=-b + \left( b + y^{\frac{1}{n}}(0) \right) e^{\frac{a}{n} t}.
\end{equation*}
Coming back to $z$, we infer
\begin{equation*}
\forall t\geqslant0,\quad z(t) =w^{n}(t)= \left( -b + \left( b + y^{\frac{1}{n}}(0) \right) e^{\frac{a}{n} t} \right)^n.
\end{equation*}
Thus, the upper bound~\eqref{eq:modified_upper_bound} follows directly.
\end{proof}

\bibliographystyle{plain}
\bibliography{biblio}

\bigskip\noindent
{\bf Martin Donati}\\
CNRS, Université de Poitiers, LMA, Poitiers, France \\
Email\: {\tt martin.donati@math.univ-poitiers.fr}

\bigskip\noindent
{\bf Emeric Roulley}\\
Università degli Studi di Milano (UniMi), Dipartimento di Matematica Federigo Enriques, Via Cesare Saldini, 50, 20133 Milano, Italy\\
Email\: {\tt emeric.roulley@unimi.it}

\end{document}